\documentclass{article}

\usepackage{appendix}
\usepackage[T1]{fontenc}
\usepackage{amsfonts,mathrsfs}
\usepackage[english]{babel}
\usepackage{graphicx,multirow}
\usepackage{indentfirst}
\usepackage[colorlinks=true]{hyperref}
\usepackage{amsmath, amsthm, amssymb}
\usepackage{hyperref}
\usepackage{subcaption}
\usepackage{pdfpages}
\usepackage{lipsum}
\usepackage[left=2cm, right=2cm, bottom=3cm]{geometry}
\usepackage{hyperref}
\usepackage{xcolor}
\usepackage{tcolorbox}
\usepackage{colortbl}
\usepackage{bbm}
\usepackage{cite}
\usepackage[resetlabels,labeled]{multibib}
\usepackage{epsf,epsfig}
\usepackage{multirow}
\usepackage{caption}
\usepackage{bbm, mathrsfs}

\usepackage{titlesec}
\usepackage[utf8]{inputenc} 
\usepackage{float}
\usepackage{geometry} 

\usepackage{graphicx}
\usepackage{listings}
\usepackage{textcomp}
\usepackage[colorlinks=true]{hyperref}
\hypersetup{linkcolor=blue}
\usepackage{xcolor}
\definecolor{mygreen}{RGB}{28,172,0} 
\definecolor{mylilas}{RGB}{170,55,241}
\usepackage{float,color}
\usepackage{array}

\usepackage[normalem]{ulem}

\newfloat{figure}{H}{lof}
\newfloat{table}{H}{lot}
\floatname{figure}{\figurename}
\floatname{table}{\tablename}

\definecolor{mycolor}{RGB}{0, 204, 204}

\usepackage{tcolorbox}

\newtheorem{theorem}{Theorem}[section]

\newtheorem{assumption}[theorem]{Assumption}

\newtheorem{lemma}[theorem]{Lemma}

\newtheorem{proposition}[theorem]{Proposition}
\newtheorem{remark}[theorem]{Remark}

\numberwithin{equation}{section}

\renewcommand{\eqref}[1]{(\ref{#1})}

\definecolor{mycolor}{RGB}{0, 204, 204}

\definecolor{mycolor1}{RGB}{   232, 247, 251 }
\definecolor{mycolor2}{RGB}{   1, 104, 137 }

\newtcolorbox{mybox}[1]{colback=red!5!white,colframe=red!75!black,fonttitle=\bfseries,title=#1}

\newtcolorbox{resultbox}[1]{colback=red!5!white,colframe=red!75!black,fonttitle=\bfseries,title=#1}
\newtcolorbox{applicationbox}[1]{colback=mycolor1,colframe=mycolor2,fonttitle=\bfseries,title=#1}

\newtcolorbox{citationbox}[1]{colback=white,colframe=white,fonttitle=\bfseries,title=#1}

\newcommand {\CR}{{\rm CR}}

\newcommand{\N}{\mathbb{N}}

\begin{document}
\title{\textbf{A Kermack-McKendrick model with age of infection starting from a single or multiple cohorts of infected patients}}
\author{\textsc{Jacques Demongeot$^{(a)}$, Quentin Griette$^{(b),}$\footnote{Q.G. acknowledges support from ANR via the project Indyana under grant agreement ANR-21-CE40-0008.} , Yvon Maday$^{(c)}$, and Pierre Magal$^{(b),}$\thanks{Corresponding author. e-mail: \href{mailto:pierre.magal@u-bordeaux.fr}{pierre.magal@u-bordeaux.fr}}}\\
{\small \textit{$^{(a)}$Université Grenoble Alpes, AGEIS EA7407, F-38700 La Tronche, France}} \\
{\small \textit{$^{(b)}$Univ. Bordeaux, IMB, UMR 5251, F-33400 Talence, France.}} \\
{\small \textit{CNRS, IMB, UMR 5251, F-33400 Talence, France.}}\\
{\small \textit{$^{(c)}$Sorbonne Université and Université de Paris, CNRS, Laboratoire Jacques-Louis Lions (LJLL), F-75005 Paris, France}}\\
{\small \textit{Institut Universitaire de France, 75005 Paris, France}}
}
\maketitle
\begin{abstract}
	During an epidemic, the infectiousness of infected individuals is known to depend on the time since the individual was infected, that is called the age of infection. Here we study the parameter identifiability of the Kermack-McKendrick model with age of infection which takes into account this dependency. By considering a single cohort of individuals, we show that the daily reproduction number can be obtained by solving a Volterra integral equation that depends on the flow of new infected individuals. We test the consistency our the method by generating data from deterministic and stochastic numerical simulations. Finally we apply our method to a dataset from SARS-CoV-1 with detailed information on a single cluster of patients. We stress the necessity of taking into account the initial data in the analysis to ensure the identifiability of the problem.
\end{abstract}

\medskip 



\section{Introduction}
The existence of an individual infection (or contagiousness) period of variable length among infected individuals and during the history of an epidemic is a proven fact in all contagious diseases. The origin of this variability is multiple. It may be due to 
\begin{itemize}
	\item[ \rm{(i)} ] 	a variation in the symptomatic state of the infected person (such as the change in frequency and intensity of the cough, source of emission of the infectious agent) due to variable immune defenses since the beginning of his infection (e.g., due to the transition between initial innate immune reaction and secondary adaptive reaction);
		\item[ \rm{(ii)}] a variation in the environmental conditions of transport and survival of the infectious agent in the atmosphere (heat, humidity, sunshine, altitude, etc.), in a more or less favorable socio-sanitary environment presenting different spreading characteristics (as compliance with distancing or confinement, existence of preventive vitamin supplementation, educational level, food quality, demographic isolate, etc.);
	\item[ \rm{(iii)}] 	a variation in the state of defense of the final host (immune defense and self-protection);
	\item[ \rm{(iv)}] 	a variation in the virulence of the infectious agent, which can mutate or possibly change of intermediary host.		
\end{itemize}

Due to the complexity of the contagion mechanisms and of their variations, accentuated by the fact that the rate of the epidemic may be the result of several geographically distant clusters starting at different times, it is preferable, in a first approach, to assume that the initial conditions are reduced to a single cluster. The challenge is then to estimate, on each day of the infection period, the transmissibility rate allowing the calculation of the daily reproduction number of the contagious disease.

\medskip

\medskip 
\noindent \textbf{Continuous time model:} Recall that the age of infection $a$ is the time since individuals become infected.  The major difficulty in comparing the data and the Kermack-McKendrick model with age of infection is to identify: 1) the initial distribution of infected with respect to the age of infection; 2) the daily reproduction number $R_0(a)$ which is the reproduction number at the age of infection $a$ (i.e. the average number of secondary cases produced by a single infected at the age of infection $a$).  We can decompose the daily reproduction number as follows  
\begin{equation*} \label{1.7}
	R_0(a)= \underset{ {\rm (A)}}{\underbrace{   \tau_0 }} \times \underset{ {\rm (B)}}{\underbrace{ S_0 }}   \times  \underset{ {\rm(C)}}{\underbrace{  \beta(a) }} \times  \underset{ {\rm (D)}}{\underbrace{  e^{-\nu a}  }},
\end{equation*}
where (A) is the rate of transmission rate at time $t_0$ (we assume the transmission rate to be constant during the period where $R_0(a)$ is evaluated), (B) is the number of susceptible individuals at time $t_0$ (we assume the number of susceptibles to be constant during the period where $R_0(a)$ is evaluated), (C) is the probability that an infected who have been for $a$ days is infectious i.e., capable to transmit the pathogen (or the fraction of infectious among the infected with age of infection $a$), (D)  the probability for an infected for $a$ days to be still infected.

Then the basic reproduction number (i.e. the number of secondary cases produced by a single infected) is given by 
\begin{equation*} \label{1.2}
	R_0=\int_0^\infty R_0(a)da. 	
\end{equation*}

Here we partly solve the problem of finding the initial distribution of infected by assuming that we start the epidemic at time $t_0$ with a single cohort of $I_0$  new infected patients. That is, the epidemic starts with $I_0$ infected patients all with age of infection $a=0$. The case of an epidemic starting from a single infected patient (usually called the patient $0$) corresponds to the case $I_0=1$. This is a common assumption in epidemiology. Note that the time $t_0$ at which the first patient becomes infected is also unknown for most epidemics.  

Assume that the epidemic starts at time $t_0$ with a cohort of $I_0$ new infected patients (i.e., all with age of infection $a=0$). Then $N(t)$ the flow of infected at time $t$  satisfies the model starting from a single cohort of infected
\begin{equation} \label{1.1}
	N(t)=\underset{ {\rm (I)}}{\underbrace{  R_0(t-t_0) \times  I_0 }}+ \underset{ {\rm (II)}}{\underbrace{\int_0^{t -t_0} R_0(s) \times  N(t-s)ds}}, \forall t \geq t_0, 
\end{equation}
where  \rm{(I)} is the flow of  infected produced directly by the $I_0$ infected individuals already present on day $t_0$; and \rm{(II)} is the flow of new infected individuals at time $t$  produced by the new infected individuals since day $t_0$.  

Here what we call ``flow of new infected individuals'' is the distribution $N(t)$ such that 
$$
\int_{t_1}^{t_2} N(\sigma) d \sigma, 
$$
is the number of new infected  individuals  during the period of time $[t_1, t_2]$.

\medskip
The equation \eqref{1.1} will be derived as an extension of the Kermac-McKendrick model starting from a single cohort of infected (see equation \eqref{4.6}). This  model remains valid as long as the transmission rate  $\tau(t)$ and the number of susceptible hosts $S(t)$ remain constant. So this model is valid when the epidemic starts.  

\medskip
Assume that $I_0$ is fixed and the function $a \to R_0(a)$ is  given. Then  the map $t \to N(t)$ can be obtained by solving \eqref{1.1}. The goal of the article is consider the converse problem. That is, assume that $I_0$ is fixed and assume that $t \to N(t)$ is given from the data. Then the map $a \to R_0(a)$ can be obtained by solving the Volterra integral equation  
\begin{equation} \label{1.2}
	R_0(a)=  \dfrac{N(a+t_0)}{I_0} -  \dfrac{1}{I_0}  \int_0^{a} R_0(s) N(a-s+t_0)ds , \forall a \geq 0. 
\end{equation}
Therefore if the map $t \to N(t)$ is known, we can theoretically derive the average dynamics of infection at the level of a single patient.

\medskip 
The standard assumption used in the literature is
	\begin{equation} \label{1.3}
		R_0(a)=0, \forall a \geq a_+, 	
	\end{equation}
	where $a_+>0$ the maximal age of infectiousness for an infected patient.  
	
	\medskip 
	Then for $t \geq t_0+a_+,$ the equation \eqref{1.1} becomes 
	\begin{equation} \label{1.4}
		N(t)=\int_0^{a_+} R_0(s) \times  N(t-s)ds. 
	\end{equation}
	Now assume for example that $N(t)=N_0 e^{\lambda t },$ we obtain after simplifications a standard characteristic equation
	\begin{equation} \label{1.5}
		1=\int_0^{a_+} R_0(s) \times  e^{-\lambda s }ds. 
	\end{equation}
	Assume that $\lambda>0$ is given. Consider $a \to \chi(a)$ any non-negative and non-null  continuous function satisfying 
	$$
	\chi(a)=0, \forall a \geq a_+.
	$$
	Then 
	$$
	R_0(a)=\dfrac{\chi(a)}{\int_0^{a_+} \chi(s) \times  e^{-\lambda s }ds }
	$$
	satisfies \eqref{1.5}. Therefore, neglecting the initial value in the Volterra equation leads to a non-identifiable problem (in general). This shows the crucial role of the initial value in identifying the function $a \to  R_0(a)$.  

\medskip 
\noindent \textbf{Day by day model:} The model \eqref{1.1} with a single cohort of infected becomes  a discrete Volterra equation
\begin{equation} \label{1.6}
	N(t)= \underset{ {\rm (I)}}{\underbrace{  R_0(t-t_0) \times I_0 }}+\underset{ {\rm (II)}}{\underbrace{    \sum_{d=1}^{t-t_0}R_0(d) \times  N(t-d) }} ,  \forall t \geq t_0, 
\end{equation}  
where  \rm{(I)} is the number of  infected produced directly by the $I_0$ infected individuals already present on day $t_0$; and \rm{(II)} is the number of new infected individuals at time $t$  produced by the new infected individuals since day $t_0$.  

\medskip 
Next, by setting $a=t-t_0$, we obtain the day by day equation for the daily reproduction number  
\begin{equation} \label{1.7}
	R_0(a) = \dfrac{N(t_0+a)}{ I_0}-  \dfrac{1 }{I_0}\sum_{d=1}^{a} R_0(d) \times  N(t_0+a-d), \forall a \geq 0.
\end{equation}  
In the above formula and throughout the paper, we use the following convention for the sum 
$$ 
\displaystyle \sum_{d=k}^{m}  =0, \text{ whenever } m<k. 
$$ 
\medskip 
If we consider the first terms of the discrete time Volterra equation \eqref{1.6}, we obtain   
\begin{equation*} 
	\begin{array}{ll}
			N(t_0)&=	 R_0(0) \times I_0, \vspace{0.1cm} \\ 
				N(t_0+1)&=	 R_0(1) \times I_0+R_0(1)\times  N(t_0), \vspace{0.1cm} \\ 
				N(t_0+2)&=	 R_0(2) \times I_0+R_0(2) \times N(t_0)+R_0(1) \times  N(t_0+1), \vspace{0.1cm} \\ 
				N(t_0+3)&=	 R_0(3) \times I_0+R_0(3) \times  N(t_0)+R_0(2) \times N(t_0+1)+ R_0(1) \times N_0(t_0+2),\\
				\quad\vdots
	\end{array}
\end{equation*}
\medskip 
In practice, we can assume that $R_0(0)=0$ since infected individuals are not infectious immediately after being infected. Under this additional assumption, we obtain the system 
 \begin{equation*} 
 	\begin{array}{ll}
 		N(t_0)&=	0, \vspace{0.1cm} \\
 		N(t_0+1)&=	 R_0(1) \times I_0, \vspace{0.1cm}\\
 		N(t_0+2)&=	 R_0(2) \times I_0+R_0(1) \times  N(t_0+1), \vspace{0.1cm} \\
 		N(t_0+3)&=	 R_0(3) \times I_0+R_0(2) \times N(t_0+1)+ R_0(1) \times N_0(t_0+2),\\
 		\quad \vdots
 	\end{array}
 \end{equation*}

\medskip 
When reliable information is available on the first cluster(s), the best formula for calculating daily basic reproduction numbers is equation \eqref{1.2} (or its the discrete time version \eqref{1.7}). Based on \eqref{1.4}, some methods have been developed in the literature to cope with the lack of a precise information.
 
For instance  in  \cite{AMM21, AMM22}, the authors following \cite{Nishiura2} propose an optimization algorithm for estimating the daily basic reproduction numbers.  In \cite{Bernoulli}, D. Bernoulli mentions in 1760 the changes in the contagiousness parameters and places as a crucial challenge for the prediction of the transition between endemic and epidemic peaks in a prophetic sentence: ``Le retour d’une épidémie longtemps suspendue fait un ravage plus terrible dans une seule année qu'une endémie uniforme ne pourrait faire pendant un nombre d'années considérable'' (The return of a long-suspended epidemic wreaks more terrible havoc in a single year than a uniform endemic could do for a considerable number of years). In \cite{DORSTW}, the authors use a deconvolution algorithm for calculating the daily basic reproduction numbers. In each case, the problem of the initial conditions is evoked at best only through the hypothesis of a unique ``patient zero''. Despite the considerable means of current investigation, in particular those of the WHO and the members of the government of the WHO, it is rare that this patient is identified (this was the case for H1N1 in Mexico). The patient zero, also called index or primary case is the first patient identified in a given population during an epidemiological investigation. It points out the source of the spread of a disease in a given reservoir, 
but this search is in general very difficult as it was the case for HIV in North America \cite{Worobey}.

	\medskip 
One of the main difficulties in estimating the $R_0(a)$ function is its non-identifiability in general. Recent studies \cite{LMSW1, DGM1, DGM2, DGM3, DGMW} developed methods to identify the various parameters for the COVID-19 pandemic by using cumulative reported cases data and differential equations models. Differential equations can be written in the form studied here by assuming that $R_0(a)$ (or equivalently, $\beta(a)$) is independent of $a$. Suppose that we are restricted to a period when the data is growing exponentially fast. If we take a fixed function $\beta(a)$, then by adapting the method developed in \cite{LMSW1}, we could identify a transmission rate $\tau$  so that the output of the model stays very close to the data, \textit{for any function $\beta(a)$}. The same could be achieved with a good phenomenological description of the data by using the method developed in \cite{DGM1, DGM2, DGM3} with a time-dependent transmission rate.   This means that the reported cases data is not sufficient to determine accurately the function $R_0(a)$. Without a good description of the initial distribution, it is hopeless to identify $R_0(a)$ by using reported cases only.

\medskip 
 In this article, we first extend the Kermack-McKendrick model to initial conditions that are a linear combination of Dirac masses. The PDE model \eqref{2.1} can not be extended in the space of measures (due to a lack of  time continuity for the solutions). However, the Volterra integral equation can be extended and still makes perfect sense whenever we use Dirac masses for the initial condition. In many real examples, the initial distribution must be a linear combination of Dirac masses since the data are discrete at the early stage of an epidemic (in a city, a country). Indeed, at the early beginning of an epidemic, the epidemic starts from a few cases imported from other places. Therefore Dirac masses make perfect sense. In practice, the early stage of an outbreak is often undocumented and generally difficult to determine. But our study applies to data from a finite collection of clusters that is easier to determine using contact tracing.

\medskip 	
Consequently, for the single cohort model, we can reverse the problem, and by assuming that the daily number of new infected is known, we can compute the daily reproduction number by solving a Volterra integral equation. The daily basic reproduction number informs us about the dynamics of infection at the level of a single patient. Therefore, knowing $R_0(a)$ should help the medical doctors decide about quarantine measures. Reported case data for clusters are particularly valuable for reconstructing the dynamics of infection at the level of a single individual.
 
\medskip 	
In this paper, we also provide an Individual-Based Model  (IBM)  (see Appendix \ref{AppendixB}). This IBM converges to the deterministic  model whenever the initial number of infected increases. We use this IBM to generate sample data to test our method and compute the daily basic reproduction number. This will allow us to test the effects of the day-to-day discretization (on the data) and the impact of stochastic  perturbations on the daily reproduction number. We conclude the paper by applying our approach to a cluster of SARS-CoV-1 in Singapore.


\medskip 
The plan of the paper is the following. In Section \ref{Section2}, we recall the Kermack-McKendrick model with age of infection. We explain how to derive the Volterra formulation of the model, and we compare it with the  Kermack-McKendrick SI model with age of infection (ODE model). In Section \ref{Section3}, we explain how to connect the model with the data. In Section \ref{Section4}, we extend the Kermack-McKendrick model with age of infection in the case where the epidemic starts from a single or multiple cohorts of infected individuals. In Section \ref{Section5} we derive an equation to  compute the daily reproduction number from the data. In Section \ref{Section6} we consider a day by day discretized Kermack-McKendrick model with age of infection. In Section \ref{Section7}, we run some numerical simulations, we compare the deterministic model  with a stochastic individual based simulation presented in appendix. In Section \ref{Section8}, we compare the model with some data from SARS-CoV-1, and we discuss the data from SARS-CoV-2.

\section{Kermack-McKendrick model with age of infection}
\label{Section2}
\subsection{Partial differential equation formulation of the model}
The age of infection  $a$ is the time since individuals become infected. Let $a \to i(t,a)$ be the distribution of population of \textit{infected individuals} at time $t$ (with respect to $a$ the age of infection). The term distribution of population means that the integral 
$$
\int_{a_1}^{a_2} i(t,a)da
$$ 
is the number of infected at time $t$ with infection age between $a_1$ and $a_2$. Therefore the total number of infected individuals at time $t$ is
$$
I(t)=\int_{0}^{+\infty} i(t,a)da.
$$
Let  $ \beta(a)  \in [0,1]$  be the probability to be contagious or infectious (i.e. capable to transmit the pathogen) at the age of infection $a$.  The quantity $ \beta(a)$ can be interpreted as the fraction of infected individuals with age of infection $a$ that are infectious. 
Then the total number of  \textit{contagious individuals} (or also called \textit{infectious individuals}) (i.e., the individuals capable of transmitting the pathogen) at time $t$ is
$$
C(t)= \int_0^{+ \infty} \beta(a) i(t,a)da. 
$$
The model of Kermack-McKendrick  \cite{Kermack-McKendrick2} with age of infection  is the following, for each  $t \geq t_0$ 
\begin{equation}\label{2.1}
	\left\lbrace
	\begin{array}{l}
		S'(t)=-  \tau(t) \,  S(t) \, \displaystyle \int_0^{+ \infty} \beta(a) \, i(t,a)  da,  \vspace{0.2cm}\\
		\partial_t i +\partial_a i= - \nu \, i(t,a), \text{ for } a \geq 0, 	 \vspace{0.2cm}\\
		i(t,0)= \tau(t) \, S(t) \,\displaystyle  \int_0^{+ \infty} \beta(a) \, i(t,a)  da, 
	\end{array}
	\right.
\end{equation}
this system is supplemented by initial data 
\begin{equation}\label{2.2}
	S(t_0)=S_0 \geq 0, \text{ and } i(t_0,a)=i_0(a) \in L^1_+(0, \infty).
\end{equation}
In the model, $S(t)$ is the number of susceptible individuals at time $t$, and  $t \to \tau(t)$ is the transmission rate at time $t$, and $ \nu \geq 0 $ is the rate at which  individuals die or recover. Here, the parameter $\nu$ is assumed to be independent of the age of infection $a$. 
This is a simplifying assumption to improve the readability of the paper.  The parameter $\nu$ combines both the specific fatality rate and the recovery rate.  

\medskip 
The above equation can be understood first as follows 
$$
I'(t) =  \underset{ {\rm (I)}}{\underbrace{   \tau(t) \, S(t) \, \int_0^{+\infty} \beta(a) \, i(t,a) da}}- \underset{{\rm (II)}}{\underbrace{   \int_0^{+\infty} \nu \, i(t,a) da}}, 
$$
where (I) is the flow of new infected, and (II) is the flow of individuals who die or recover.

\medskip 
We make the following assumption. 
\begin{assumption} \label{ASS2.1}	We assume that  
	\begin{itemize}
		\item[{\rm (i)}] The transmission rate $t \to \tau(t)$  is a bounded continuous map from $[t_0, +\infty)$ in $[0, +\infty)$; 
	  \item[{\rm (ii)}] The probability to be infectious at the age of infection $a \to \beta(a) \in L^\infty_+(0, +\infty)$  is a non-negative and measurable function of $a$  which is bounded by $1$; 
	\end{itemize}
\end{assumption}

%

\subsection{Volterra integral equation formulation of the model}
In the model \eqref{2.1}, the quantity 
\begin{equation}\label{2.3}
	N(t):=\tau(t) \,  S(t) \,  \int_0^{+\infty} \beta(a) \, i(t,a) da,  
\end{equation}
is the flow of new infected individuals at time $t$. 

\medskip 
By using the $S$-equation in system \eqref{2.1}, we obtain 
\begin{equation} \label{2.4}
 	S(t)=S_0- \int_{t_0}^t N( \sigma) d \sigma, \forall t \geq t_0.
\end{equation} 
By integrating the $i$-equation of system \eqref{2.1} along the characteristics, we obtain 
\begin{equation} \label{2.5}
	i(t,a) =\left\lbrace 
	\begin{array}{ll}
		e^{- \nu \, \left( t-t_0\right)   } \, i_0 \left(a- \left(t-t_0\right) \right),& \text{ if } a \geq t-t_0, \vspace{0.2cm} \\
		e^{-  \nu \, a }  N(t-a),  &\text{ if } t-t_0\geq a. \\
	\end{array}
	\right.
\end{equation}
By using \eqref{2.5}, we deduce that $ t \to N(t)$ satisfies the following Volterra integral equation  
\begin{equation} \label{2.6}
	N(t)=\underset{{\rm (I)}}{\underbrace{  \tau(t) \, S(t) \,  \int_{t-t_0}^{+ \infty}  \beta(a) \,  e^{-  \nu \, \left( t-t_0\right)  } i_0 \left(a- \left(t-t_0\right) \right) da}}+ \underset{{\rm (II)}}{\underbrace{  \tau(t) \, S(t) \, \int_0^{t-t_0} 	\beta(a) 	e^{-  \nu \, a }   N(t-a)\, da  }}
\end{equation}
where ${\rm (I)}$ is the flow of new infected individuals at time $t$ produced by the infected individuals already present on day $t_0$; ${\rm (II)}$ is the flow of new infected individuals at time $t$   produced by the new infected individuals since day $t_0$. 


\medskip 
By using equations \eqref{2.4} and \eqref{2.6},  we can summarize the epidemic model \eqref{2.1},  by saying that $t \to N(t)$ is the unique continuous map satisfying  
\begin{equation} \label{2.7}
N(t)=\tau(t)\, S(t) \left[ \Lambda(t)+ \int_0^{t-t_0} 	\beta(a) 	e^{-  \nu \, a }   N(t-a) \, da\right], \forall t \geq t_0, 
\end{equation}
where 
\begin{equation} \label{2.8}
	S(t)=S_0-\int_{t_0}^t N(\sigma ) d \sigma, \forall t \geq t_0, 
\end{equation}
and
\begin{equation} \label{2.9}
	\Lambda(t):=  e^{-  \nu \, \left( t-t_0\right)  } \int_{t-t_0}^{+ \infty} \beta(a) \,	 i_0 \left(a- \left(t-t_0\right) \right) da , \forall t \geq t_0.
\end{equation} 
The function $\Lambda(t)$ is the number of infectious individuals (capable to transmit the pathogen) at time $t$ among the infected individuals already present at time $t_0$.

\medskip 
The function  $t \to \Lambda(t)$ plays a fundamental role in solving the Volterra equation. Indeed, the quantity 
$$
\int_{t_1}^{t_2}\tau(\sigma)\, S(\sigma) \Lambda(\sigma) d \sigma ,
 $$
is the number of  infected produced between the instants $t_1$ and $t_2$ by the infected already present at time $t_0$.  So, for example, if no new infected are produced by the infected already present at time $t_0$, that is if 
$$
\Lambda(t)=0, \forall t \geq t_0,
$$
then there will be no new infected at all after the time $t_0$, that is 
$$
N(t)=0, \forall t \geq t_0. 
$$ 
The function $t \to \Lambda(t)$ can be regarded as the initial distribution of the Volterra integral equation \eqref{2.7}.  

\section{Connecting the data and the model}
\label{Section3}
The data are represented by the function $t \to \CR(t)$ which is the cumulative number of reported cases at time $t$. We assume that the flow of reported cases is a fraction  $0 \leq f \leq 1$  of the flow of recovering individuals, that is 
\begin{equation} \label{3.1}
	\CR'(t)=f \,  \nu \, \int_0^{+ \infty} i(t,a)da.
\end{equation}
By using \eqref{2.5}, we can compute the number of infected at time $t$. That is  
\begin{equation} \label{3.2}
	\int_0^{+ \infty} i(t,a)da=  e^{-  \nu \, \left(t-t_0\right)  }  I_0+  \int_0^{t-t_0}	 	e^{-  \nu \, a }   N(t-a) \, da, 
\end{equation}
where 
$$
I_0=\int_0^{+\infty} i_0(a)da
$$ 
is the total number of infected at time $t_0$.

By using equations \eqref{3.1} and \eqref{3.2}, we obtain 
\begin{equation*}
	\CR'(t)=\nu f \left[   e^{-  \nu \, \left(t-t_0\right)  }  I_0+  \int_0^{t-t_0}	 	e^{-  \nu \, a }   N(t-a) da   \right], 
\end{equation*}
or equivalently (by using the change of variable $\sigma=t-a$)
\begin{equation*}
	\CR'(t)=\nu f \left[   e^{-  \nu \, \left(t-t_0\right)  }  I_0+  \int_{t_0}^{t}	 	e^{-  \nu \, \left(t-\sigma\right) }   N(\sigma) d\sigma   \right].
\end{equation*}
By choosing $t=t_0$ we obtain 
$$
I_0= \dfrac{ \CR'(t_0)}{\nu f},
$$
and
$$
\int_{t_0}^t e^{\nu \sigma} N( \sigma) d \sigma=     \dfrac{e^{  \nu \, t  }  \CR'(t)}{\nu f} - e^{  \nu \, t_0  } I_0, 
$$
and by differentiating  both sides of the above equation, we obtain 
$$
e^{\nu t} N(t)=  \dfrac{  \nu e^{  \nu \, t  } \CR'(t)+ e^{  \nu \, t  } \CR''(t) }{\nu f}  .
$$ 
Therefore we obtain the following connection between the data and the model. 
\begin{mybox}{Connection between the data and the model}
	Let $t \to \CR(t)$ be the cumulative number of reported cases. Then the initial number of infected is given by 
	\begin{equation} \label{3.3}
		I_0=\dfrac{ \CR'(t_0)}{\nu f},
	\end{equation}
	and the flow of new infected individuals $N(t)$ at time $t$ is given by 
	\begin{equation} \label{3.4}
		N(t)= \dfrac{\nu  \CR'(t) +  \CR''(t) }{\nu f},  \forall t \geq t_0. 
	\end{equation}
\end{mybox}

\section{Kermack-McKendrick model  starting from a single and multiple cohorts of infected patients}
\label{Section4}
The major difficulty to compare the model \eqref{2.4} with the data is to identify the functions $a \to i_0(a)$  and $a \to \beta(a)$. To simplify the discussion, let us consider the model at the early stage of the epidemic.  When the epidemic just starts we can assume that the transmission rate $t \to \tau(t) $ remains constant, and the number of susceptible individuals $t \to S(t)$ is constant and equal to $S_0$. Under such a simplifying assumption the Volterra equation \eqref{2.4} becomes 
\begin{equation} \label{4.1}
		N(t)=\tau\, S_0 \left[ \Lambda(t)+ \int_0^{t-t_0} 	\beta(a) 	e^{-  \nu \, a }   N(t-a) \, da\right], \forall t \geq t_0. 
\end{equation}

\subsection{Initial distribution for a single cohort of infected with age of infection $a=0$}
In order to understand the mathematical concept of Dirac mass centered at $0$, we first consider an approximation by an exponential law  
$$
i_0(a)= I_0  \,  \kappa \, e^{-\kappa a} ,
$$
mean and standard deviation  equal to $1/\kappa$. Then a Dirac mass centered at $0$ can be understood as the limit of such a distribution when $\kappa$ goes to  $+\infty$. The limit of needs some explanations. Recall that 
$$
\int_{a_1}^{a_2} i_0(a)da= I_0  \,  \left[e^{-\kappa a_1}-e^{-\kappa a_2} \right] , 
$$
is the initial number of infected individuals with infection age $a$ in between $a_1$ and $a_2$ at time $t=0$. 

We deduce that 
$$
\lim_{\kappa \to \infty} \int_{a_1}^{a_2} i_0(a)da =
\left\lbrace 
\begin{array}{lr}
	0, &\text{ if } a_2>a_1>0, \\
	I_0,& \text{ if } a_2>a_1=0.\\
\end{array}
\right.
$$
That is to say that,  when $\kappa$ tends to $ +\infty$, the initial distribution of population $i_0(a)$ is approaching the case where all the infected individuals at time $t_0$ have the same age of infection $a=0$.

For short, we write 
$$
i_0(a) = I_0  \, \delta_0(a),
$$  
where $ \delta_0(a)$ is called the  Dirac mass centered at $0$.


\subsection{Model starting  from a single cohort of infected with age of infection $a=0$}

Recall that 
\begin{equation*}  
	\Lambda(t)= e^{-  \nu \, \left( t-t_0\right)  } \int_{0}^{+ \infty} \beta\left(a+ \left(t-t_0\right) \right) \,	 i_0 \left(a  \right) da,
\end{equation*} 
so when $i_0(a)$ is replaced by $I_0 \times \kappa \times e^{-\kappa a}$ we obtain 
\begin{equation*} \label{4.2}  
	f_\kappa(t):= I_0  e^{-  \nu \, \left( t-t_0\right)  } \int_{0}^{+ \infty} \beta\left(a+ \left(t-t_0\right) \right) \times	 \kappa \times e^{-\kappa a}da.
\end{equation*} 
In order to derive the Kermack-McKendrick model  with Dirac mass initial distribution as limit, we first need the following result. 
\begin{lemma} \label{LE4.1} Let Assumption \ref{ASS2.1} be satisfied, and assume in addition that $a \to \beta(a)$ is continuous. Then we have 
\begin{equation*} \label{4.3}
	\lim_{\kappa \to \infty} f_\kappa(t)=  I_0  e^{-  \nu \, \left( t-t_0\right)  } \beta\left( t-t_0 \right), 
\end{equation*}
where the limit is uniform in $t \geq t_0$. That is 
\begin{equation*}  \label{4.4}
	\lim_{\kappa \to +\infty} \sup_{t \geq t_0} \vert f_\kappa(t)- I_0  e^{-  \nu \, \left( t-t_0\right)  } \beta\left( t-t_0 \right)\vert=0. 
\end{equation*}
\end{lemma} 
\begin{proof}
	Let $\varepsilon>0$. We observe that 
	\begin{equation*}   
		\begin{array}{ll}
f_\kappa(t)- I_0  e^{-  \nu \, \left( t-t_0\right)  } \beta\left( t-t_0 \right)&=  I_0 \,  e^{-  \nu \, \left( t-t_0\right)  } \displaystyle \int_{0}^{+ \infty} \left[\beta\left(a+ \left(t-t_0\right) \right) -\beta\left(t-t_0\right)\right]    	 \kappa \,  e^{-\kappa a}da, \vspace{0.2cm}\\
&=  I_0\,   e^{-  \nu \, \left( t-t_0\right)  } \displaystyle  \int_{0}^ \eta \left[\beta\left(a+ \left(t-t_0\right) \right) -\beta\left(t-t_0\right)\right]   	 \kappa \,  e^{-\kappa a}da, \vspace{0.2cm}\\
&\hspace{0.2cm}+ I_0\,  e^{-  \nu \, \left( t-t_0\right)  } \displaystyle  \int_{\eta}^{+ \infty}  \left[\beta\left(a+ \left(t-t_0\right) \right) -\beta\left(t-t_0\right)\right]     \kappa \,   e^{-\kappa a}da. 
		\end{array}
	\end{equation*}
Let $t_1>t_0$ be such that 
$$
I_0 \,  e^{-  \nu \, \left( t-t_0\right)  } \displaystyle  \sup_{a \geq 0} \beta(a) \leq \dfrac{\varepsilon}{4}, \forall t \geq t_1. 
$$
Let $\eta >0$ be such that 
$$
a \leq \eta \Rightarrow \vert \beta\left(a+ t\right)-\beta\left(t\right) \vert \leq\dfrac{\varepsilon}{2}, \forall t \in [t_0,t_1]. 
$$
Then we have 
\begin{equation*} 
\vert 	I_0 \,  e^{-  \nu \, \left( t-t_0\right)  } \int_{0}^ \eta \left[\beta\left(a+ \left(t-t_0\right) \right) -\beta\left(t-t_0\right)\right]   	 \kappa \,  e^{-\kappa a}da  \vert \leq \left\{ 
	\begin{array}{ll}
	I_0 \,  e^{-  \nu \, \left( t-t_0\right)  } \, 2 \, \displaystyle \sup_{a \geq 0} \beta(a),	& \text{ if } t \geq t_1, \vspace{0.2cm}\\
			I_0 \,  e^{-  \nu \, \left( t-t_0\right)  } \int_{0}^ \eta   \dfrac{\varepsilon}{2}	 \kappa \,  e^{-\kappa a}da 	& \text{ if } t \in [t_0, t_1],
	\end{array}
\right. 
\end{equation*}
therefore 
$$
\vert f_\kappa(t)- I_0 \,  e^{-  \nu \, \left( t-t_0\right)  } \beta\left( t-t_0 \right)\vert \leq \dfrac{\varepsilon}{2}+ \vert I_0 \,   e^{-  \nu \, \left( t-t_0\right)  } \int_{\eta}^{+ \infty}  \left[\beta\left(a+ \left(t-t_0\right) \right) -\beta\left(t-t_0\right)\right]     \kappa \,   e^{-\kappa a}da \vert .
$$
The result from the fact that 
$$
 \vert I_0 \,  e^{-  \nu \, \left( t-t_0\right)  } \int_{\eta}^{+ \infty}  \left[\beta\left(a+ \left(t-t_0\right) \right) -\beta\left(t-t_0\right)\right]     \kappa \,   e^{-\kappa a}da \vert \leq 2\, I_0  \,   \sup_{a \geq 0} \beta(a)  \int_{\eta}^{+ \infty}    \kappa \,   e^{-\kappa a}da
$$
hence 
	\begin{equation*}
 \vert I_0  e^{-  \nu \, \left( t-t_0\right)  } \int_{\eta}^{+ \infty}  \left[\beta\left(a+ \left(t-t_0\right) \right) -\beta\left(t-t_0\right)\right]     \kappa \,   e^{-\kappa a}da \vert \leq  2\, I_0    \sup_{a \geq 0} \beta(a)  e^{-\kappa \eta} \to 0, \text{ as } \kappa \to \infty. \qedhere
	\end{equation*}
\end{proof}

\medskip 
\noindent Define 
\begin{equation} \label{4.2}
	\Gamma(a)=  e^{-  \nu \, a } \beta\left(a\right), \forall a \geq 0. 
\end{equation}
Then by using \eqref{2.6}, the Kermack-McKendrick model can be reformulated for  $t  \geq t_0$,  as the following system
\begin{equation*}  \label{4.6}
		N_\kappa(t) =\tau(t)  S_\kappa(t) \left[ f_\kappa(t) +  \int_0^{t-t_0} 	\Gamma(a) \, N_\kappa(t-a) da \right],
\end{equation*}
where $f_\kappa(t)$ is defined in \eqref{4.2}, and
\begin{equation*} \label{4.6}
S_\kappa(t)=S_0-\int_{t_0}^t N_\kappa( \sigma) d \sigma. 
\end{equation*}
By taking first a formal limit when $\kappa \to +\infty$, we obtain the model starting from a single cohort of infected. 
\begin{mybox}{Kermack-McKendrick model  starting from a single cohort of infected}
	Assume that the initial distribution of infected only contains a single cohort composed of $I_0$ individuals all with age of infection $a=0$ at time $t_0$. Then the flow of new infected $t \to N(t)$ is the unique continuous solution of the Volterra integral equation 
\begin{equation}  \label{4.3}
	N(t) =\tau(t)  S(t) \left[ I_0 \times	\Gamma\left(t-t_0\right) +  \int_0^{t-t_0} 	\Gamma(a) \, N(t-a) da \right], \forall t \geq t_0, 
\end{equation}
where $S(t)$ is obtained from \eqref{2.4}.
\end{mybox}
The following theorem says that the model with a single cohort of infected extends the earlier model of Kermack-McKendrick with initial distribution in $L^1$. This theorem is a consequence of Lemma \ref{LE4.1} and the continuity of the semiflow generated by the Volterra integral equation. We refer to Ducrot and Magal \cite{DM} for more results on this topic. 
\begin{theorem} Let Assumption \ref{ASS2.1} be satisfied, and assume in addition that $a \to \beta(a)$ is continuous.  Then
	\begin{equation*}
		\lim_{\kappa \to \infty} N_\kappa (t)=N(t), 
	\end{equation*} 
where the limit is uniform in $t$ on every closed and bounded interval of $[t_0,+ \infty)$, and the map $t \to N(t)$ is the unique continuous solution of the Volterra integral equation \eqref{4.3}-\eqref{4.4}. 
\end{theorem}
\begin{remark} When the initial distribution is a Dirac mass centered at $a=0$,  the total number of infected individuals at time $t$ is 
$$
C(t)=e^{-  \nu \, (t-t_0)}  I_0 +\int_{0}^{t-t_0} 	e^{-  \nu \, a }  N(t-a)da, \forall t \geq t_0,
$$
and the number of infectious individuals  at time $t$ is
$$
I(t)=\beta\left( t-t_0\right)  e^{-  \nu \, \left( t-t_0\right)}  I_0+ \int_{0}^{t-t_0}  \beta(a) e^{-  \nu \, a }  N(t-a) da,  \forall t \geq t_0.
$$	
\end{remark} 
\begin{mybox}{Kermack-McKendrick model  starting from multiple cohorts of infected}
Assume that the initial distribution of infected consists in $n\geq 1 $ cohorts of infected with age of infection $a_1 <a_2 < \ldots < a_n$ at time $t_0$. That is 
$$
i_0(a) =I_0^1  \, \delta_{a_1}(a)+ \ldots+ I_0^n  \, \delta_{a_n}(a). 
$$  
where $I_0^j$ is the number of infected in the $j^{th}$-cohort at time $t_0$.
    
\medskip  
Then the flow of infected  $t \to N(t)$ satisfies the following Volterra integral equation 
\begin{equation*}  
	\begin{array}{l}
		N(t) = \tau(t) S(t) \left[  \displaystyle \sum_{j=1}^n   \Gamma(t-t_0+a_j) \dfrac{I_0^j }{e^{-\nu a_j}}  + \displaystyle \int_0^{t-t_0} 	\Gamma(a) \, N(t-a) da  \right],
	\end{array}
\end{equation*}
where $S(t)$ is obtained from \eqref{2.4}.
\end{mybox}

\subsection{Basic reproduction number for the extended model \eqref{4.3}}
In this section, we 	assume that the transmission $t \to \tau(t) $ is constant equal to $\tau$, and $t \to S(t)$ is constant equal to $S_0$. 

\medskip 
Define  the  \textbf{daily reproduction numbers}
\begin{equation} \label{4.5}
	R_0(a)=\tau \times S_0 \times  \Gamma(a)= \tau \times S_0 \times  \beta\left(a\right)  \times  e^{-  \nu \, a }, \forall a \geq 0. 
\end{equation}
Assuming that the number of susceptible individuals $t \to S(t)$ is constant and equal $S_0$ in the $N$-equation \eqref{4.3}, then we obtain  
\begin{equation}\label{4.6}
	N(t) =   I_0 \times R_0(t-t_0) +  \int_0^{t-t_0} 	R_0(a) \, N(t-a) da ,  \hspace{0.2cm} \forall t \geq t_0. 
\end{equation}
By using the change of variable $s=t-t_0$, 
\begin{equation*} 
	N(s+t_0) = I_0 \times R_0(s) +  \int_0^{s} 	R_0(a) \, N(s+t_0-a) da ,  \hspace{0.2cm}  \forall s \geq 0. 
\end{equation*}
 Replacing the notation $s$ by $t$, and define 
$$
N_{t_0}(t)=N(t+t_0), \forall t \geq 0, 
$$
the equation \eqref{4.6} becomes 
\begin{equation} \label{4.7}
	N_{t_0}(t) = \left[  I_0 \times R_0(t) +  \int_0^{t} 	R_0(a) \, N_{t_0}(t-a) da \right], \forall t \geq 0. 
\end{equation}
By replacing $N_{t_0}(t)$ by the right hand side of \eqref{4.7} in the integral term of \eqref{4.7} we obtain 
\begin{equation*} 
	N_{t_0}(t)=I_0 \times R_0(t) +  \int_0^t 	R_0(a) \, I_0 \times R_0(t-a) da+  \int_0^t 	R_0(a_0)\int_0^{t-a_0} 	R_0(a_1) \, N_{t_0}(t-a_0-a_1) da_1 da_0
\end{equation*}
and by induction 
\begin{equation*} 
	N_{t_0}(t)=I_0  R_0(t) + I_0   (R_0*R_0)(t)+ I_0  (R_0*R_0*R_0)(t)+\ldots+  I_0 \underbrace{ (R_0*R_0* \ldots*R_0)}_{\text{n times}} (t)+ \ldots
\end{equation*}
where the convolution is defined by  
\begin{equation*}  
	(U*V)(t)=\int_0^t U(a)V(t-a)da=\int_0^t U(t-a)V(a)da.
\end{equation*}
We define 
\begin{equation*} 
	\left( U^{*(2)} \right) (t)=(U*U)(t),
\end{equation*}
and for each integer $n \geq 3$,
\begin{equation*}
	\left(U^{*(n)} \right)(t)=\left(U^{*(n-1)} *U \right)(t)=\left(U*U^{*(n-1)} \right)(t)=\underbrace{ (U*U* \ldots*U)}_{\text{n times}}(t).
\end{equation*}
We can interpret the $N$-equation \eqref{4.6} concretely as follows 
\begin{equation*}
	\begin{array}{ll}
		N_{t_0}(t)=& \underbrace{ I_0 \, R_0(t)}_{	\parbox{300pt}{\scriptsize\centering Flow of infected produced at time $t$ by the first $I_0$ infected individuals}}\vspace{0.2cm}\\
		&+ \underbrace{I_0   \,\left( R_0^{*(2)} \right) (t)}_{ \parbox{300pt}{\scriptsize\centering Flow of infected produced at time $t$ by the second generation of infected individuals}} \vspace{0.2cm}\\
		&+ \underbrace{I_0  \, \left( R_0^{*(3)} \right)(t)}_{	\parbox{300pt}{\scriptsize\centering Flow of infected produced at time $t$ by the third generation of infected individuals}}  \vspace{0.2cm}\\
		&+\\
		&\vdots\\
		&+ \underbrace{I_0 \, \left( R_0^{*(n)} \right)(t)}_{\parbox{300pt}{\scriptsize\centering Flow of infected produced at time $t$ by the $n^{th}$ generation of infected individuals}} \vspace{0.2cm}\\
		&+\\
		&\vdots
	\end{array}	
\end{equation*}
\begin{mybox}{Basic reproduction number}
The total number of the first generation of new infected produced by a single infected patient with age of infection $a=0$ at time $t=t_0$ is called the \textbf{basic reproduction number}. That is  
$$
R_0= \int_{0}^{\infty} R_0(a)\, da. 
$$
The flow of the first generation  of new infected produced  by a single infected patient have been infected for $a$ days  is called the  \textbf{daily reproduction numbers}. When the time scale is one day, the function $R_0(a)$ is also the average daily number of case produced by a single patient at the age of infection $a$.   
\end{mybox}

\begin{proposition}
	The total number of cases produced by the $n^{th}$ generation of infected resulting from a single infected patient is 
	\begin{equation*}
		\int_{0}^{\infty}  \left(R_0^{*(n)} \right)(t)dt=\left(R_0 \right)^n.
	\end{equation*} 
\end{proposition}
\begin{proof} By using Fubini's theorem we have 
	\begin{equation*}
		\int_{0}^{\infty}	(R_0*V)(t)\, dt	=\int_{0}^{\infty}  \int_0^t R_0(t-a)V(a)\, da \, dt=	\int_{0}^{\infty}  \int_a^{+ \infty}  R_0(t-a)\,dt \,V(a)\, da 
	\end{equation*}
	and by making the change of variable $l=t-a$ we obtain  
	\begin{equation*}  
		\int_{0}^{\infty}	\left(R_0*V \right) (t)\, dt=	\int_0^{+ \infty} R_0(l)\,dl   \times \int_0^{+ \infty} V(a)\, da.
	\end{equation*}
	Replacing $V(t)$ by $(R_0^{*(n-1)})(t)$ in the above equation we obtain 
	\begin{equation*}
		\int_{0}^{\infty}  \left( R_0^{*(n)} \right) (t)dt=	\int_0^{+ \infty} R_0(l)\,dl  \int_{0}^{\infty}  \left( R_0^{*(n-1)} \right) (t)dt,
	\end{equation*}
	and the result follows by induction. 
\end{proof}

\section{Computing the age dependent reproduction number $	\Gamma(a)$  from the data}
\label{Section5}
By using \eqref{4.3}, we obtain the following result.
\begin{mybox}{Computing  $	\Gamma(a)$  from the data}
Assume in addition that the parameters $t_0, S_0>0, I_0, \nu > 0,$ and the function $t \to \tau(t)$ are given. Then the function $t \to \Gamma(t)$ can be obtained from  the flow of new infected  $t \to N(t)$,  as the unique solution of the Volterra integral equation 
	\begin{equation}\label{5.1}
		\Gamma(t-t_0)=\dfrac{ 1}{I_0}\left( \dfrac{N(t)}{\tau(t) S(t)} -   \int_0^{t-t_0} 	\Gamma(a) \, N(t-a) da \right), \forall  t \geq t_0,  
	\end{equation}
where $S(t)$ is obtained by using \eqref{2.4}.
\end{mybox}
\begin{remark} 
	Assume that patients can not transmit the pathogen when the age of infection is above $a^+>0$. That is  
	 $$
	 	\Gamma(a)=0, \forall  a \geq a^+.
	 $$
Then the equation \eqref{5.1} becomes for all $t \geq t_0+a^+$,
	 $$
	 \dfrac{N(t)}{ \tau(t) S(t)} = \int_0^{a^+} 	\Gamma(a) \, N(t-a) da \Leftrightarrow N(t)= \tau(t) S(t) \int_0^{a^+} 	\Gamma(a) \, N(t-a) da. 
	 $$
	\end{remark}
\subsection{Some explicit examples of $ a \to R_0(a)$}
We assume that the transmission $t \to \tau(t) $ is constant equal to $\tau$, and $t \to S(t)$ is constant equal to $S_0$.  Then 
$$
R_0(a)= \tau \times S_0 \times  \Gamma(t-t_0),
$$
and by setting  $ a=t-t_0$, the equation  \eqref{5.1} becomes 
	\begin{equation}\label{5.5}
R_0(a)=   \dfrac{N(a+t_0)}{I_0}  - \dfrac{1}{I_0}  \int_0^{a} R_0(\sigma) \, N(a- \sigma+t_0)  d \sigma, \forall a \geq 0. 
\end{equation}
 
\subsubsection{Exponential decay}
Assume that 
\begin{equation*}
 N(t)=\chi_1 \, I_0 \, e^{-\chi_2 \left(t-t_0\right) }. 
\end{equation*}
Then we obtain 
	\begin{equation*} 
R_0(a)= \chi_1 \,  e^{-\chi_2 a } -  \int_0^{a} R_0(\sigma) \, \chi_1 \,e^{-\chi_2 \left(a-\sigma \right) }   d \sigma, \forall  a \geq 0. 
\end{equation*}
Therefore 
\begin{equation*}
		R_0'(a)=- \left( \chi_1 + \chi_2   \right) R_0(t), \forall a \geq 0, \text{ and } R_0(0)= \chi_1 ,
\end{equation*}
and we obtain 
\begin{equation}
R_0(a)=\chi_1  e^{- \left( \chi_1  + \chi_2   \right) a} , \forall a \geq 0. 
\end{equation}

\subsubsection{Exponential decay with latency}
Let $t_1 >t_0$. Assume that 
\begin{equation*}
	N(t)= \left\{ 
	\begin{array}{l}
		I_0 \, \chi_1 \, e^{-\chi_2 \left(t-t_0\right) }, \text{ if } t \geq t_1, \vspace{0.2cm}\\
		0,  \text{ if } t_0 \leq t < t_1.
	\end{array}
\right.
\end{equation*}
Then we obtain 
\begin{equation*}
	R_0(a)= 0,  \forall a \in \left[0,t_1-t_0 \right], 
\end{equation*}
and 
\begin{equation*}
R_0(a)=\chi_1 \,  e^{-\chi_2 a } -  \int_{t_1-t_0}^{a}  R_0(\sigma)  \, \chi_1 \, e^{-\chi_2 \left(a -\sigma \right) } \, 	 d \sigma , \forall  t \geq t_1.    
\end{equation*}
By setting $a_1=t_1-t_0$, we obtain 
\begin{equation*}
R_0(a)= 0,  \forall  a \in \left[0,a_1 \right), 
\end{equation*}
\begin{equation*}
R_0(a)=  \chi_1 \,  e^{-\chi_2a  }  -  \int_{a_1}^{a}  R_0(\sigma) \, \chi_1 \, e^{-\chi_2 \left(a -\sigma \right) }   d \sigma , \forall  a \geq a_1,    
\end{equation*}
hence by setting $a=\widehat{a}+a_1$,
\begin{equation*}
R_0(\widehat{a}+a_1)=  \chi_1 \,  e^{-\chi_2 \left( \widehat{a}+a_1\right) } -  \int_{a_1}^{\widehat{a}+a_1} \chi_1 \, e^{-\chi_2 \left(\widehat{a}+a_1-\sigma \right) } \, 	R_0(\sigma)  d \sigma , \forall  \widehat{a} \geq 0, 
\end{equation*}
and by setting $\widehat{a}=a-a_1$, and $ \widehat{R}_0(\widehat{a})=R_0(\widehat{a}+a_1)$, we obtain 
\begin{equation*}
\widehat{R}_0(\widehat{a})=  C_1  e^{-\chi_2 \widehat{a}  }-  \int_{0}^{\widehat{a}} \chi_1 \, e^{-\chi_2 \left(\widehat{a}-\sigma \right) } \, 	\widehat{R}_0(\sigma)  d \sigma ,  \forall  \widehat{a} \geq 0,  
\end{equation*}
where 
$$
C_1= \chi_1 \,  e^{-\chi_2 \, a_1  }  .
$$
Thus 
\begin{equation*}
	\widehat{R}_0(\widehat{a})= C_1 e^{-\left( \chi_1+ \chi_2 \right) \widehat{a} },  \forall  \widehat{a} \geq 0.  
\end{equation*}
We conclude that 
\begin{equation}
R_0(a)= \left\{ 
	\begin{array}{ll}
	0, &\text{ if } 0 \leq a \leq a_1 , \vspace{0.2cm}\\
	\chi_3   e^{- \left( \chi_1+ \chi_2 \right) \, a  } , & \text{ if }  a > a_1,
	\end{array}
	\right.
\end{equation}
where 
$$
\chi_3= \chi_1 \,  e^{\chi_1 \, a_1  }     , \text{ and }   a_1=t_1-t_0. 
$$ 

\section{Day by day Kermack-McKendrick model with age of infection} 
\label{Section6}
The variation of the number of susceptible individuals $S(t)$ is given each day $ t=t_0, t_0+1 ,\ldots $, by 
\begin{equation} \label{6.1}
	S(t)=S_0-\sum_{d=t_0}^{t-1} N(d) , 
\end{equation}
where $S_0$ is the number of susceptible on day $0$, $S(t)$ is the number of susceptible on day $t$ and $N(d)$ is the daily number of new infected individuals on day $d$.
By analogy with the equation \eqref{2.6}, the daily number of new infected individuals satisfies the following discrete time Volterra integral equation for all $\forall t=t_0, t_0+1, t_0+2 ,\ldots,$

\begin{equation} \label{6.2}
	N(t)=\tau(t) S(t)  \sum_{d=t-t_0}^{+\infty}  \Gamma(d) \dfrac{  I_0(d-(t-t_0)) }{e^{-  \nu \, (d-(t-t_0))} } +\tau(t) S(t) \sum_{d=1}^{t-t_0} \Gamma(d) \times  N(t-d), 
\end{equation}  
where 
\begin{equation*}  
	\Gamma(d):=  \beta(d) \,  e^{-  \nu \, d} , \forall d=0, 1, 2, \ldots,
\end{equation*}
and $\tau(t)$ is the transmission rate, $\beta(d)$ is the probability to be infectious (i.e. capable to transmit the pathogen) after $d$ days of infection and $e^{-  \nu \, d}  $ is the probability to stay infected after  $d$ days of infection (i.e. the probability not to recover or to die after $d$ days of infection).  The quantities $I_0(d)$ is the number of infected on day $0$ which have being infected $d$ days ago.

\medskip 
The model \eqref{4.3} with a single cohort of infected becomes 
\begin{equation} \label{6.3}
	N(t)= \tau(t) S(t) \left[   \Gamma(t-t_0) I_0(0)+  \sum_{d=1}^{t-t_0} \Gamma(d) \times  N(t-d) \right].
\end{equation}  

\begin{mybox}{Day by day single cohort model and daily basic reproduction number}
	Assume that $t \to \tau(t)$ equal $\tau_0$, and $t \to S(t)$ is constant equal to $S_0$. Assume that the epidemic starts at time $t_0$ with a cohort of $I_0$ new infected patient (i.e. with age of infection $a=0$).  The model \eqref{4.6} with a single cohort of infected becomes  a discrete Volterra equation
	\begin{equation} \label{6.4}
		N(t)= \left[  R_0(t-t_0) \times I_0+  \sum_{d=1}^{t-t_0}R_0(d) \times  N(t-d) \right], \forall t \geq t_0.
	\end{equation}  
We obtain the day by day equation for the daily reproduction number  
\begin{equation} \label{6.5}
R_0(a) = \dfrac{N(t_0+a)}{ I_0}-  \dfrac{1 }{I_0}\sum_{d=1}^{a} R_0(d) \times  N(t_0+a-d), \forall a \geq 0.
\end{equation}  
\end{mybox}

\section{Numerical simulations}
\label{Section7}
\subsection{Comparison of deterministic and stochastic simulations}
In  the simulations, the unit of time is one day, and we fix 
$$
S_0=10^7=10\, 000 \, 000, \; 1/\nu=9 \text{ days}, \text{ and } R_0=1.1.
$$
For each function $\beta(a)$ described below, the parameter $\tau$ is obtained numerically by using the following formula 
$$
\tau=\dfrac{R_0}{S_0 \int_0^\infty \beta(a) e^{-\nu a}da }, 
$$
where the integral is computed by using the Simpson integration method. 

\medskip 
 In the following, we use the numerical scheme described in Appendix \ref{AppendixA} to run the simulation of the Volterra integral equation  \eqref{4.3}-\eqref{4.4}. We use the Individual Based Model (IBM) described in  Appendix \ref{AppendixB} to run the stochastic simulations of the model. 
  In the following, we illustrate the convergence of the IBM to the deterministic model whenever $I_0$ increases.

\subsubsection{Example 1}
We assume that the probability to be infectious  is a  shifted gamma like distribution. That is 
\begin{equation}\label{7.2}
\beta(a)=\beta_0 (a-a_0)^+ e^{- \beta_1 (a-a_0)^+ },
\end{equation}
with 
$$
a_0=3\text{ days}, \beta_0=e/2=1.3591, \text{ and }\beta_1=1/2=0.5.
$$ 
  \begin{figure}[H]
	\begin{center}
		\includegraphics[scale=0.15]{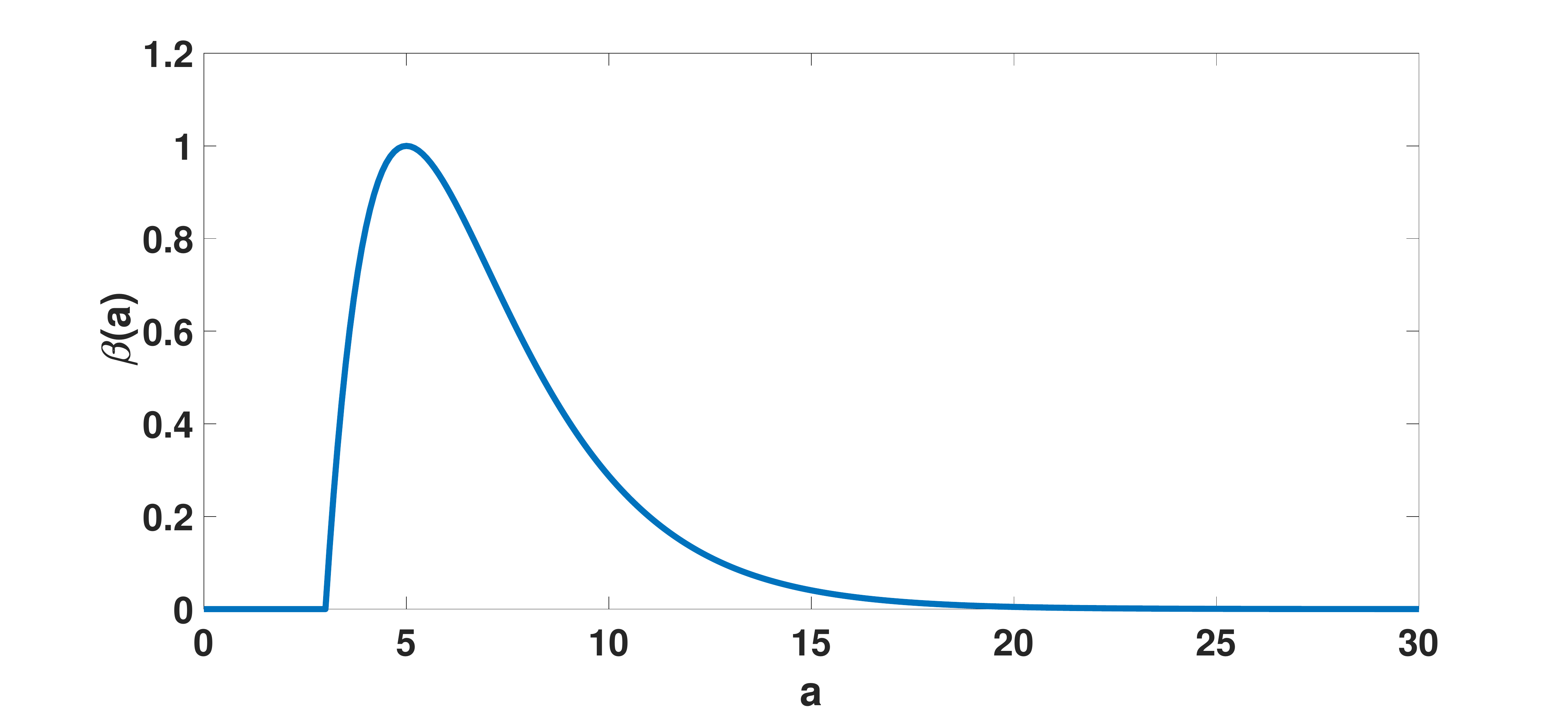}
		\includegraphics[scale=0.15]{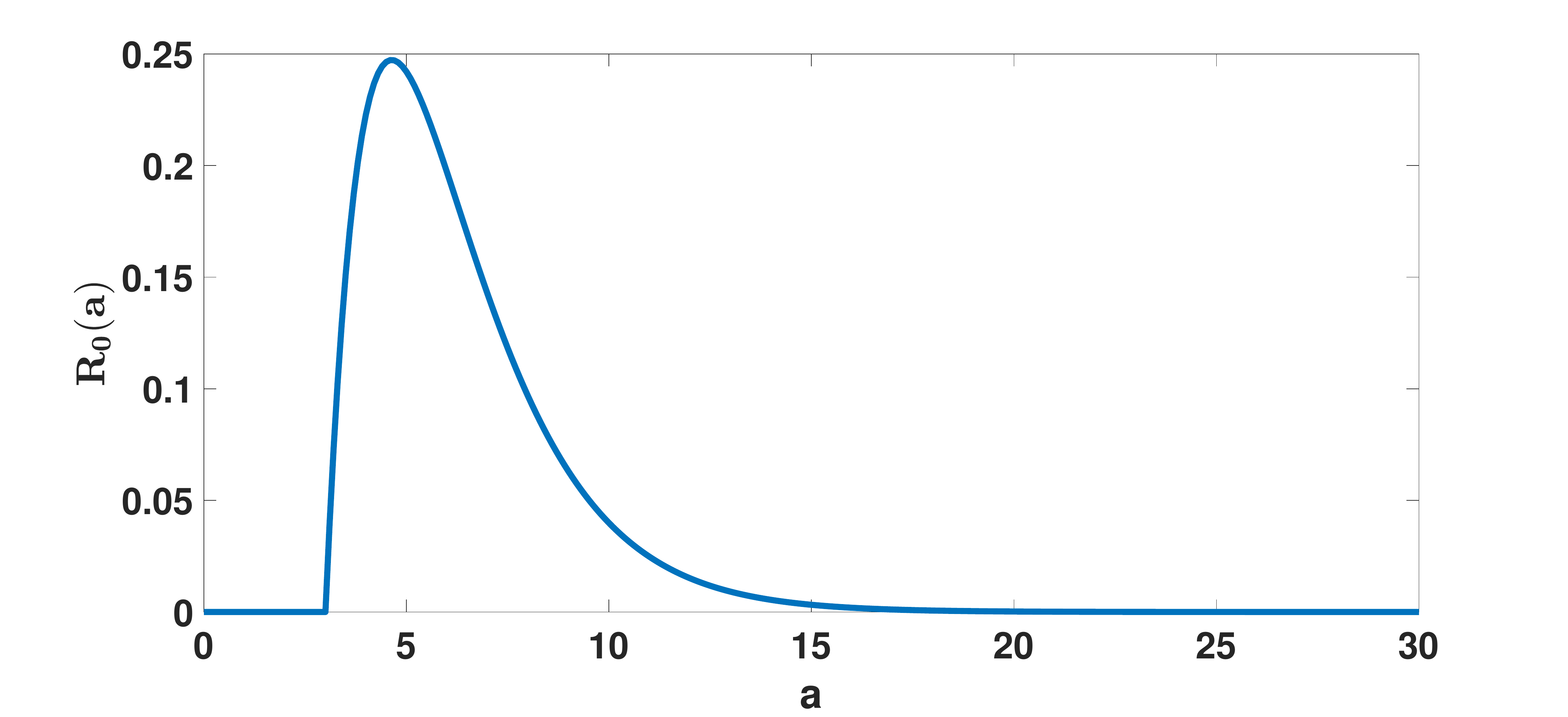}
	\end{center}
	\caption{\textit{On the left-hand side,  we plot the function $a \to \beta(a)$. On the right-hand side, we plot the function $a \to R_0(a)=\tau_0 \times S_0  \times \beta(a) \times  e^{-\nu a}$.}}\label{Fig14}
\end{figure}
In the Figures \ref{Fig15}-\ref{Fig16} we use the IBM  to investigate some properties of the clusters obtained from the stochastic simulations. We compare such a stochastic sample with the original  $a \to R_0(a)$. 
		
		By comparing Figures \ref{Fig17}-\ref{Fig18} (for $I_0=10$) and   Figures \ref{Fig22}-\ref{Fig23} (for $I_0=1\, 000$), we observe the convergence of the IBM to the deterministic model.  
		
		Then in Figures \ref{Fig19}-\ref{Fig21} (for $I_0=10$) and   Figures \ref{Fig24}-\ref{Fig26} (for $I_0=1\,000$), we apply the discrete-time equation \eqref{6.5} to reconstruct $a \to R_0(a)$ from the trajectories of the deterministic or stochastic models. In the deterministic model, we observe the effect of the day-by-day discretization (which corresponds to the daily reported data). In the stochastic case, we observe the effect of the stochasticity of the IBM. 
		
\medskip 
\noindent \textbf{First generation of secondary cases produced by a single infected:}
\begin{figure}[H]
	\centering
	\begin{minipage}{0.48\textwidth}
		\centering
		\textbf{(a)} \\
		\includegraphics[scale=0.15]{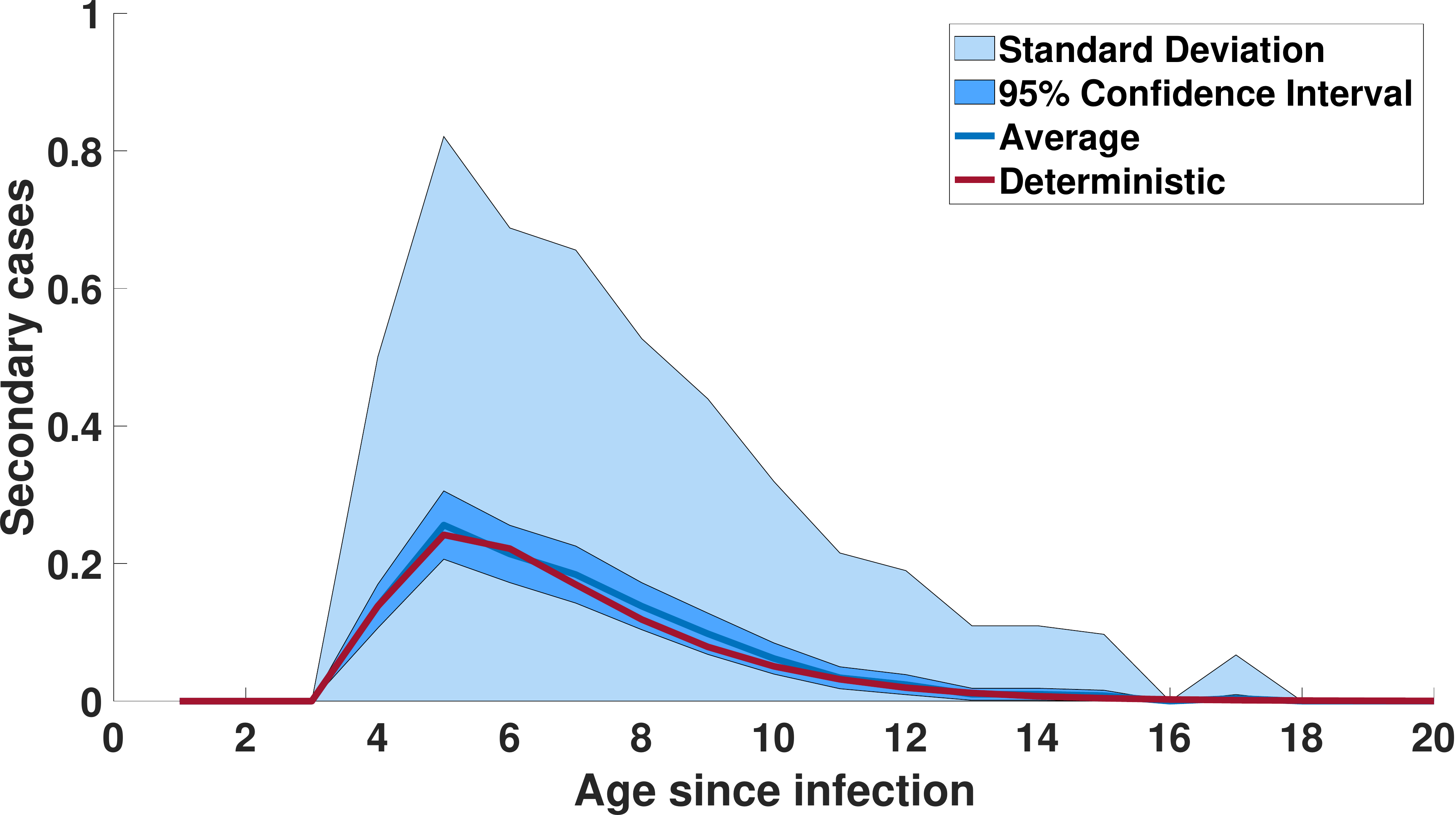}
	\end{minipage} 
	\begin{minipage}{0.48\textwidth}
		\centering
		\textbf{(b)} \\
		\includegraphics[scale=0.15]{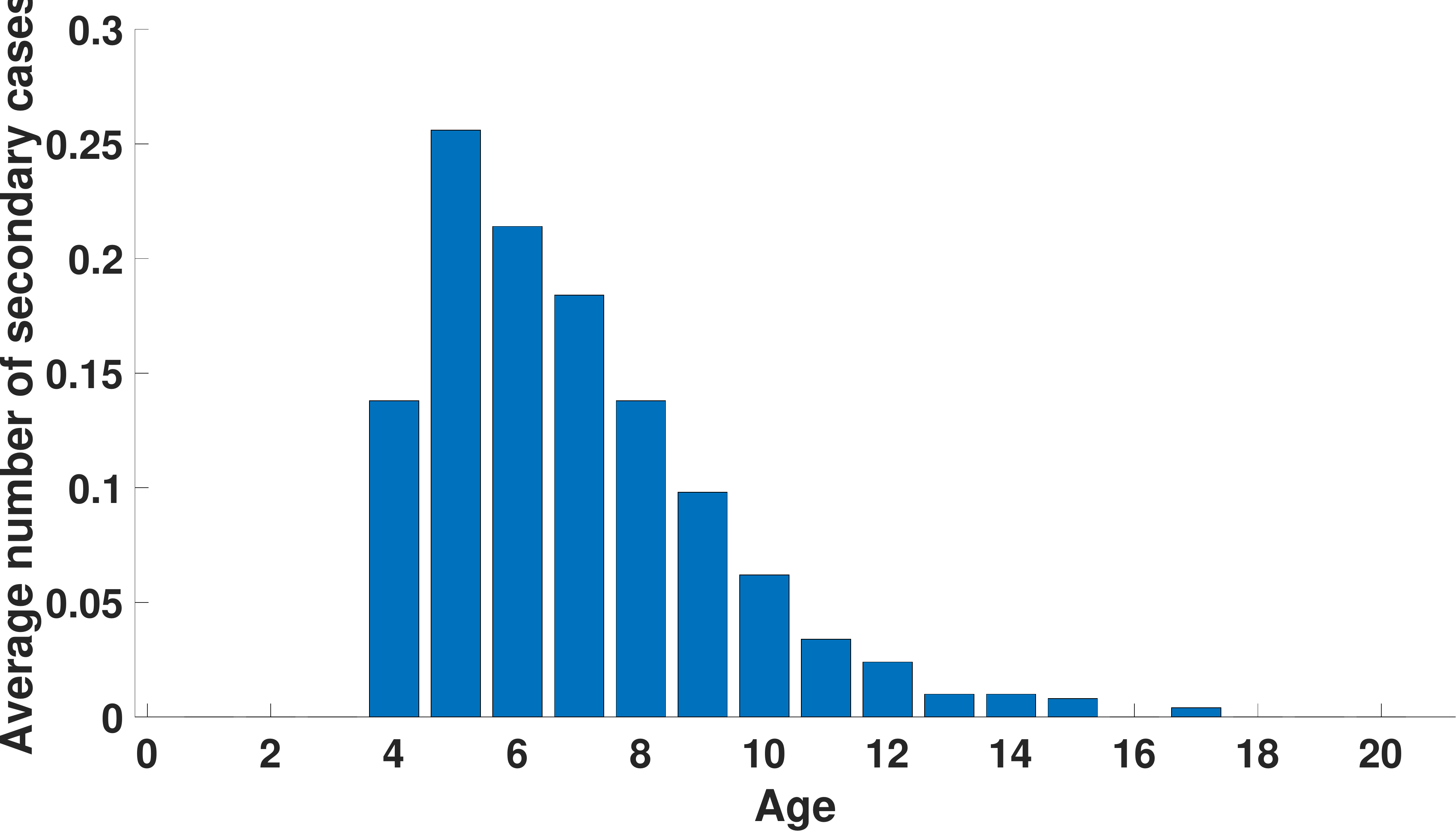}
	\end{minipage} 
	\caption{\textit{In these figures, we present  sets of 500 samples of secondary cases produced by a single infected individual in a population of $S=10^7$ susceptible hosts. Theses samples are produced by using the IBM. \textbf{(a)} Statistical summary: the blue curve represents the average number of cases at age of infection $a$; the dark blue area is the 95\% confidence interval of this average obtained by fitting a Gaussian distribution to the data; the light blue area corresponds to the standard deviation; the orange curve is the deterministic daily basic reproductive number at age $a$. \textbf{(b)} Bar graph of the average number of secondary cases as a function of the age since infection.}} \label{Fig15}
\end{figure}
\begin{figure}[H]
	\begin{center}
		\includegraphics[scale=0.15]{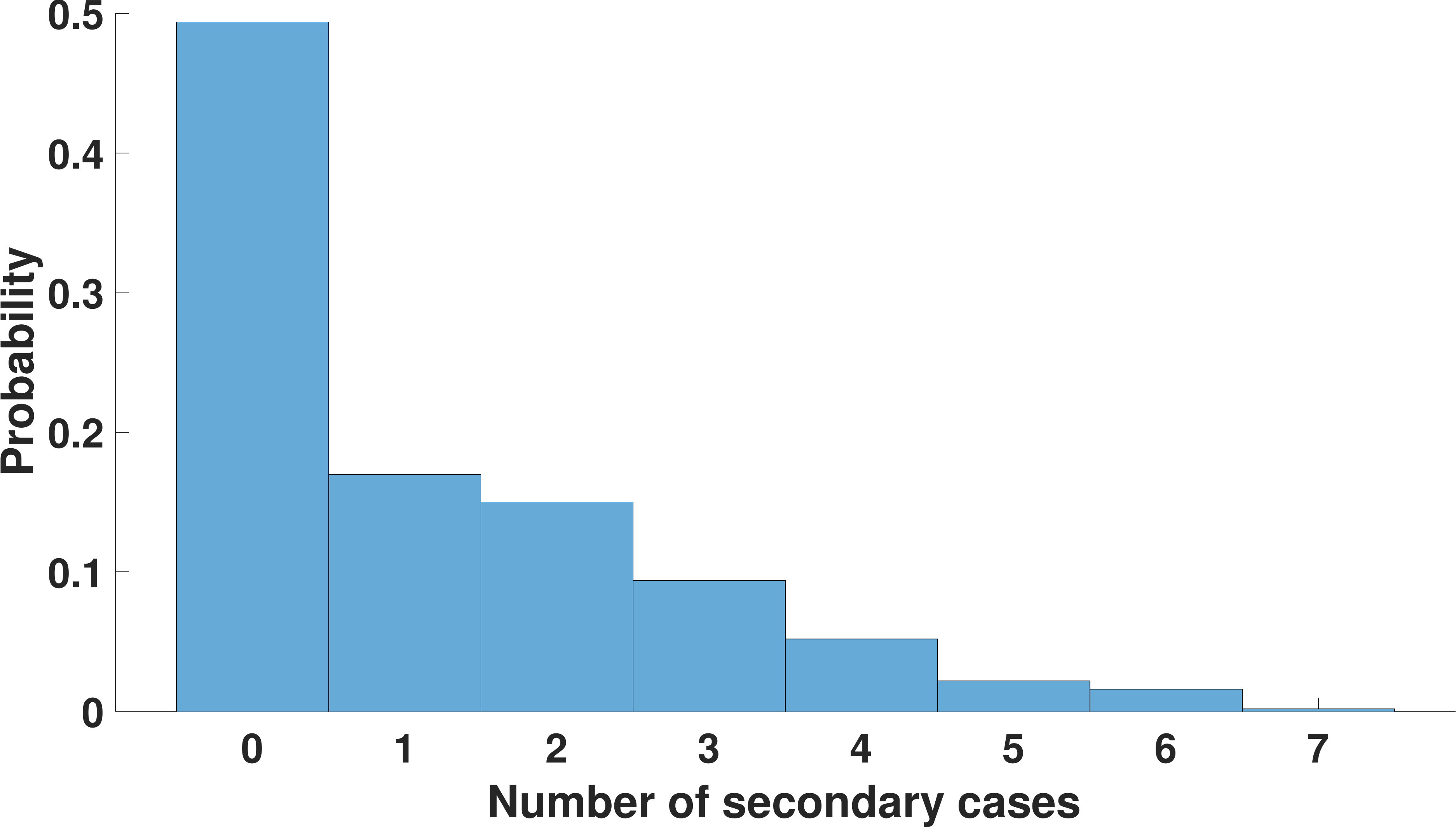}
	\end{center}
	\caption{\textit{In these figures, we present  sets of 500 samples of secondary cases produced by a single infected individual in a population of $S=10^7$ susceptible hosts. Theses samples are produced by using the IBM. We plot a histogram of the total number of secondary cases produced during the whole infection. This estimates the probability of a single infected to generate $n$ secondary cases (with $n$ in the abscissa).}} \label{Fig16}
\end{figure}
\noindent \textbf{Simulations for $I_0=10$:}
 \begin{figure}[H]
	\begin{center}
		\includegraphics[scale=0.15]{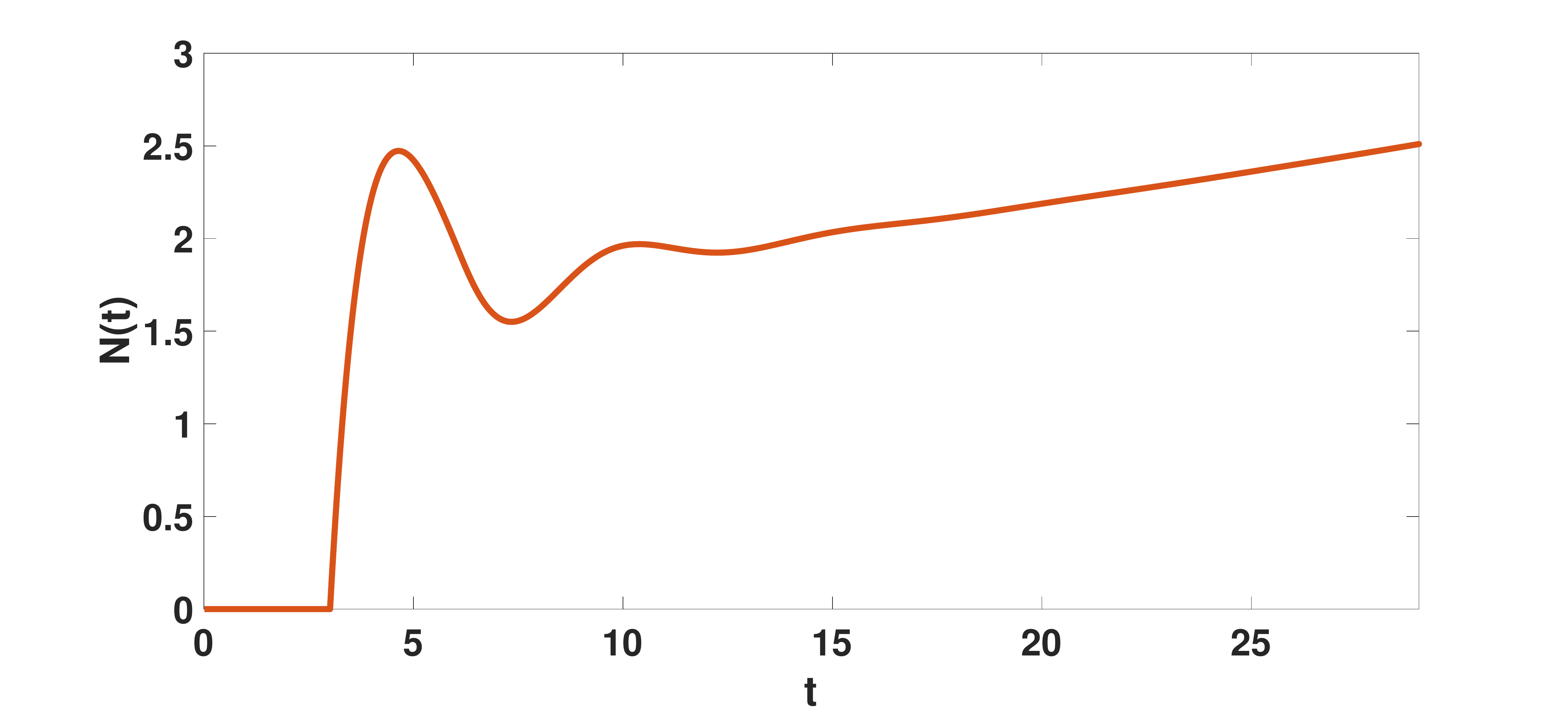}
		\includegraphics[scale=0.15]{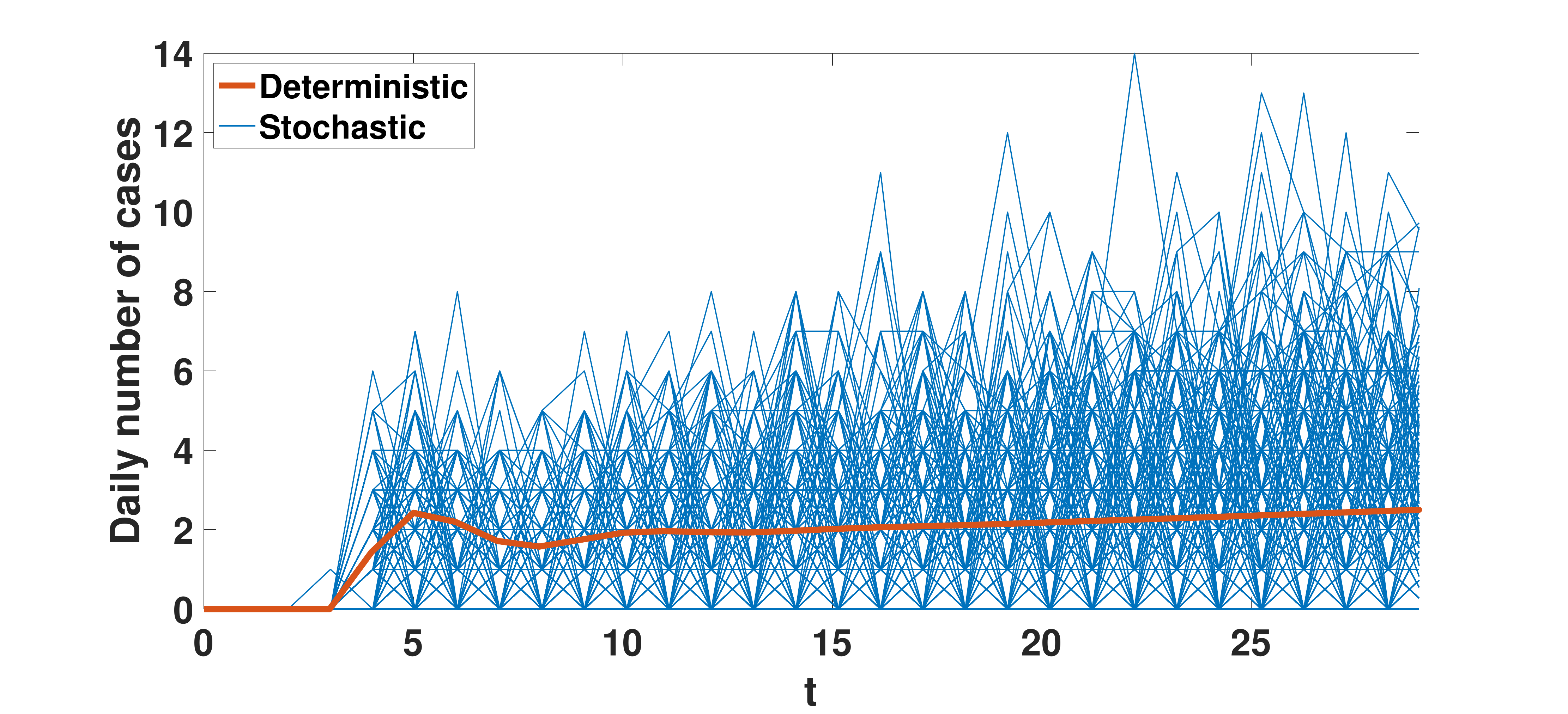}
	\end{center}
	\caption{\textit{On the left-hand side,  we plot the function $t \to N(t)$ solution of \eqref{4.3} with \eqref{2.4}. On the right-hand side, we plot the function $t \to \int_{t-1}^{t} N(s)ds$ (for $t=1,2, \ldots$) which corresponds to the daily number of cases obtained from by solving \eqref{4.3} with \eqref{2.4}, and we compare it with the daily number of cases obtained from $500$ runs of the IBM.}}\label{Fig17}
\end{figure}

\begin{figure}[H]
	\begin{center}
		\includegraphics[scale=0.15]{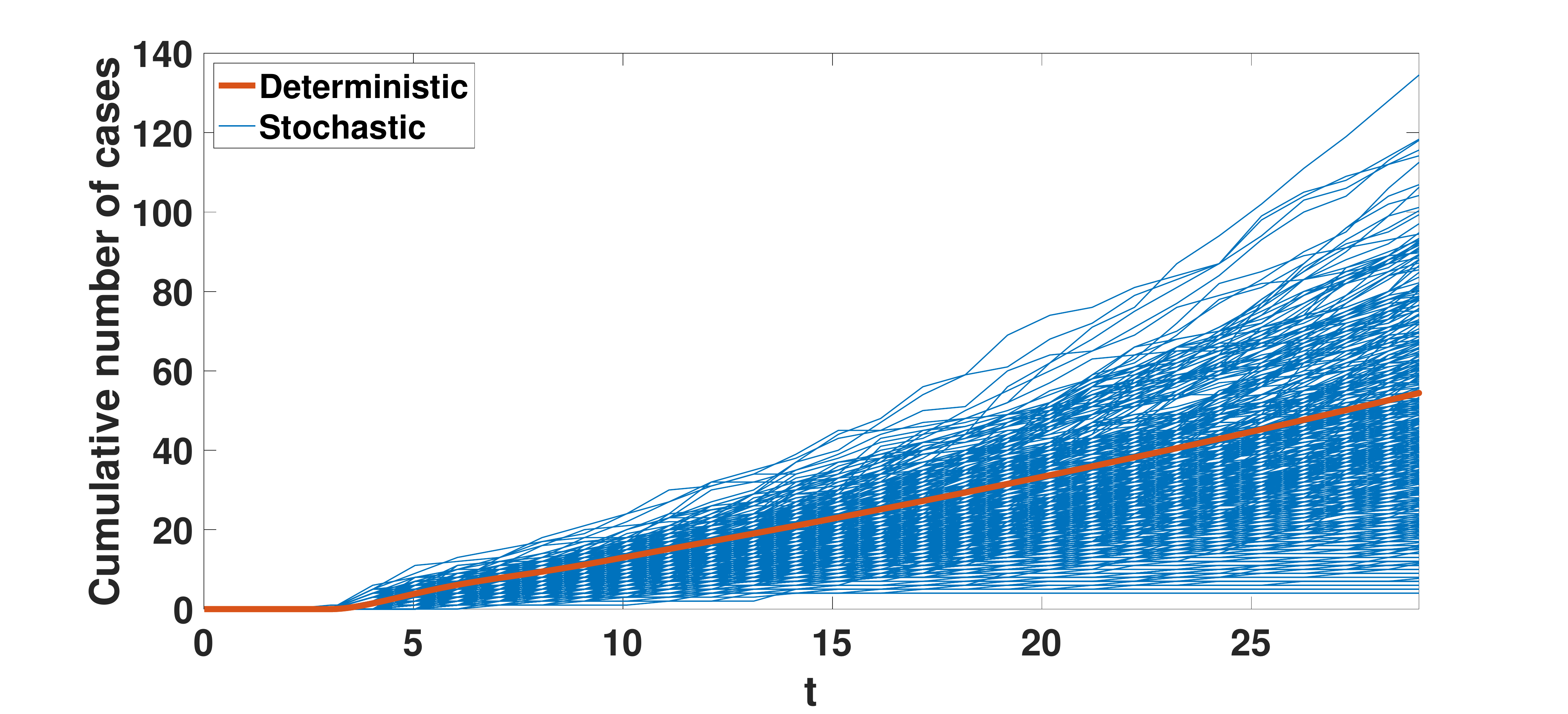}
		\includegraphics[scale=0.15]{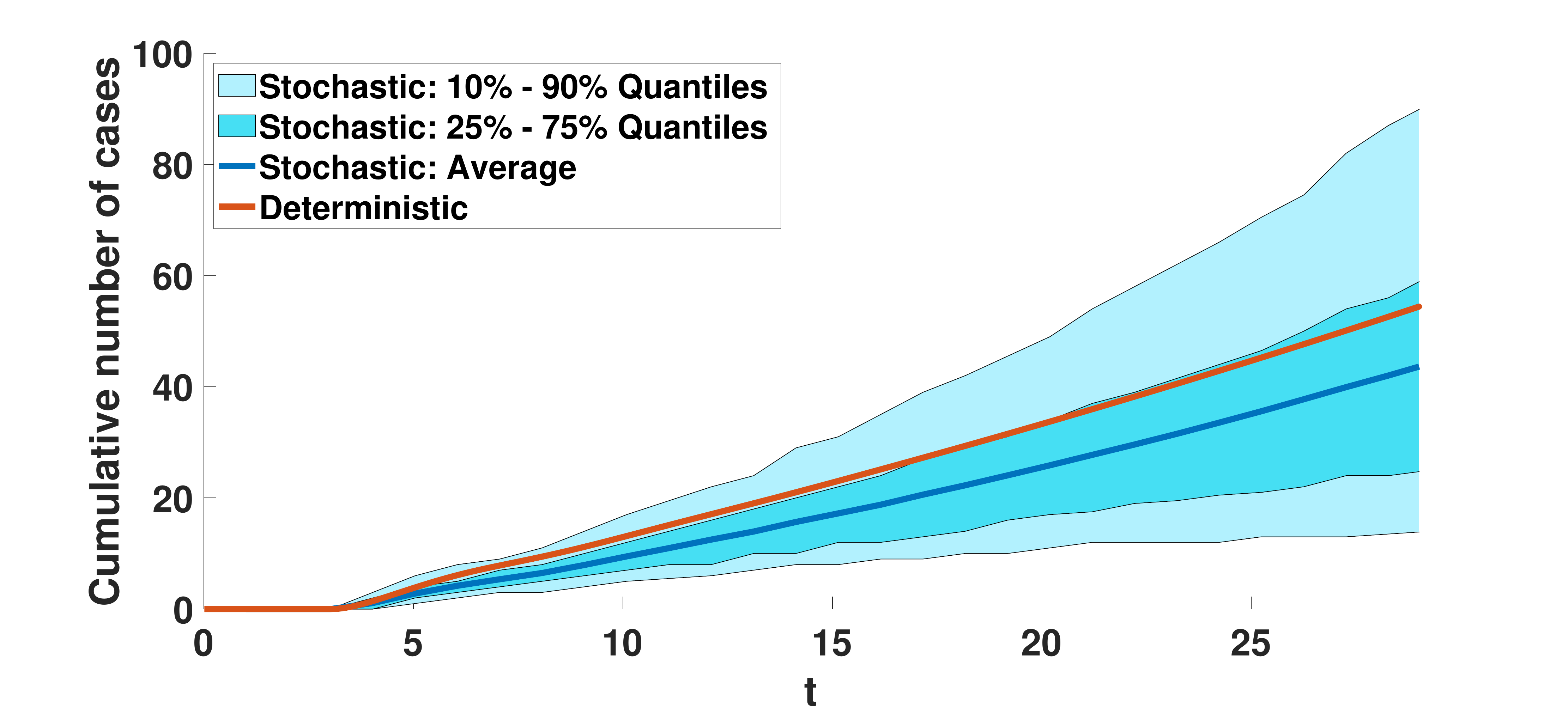}
	\end{center}
	\caption{\textit{On the left-hand side,  we plot the function $t \to \int_0^t N(s)ds$ (for $t=0,1,2, \ldots$) which corresponds to the cumulative number of cases obtained from by solving \eqref{4.3} with \eqref{2.4}, and we compare it with the cumulative number of cases obtained from $500$ runs of the IBM. On the right-hand side, we plot the average values of the $500$ runs obtained from the IBM as well as the quantiles ($10\%-90\%$ (light blue) and  $25\%-75\%$ (blue)).}}\label{Fig18}
\end{figure}

In Figures \ref{Fig19}-\ref{Fig21} we focus on the reconstruction of the daily reproduction number  $R_0(a)=\tau \, S_0 \,  e^{-  \nu \, a } \,  \beta\left(a\right)$. In Figure \ref{Fig19}, we focus on the reconstruction of the daily reproduction number  from deterministic simulations, while in Figures \ref{Fig20}-\ref{Fig21} we focus on the reconstruction of the daily reproduction number  from stochastic simulations. 
\begin{figure}[H]
	\begin{center}
		\includegraphics[scale=0.15]{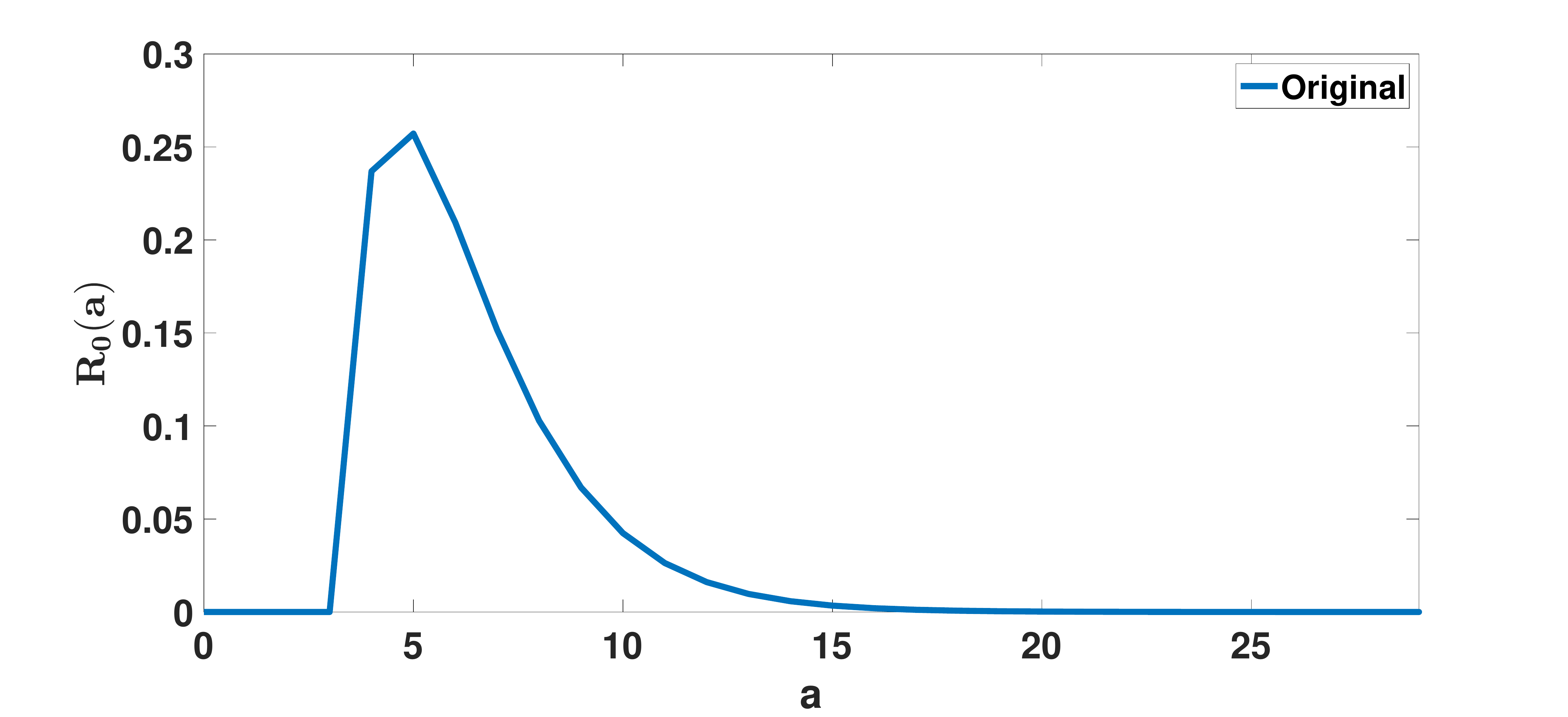}
		\includegraphics[scale=0.15]{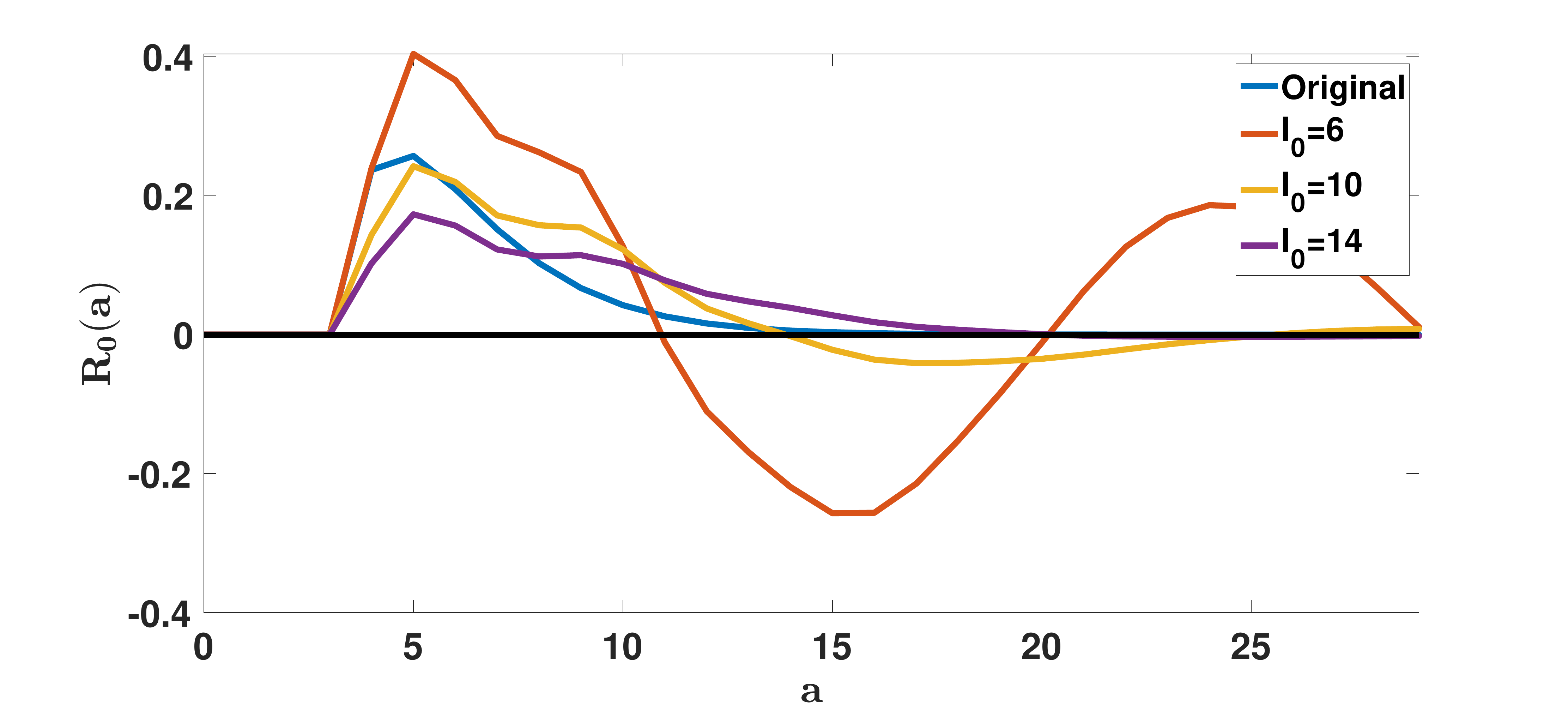}
	\end{center}
	\caption{\textit{On the left hand side, we plot the daily basic reproduction number by using the original formula $ R_0(a)$ with \eqref{7.2}.  On the right-hand side, we apply formula \eqref{6.5} to the flow of new infected obtained from the deterministic model. We vary $I_0=6, 10, 14$. The value $I_0=10$ corresponds to the value used for the simulation of the deterministic model.   The yellow curve gives the best visual fit, and the $R_0(a)$ becomes negative whenever $I_0$ becomes too small. }}\label{Fig19}
\end{figure}

\begin{figure}[H]
	\begin{center}
		\includegraphics[scale=0.15]{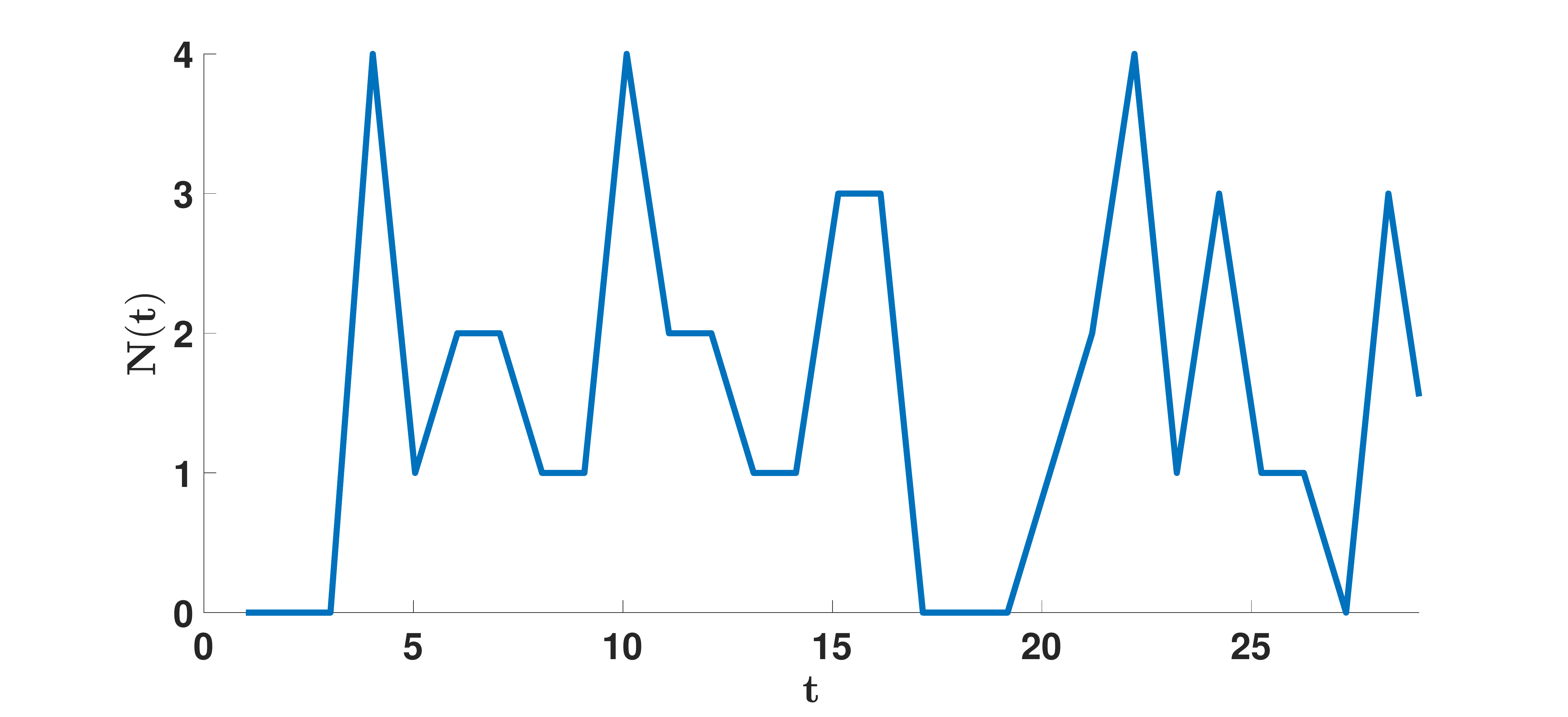}
		\includegraphics[scale=0.15]{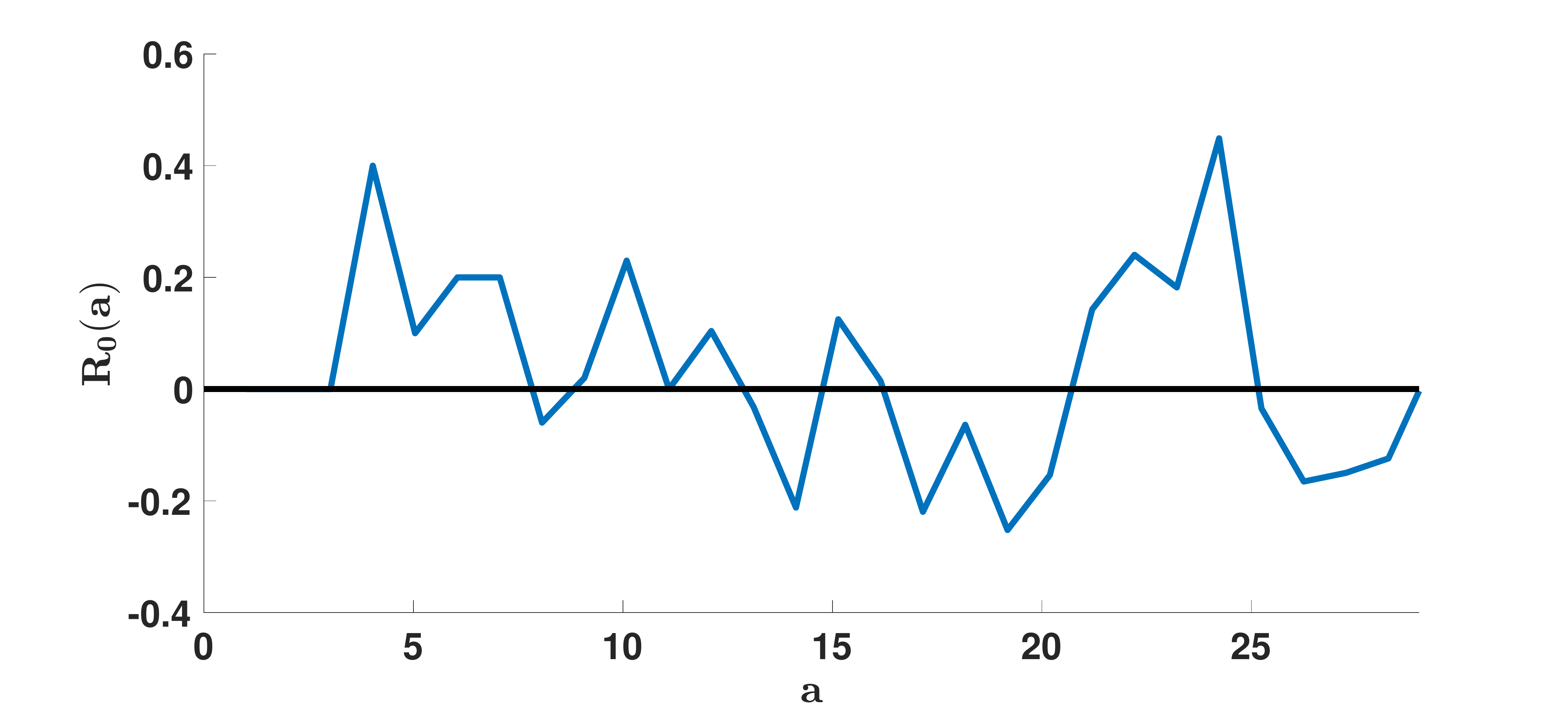}
	\end{center}
	\caption{\textit{On the left-hand side,  we plot the daily number of cases $t \to N(t) $ (for $t=0,1,2, \ldots$) obtained from a single run of the IBM.  On the right-hand side, we apply formula \eqref{6.5}  (with $I_0=10$) to  the daily number of cases of case obtained from the IBM.   }}\label{Fig20}
\end{figure}

\begin{figure}[H]
	\begin{center}
		\includegraphics[scale=0.15]{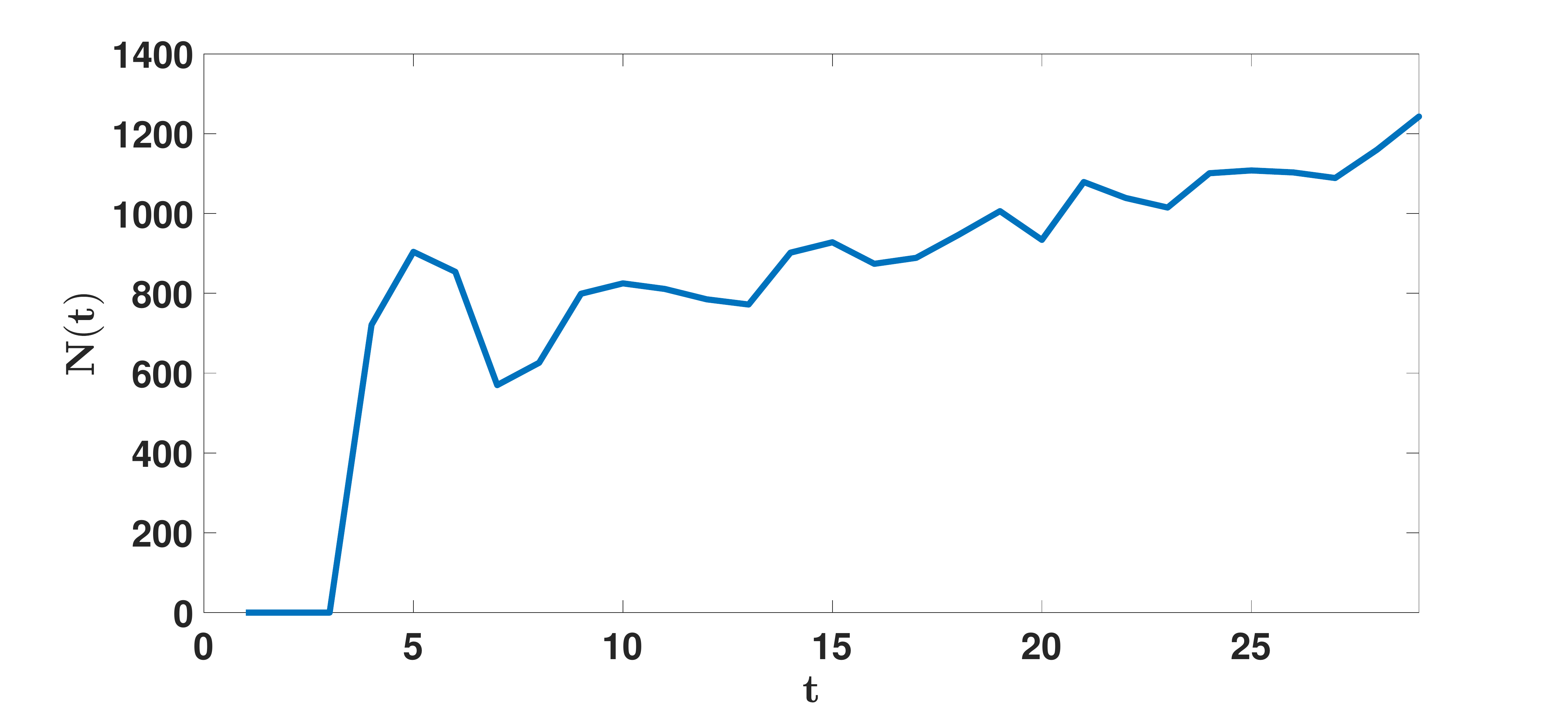}
		\includegraphics[scale=0.15]{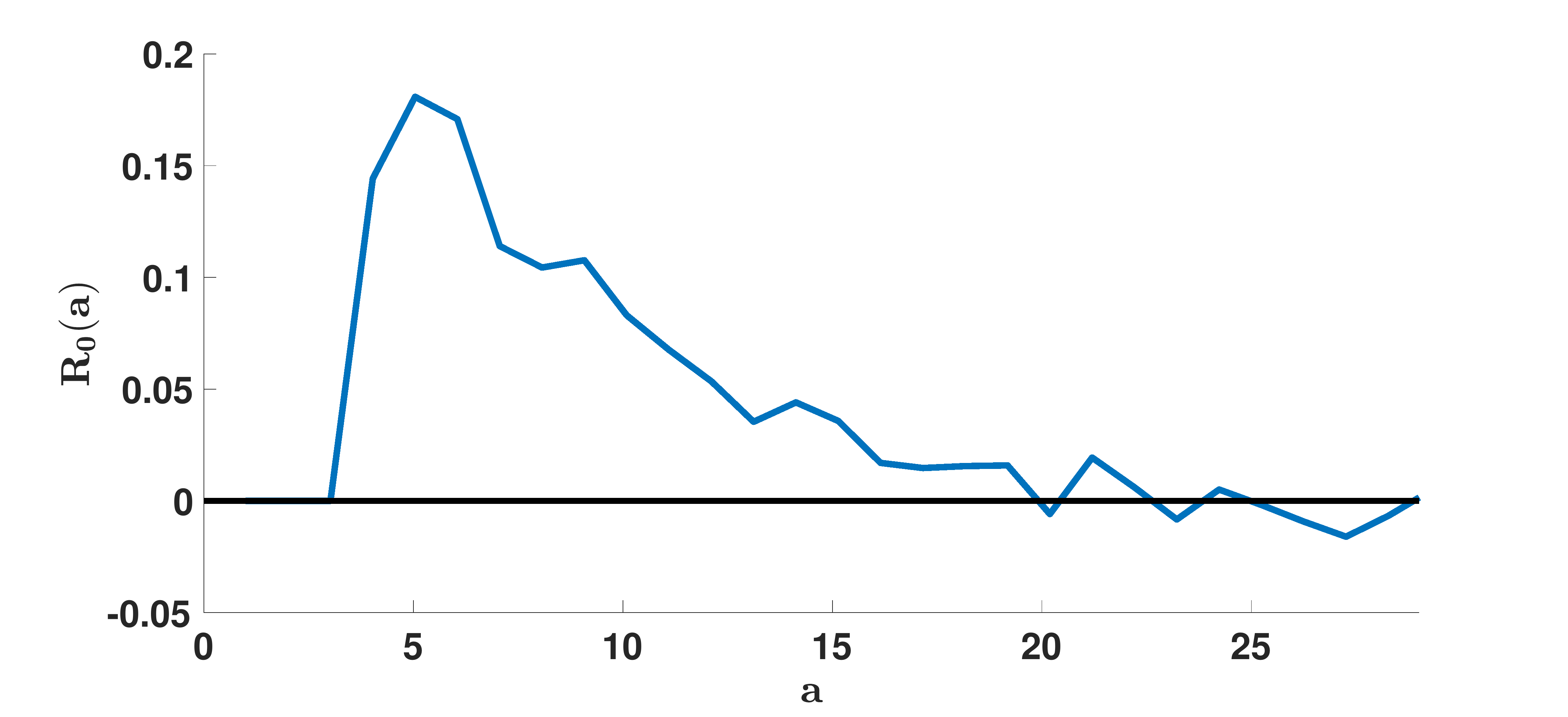}
	\end{center}
	\caption{\textit{On the left-hand side,  we plot the daily number of cases $t \to N(t)$ (for $t=0,1,2, \ldots$) obtained by summing the daily number of cases for $500$ runs of the IBM.  On the right-hand side, we apply formula \eqref{6.5}  (with $I_0=500 \times 10$) to the daily number of cases obtained from the IBM. }}\label{Fig21}
\end{figure}
\noindent \textbf{Simulations for $I_0=1000$:}
 \begin{figure}[H]
	\begin{center}
		\includegraphics[scale=0.15]{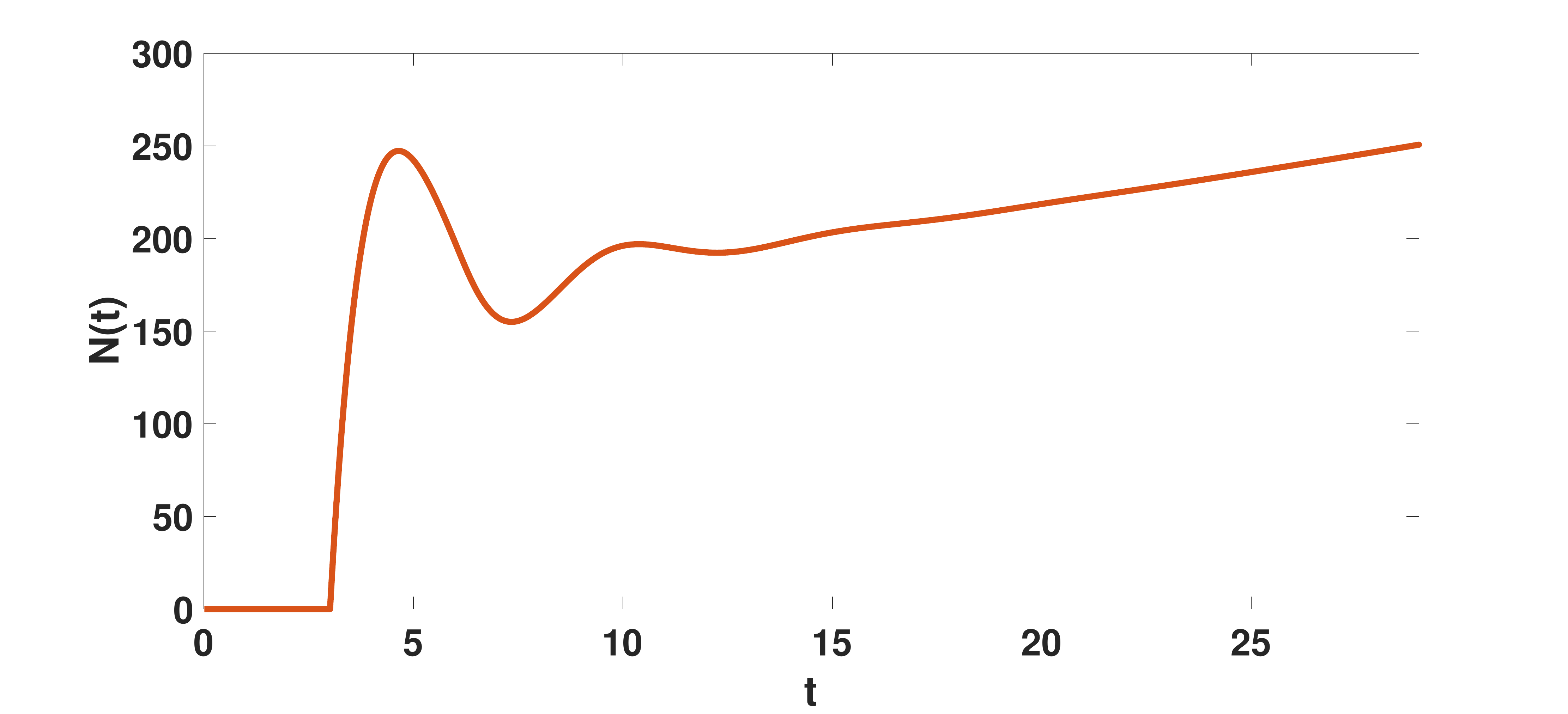}
		\includegraphics[scale=0.15]{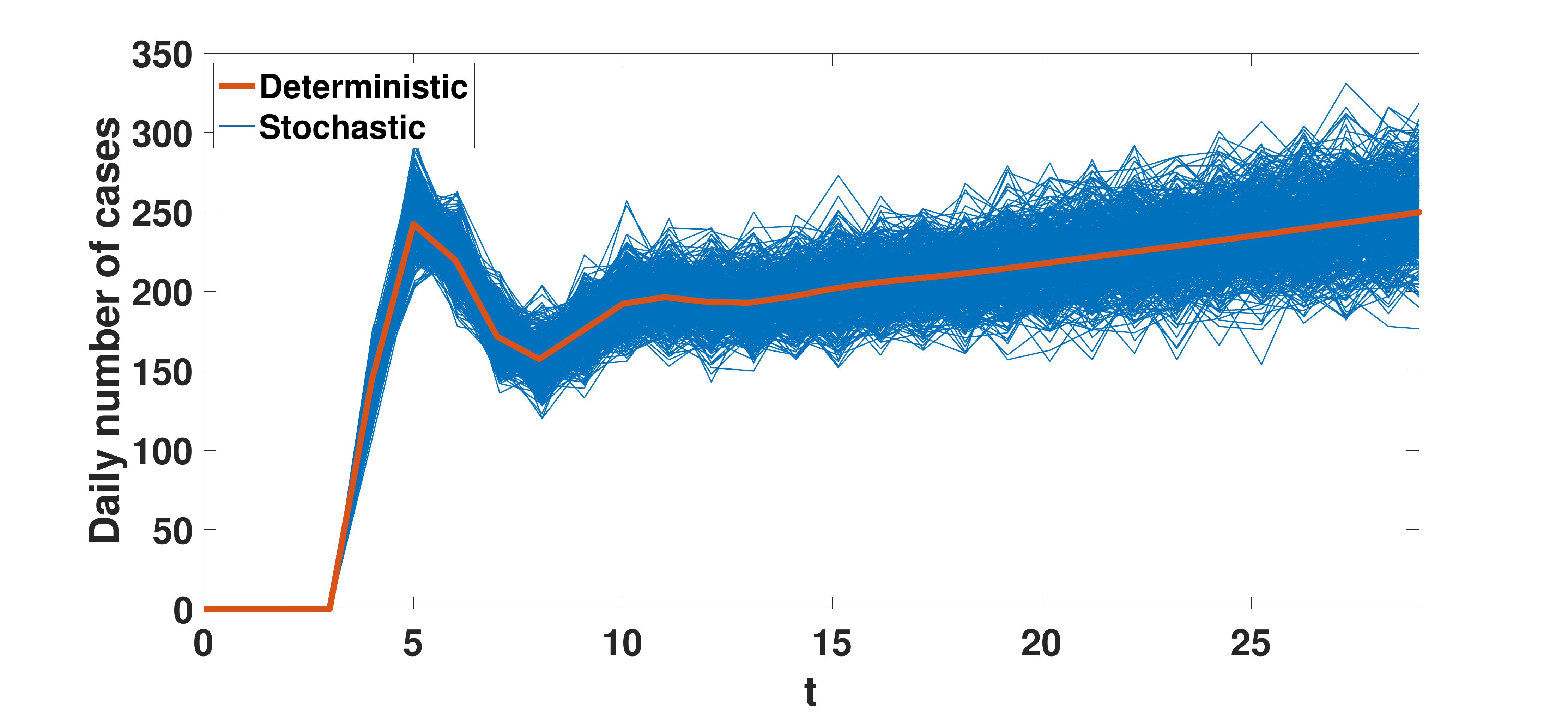}
	\end{center}
	\caption{\textit{On the left-hand side,  we plot the function $t \to N(t)$ solution of \eqref{4.3} with \eqref{2.4}. On the right-hand side, we plot the function $t \to \int_{t-1}^{t} N(s)ds$ (for $t=1,2, \ldots$) which corresponds to the daily number of cases obtained from by solving \eqref{4.3} with \eqref{2.4}, and we compare it with the daily number of cases obtained from $500$ runs of the IBM.}}\label{Fig22}
\end{figure}

\begin{figure}[H]
	\begin{center}
		\includegraphics[scale=0.15]{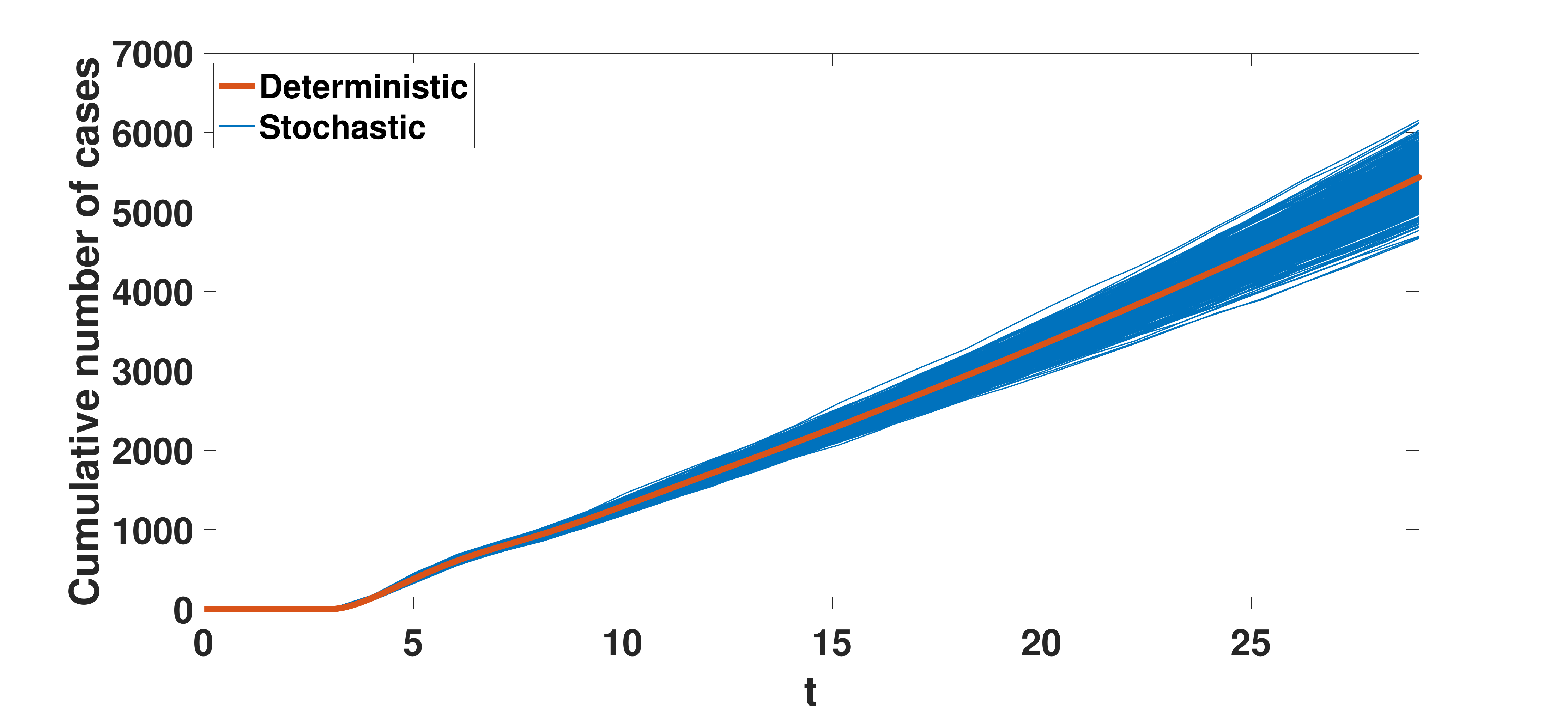}
		\includegraphics[scale=0.15]{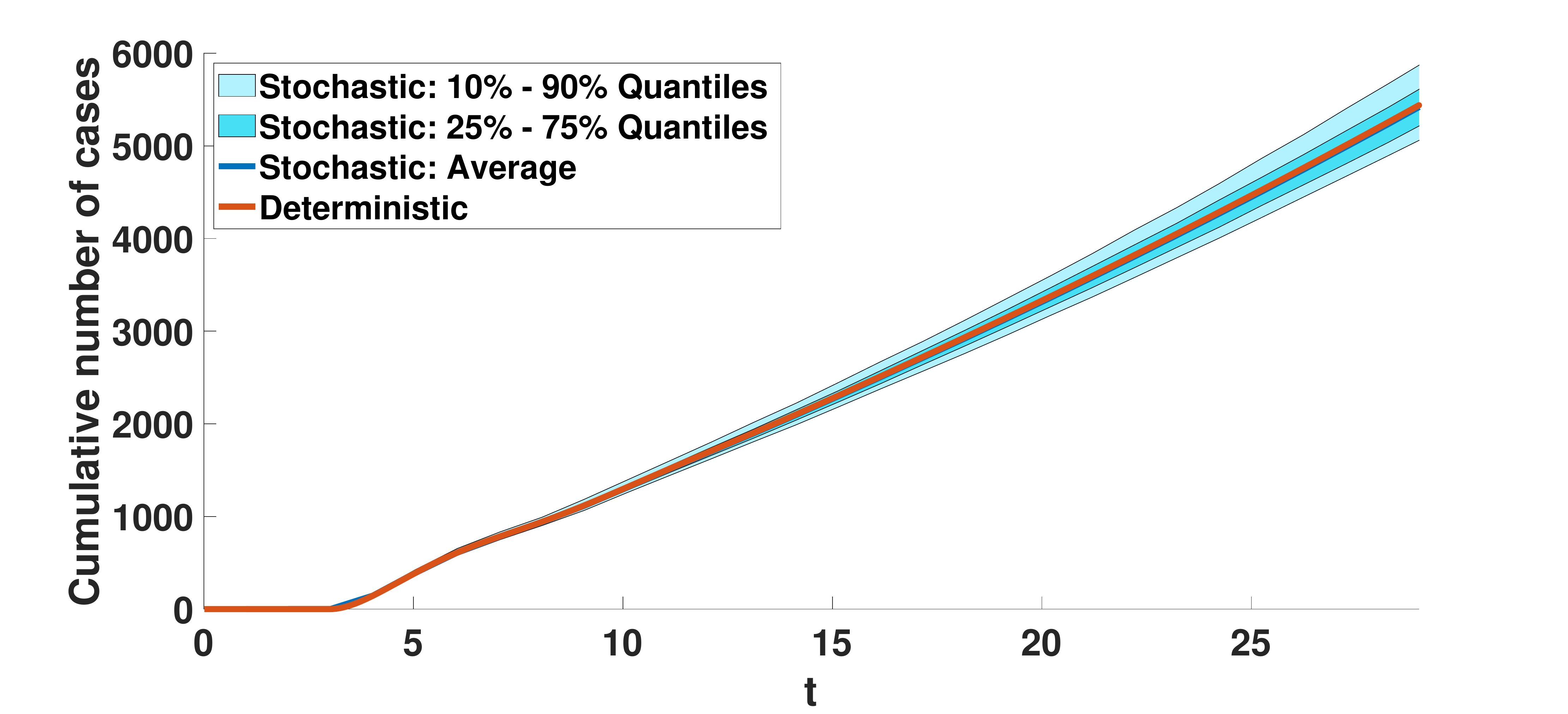}
	\end{center}
	\caption{\textit{On the left-hand side,  we plot the function $t \to \int_0^t N(s)ds$ (for $t=0,1,2, \ldots$) which corresponds to the cumulative number of cases obtained from by solving \eqref{4.3} with \eqref{2.4}, and we compare it with the cumulative number of cases obtained from $500$ runs of the IBM. On the right-hand side, we plot the average values of the $500$ runs obtained from the IBM as well as the quantiles ($10\%-90\%$ (light blue) and  $25\%-75\%$ (blue)).}}\label{Fig23}
\end{figure}

In Figures \ref{Fig24}-\ref{Fig26} we focus on the reconstruction of the daily reproduction number  $R_0(a)=\tau \, S_0 \,  e^{-  \nu \, a } \,  \beta\left(a\right)$. In Figure \ref{Fig24}, we focus on the reconstruction of the daily reproduction number  from deterministic simulations, while in Figures \ref{Fig25}-\ref{Fig26} we focus on the reconstruction of the daily reproduction number  from stochastic simulations. 
\begin{figure}[H]
	\begin{center}
		\includegraphics[scale=0.15]{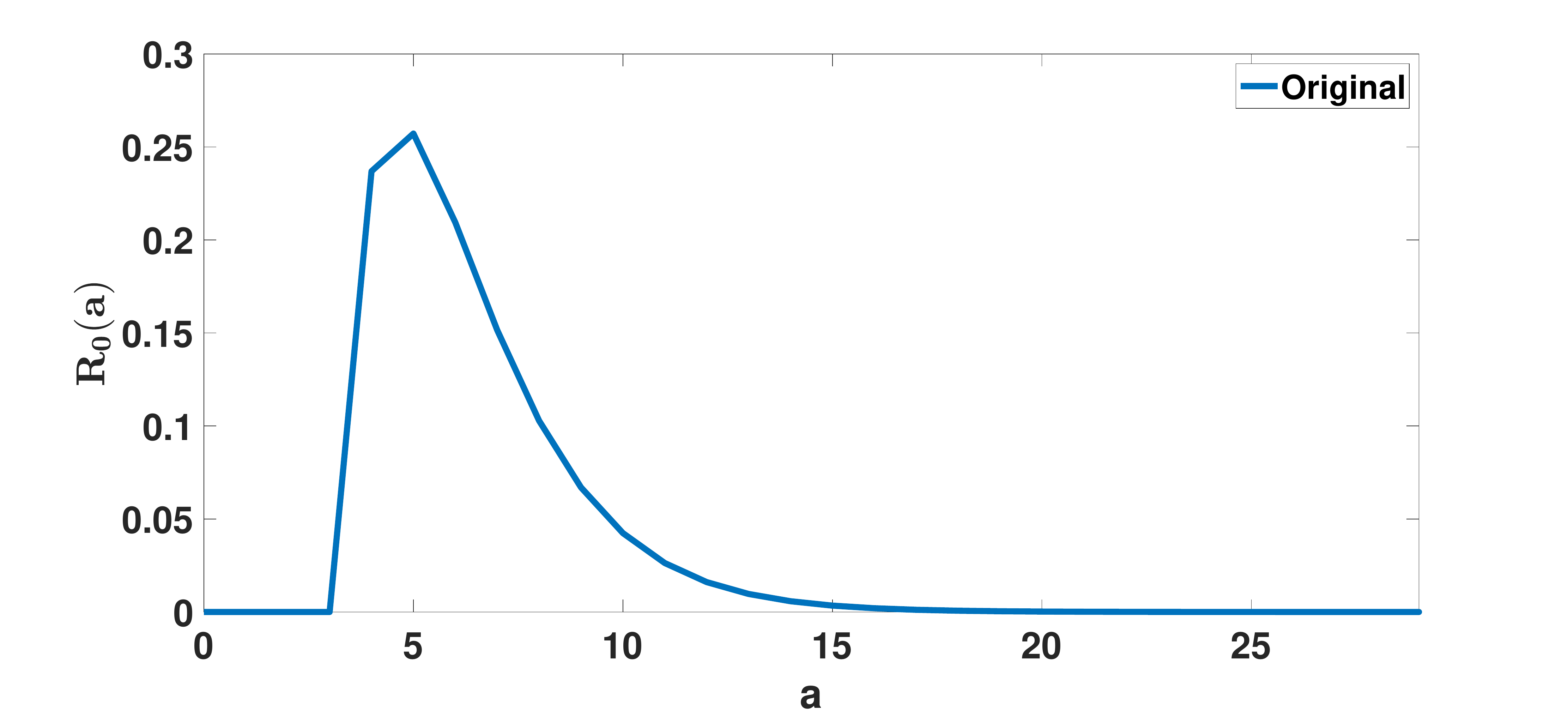}
		\includegraphics[scale=0.15]{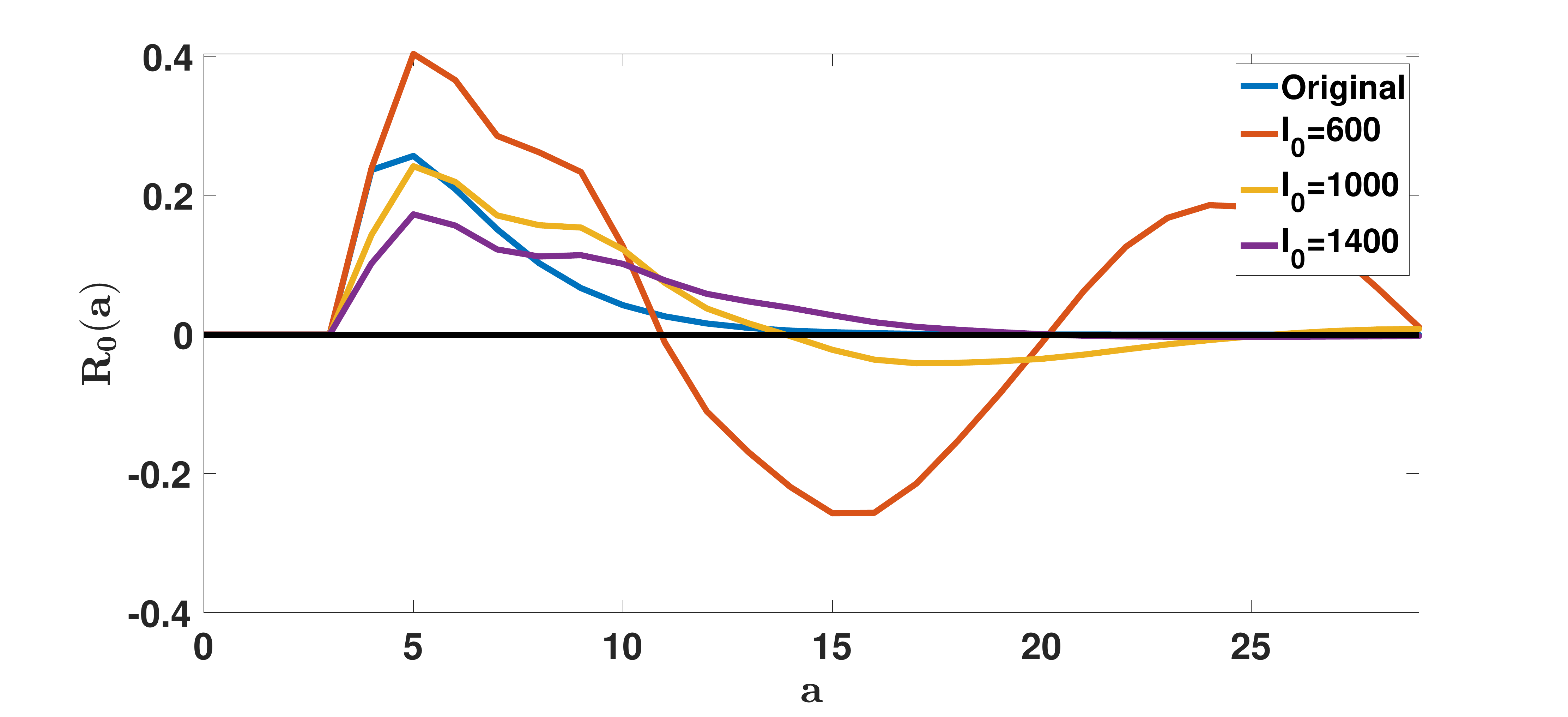}
	\end{center}
	\caption{\textit{On the left hand side, we plot the daily basic reproduction number by using the original formula $ R_0(a)$ with \eqref{7.2}.  On the right-hand side, we apply formula \eqref{6.5} to the flow of new infected obtained from the deterministic model. We vary $I_0=600, 1000, 1400$. The value $I_0=1000$ corresponds to the value used for the simulation of the deterministic model.   The yellow curve gives the best visual fit, and the $R_0(a)$ becomes negative whenever $I_0$ becomes too small. }}\label{Fig24}
\end{figure}

\begin{figure}[H]
	\begin{center}
		\includegraphics[scale=0.15]{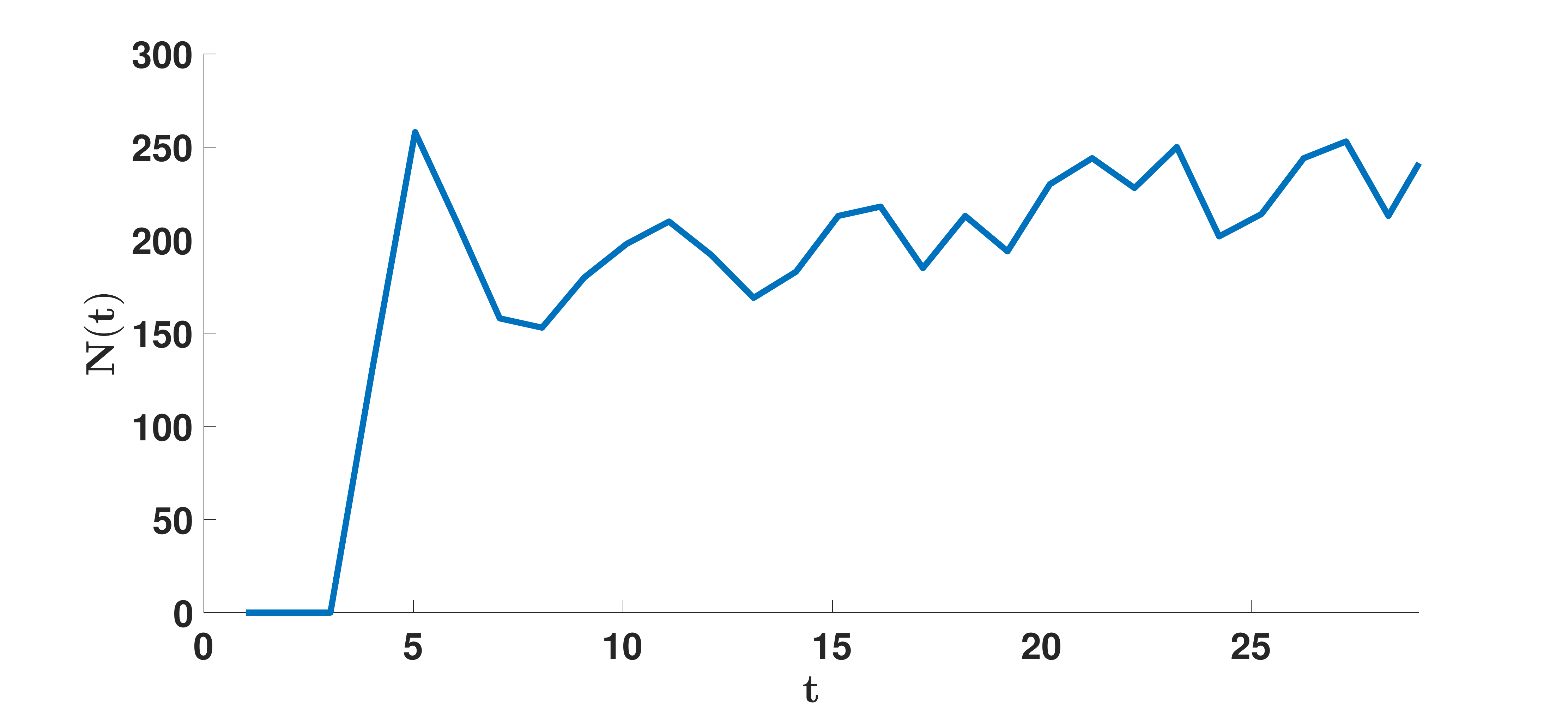}
		\includegraphics[scale=0.15]{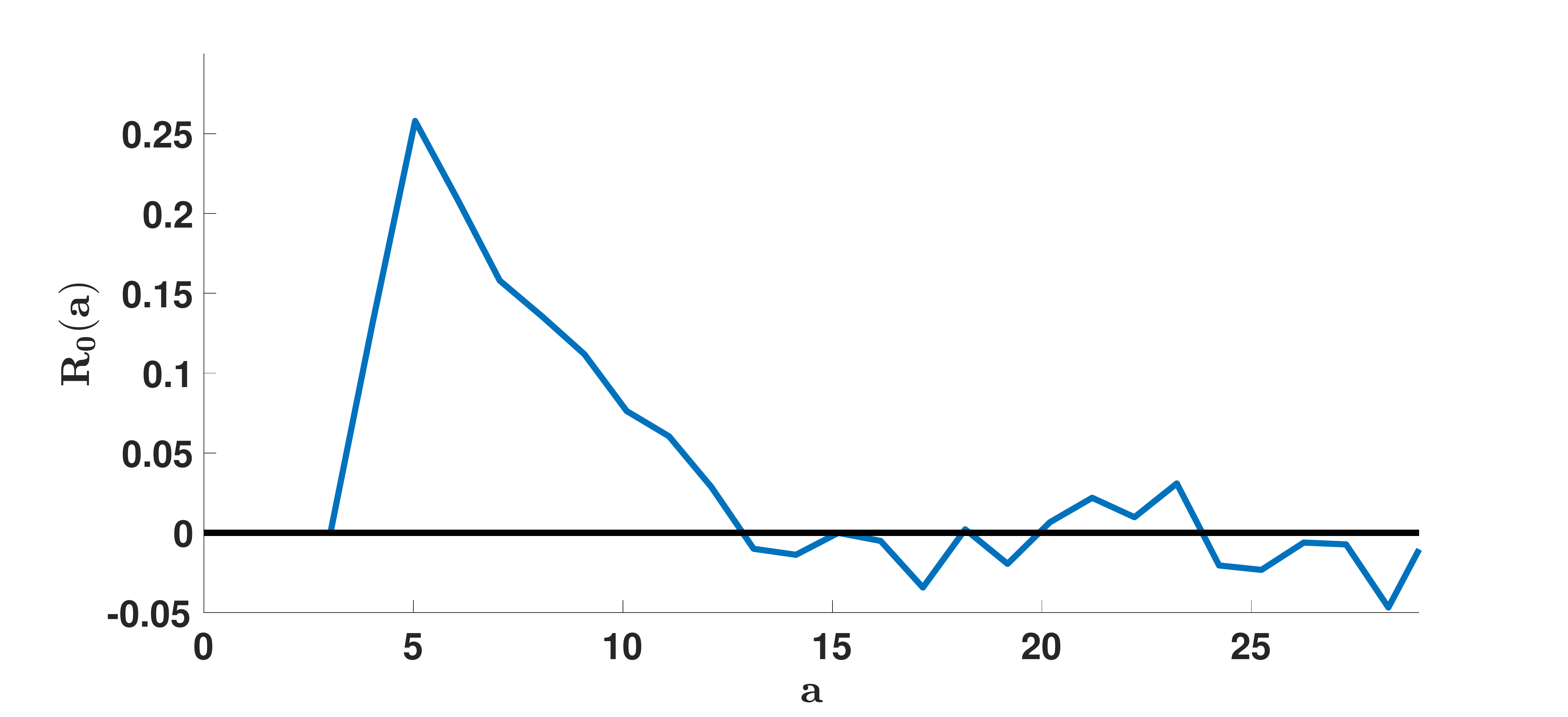}
	\end{center}
	\caption{\textit{On the left-hand side,  we plot the daily number of cases $t \to N(t) $ (for $t=0,1,2, \ldots$) obtained from a single run of the IBM.  On the right-hand side, we apply formula \eqref{6.5}  (with $I_0=1000$) to  the daily number of cases of case obtained from the IBM.   }}\label{Fig25}
\end{figure}

\begin{figure}[H]
	\begin{center}
		\includegraphics[scale=0.15]{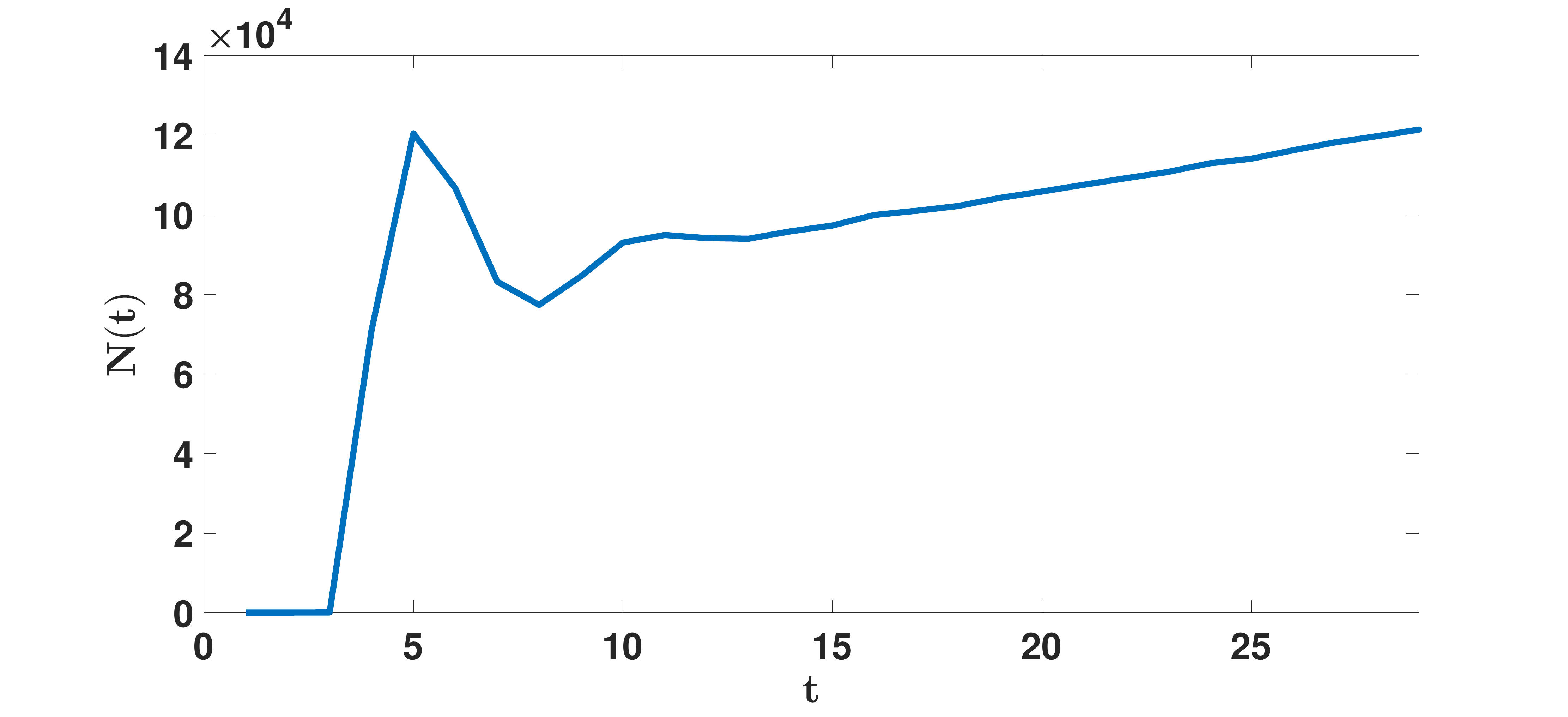}
		\includegraphics[scale=0.15]{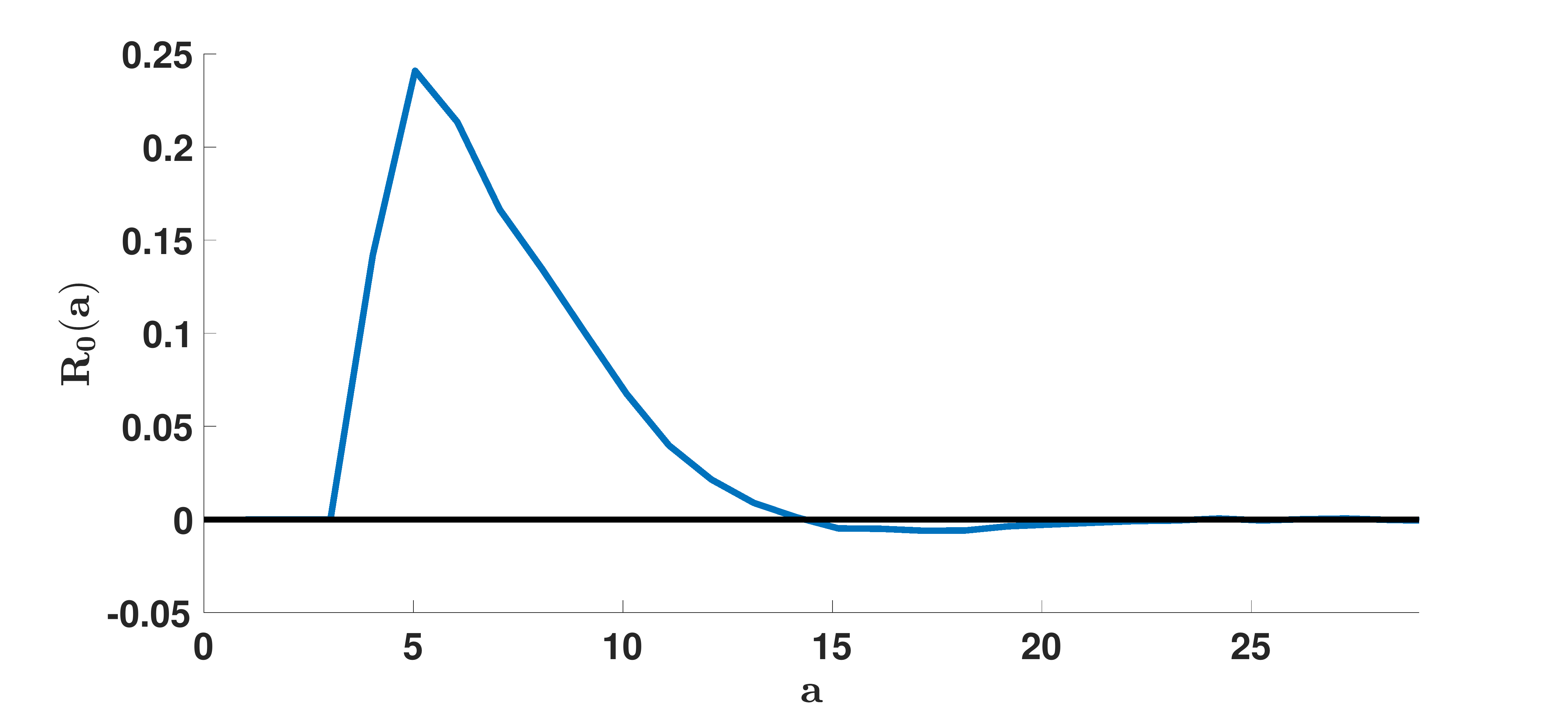}
	\end{center}
	\caption{\textit{On the left-hand side,  we plot the daily number of cases $t \to N(t)$ (for $t=0,1,2, \ldots$) obtained by summing the daily number of cases for $500$ runs of the IBM.  On the right-hand side, we apply formula \eqref{6.5}  (with $I_0=500 \times 1000$) to the daily number of cases obtained from the IBM. }}\label{Fig26}
\end{figure}
\subsubsection{Example 2}
It is common to see biphasic flu clinically: after incubation of
	one day, there is a high fever, then a drop in temperature before rising again,
	hence the term "V" fever \cite{Chao}. Such a biphasic contagiousness is also observed in Covid-19. The viral load in throat swab and sputum has been measured for  Covid-19 patients, which leads to biphasic contagiousness \cite{DORSTW, Pan}. To cover these type of infectious diseases, we introduce the following form for the probability to be infectious 
\begin{equation}\label{7.3}
\beta(a)=  0.5\times 4q \left\{(a-a_0)  \left(1-q (a-a_0) \right) \right\}^+ + 4q  \left\{ (a-pa_0) \left(1-q(a-pa_0) \right)  \right\}^+ ,  
\end{equation}
with 
$$
a_0=3\text{ days},  p=2.5, \text{ and } q=0.3. 
$$
  \begin{figure}[H]
	\begin{center}
		\includegraphics[scale=0.15]{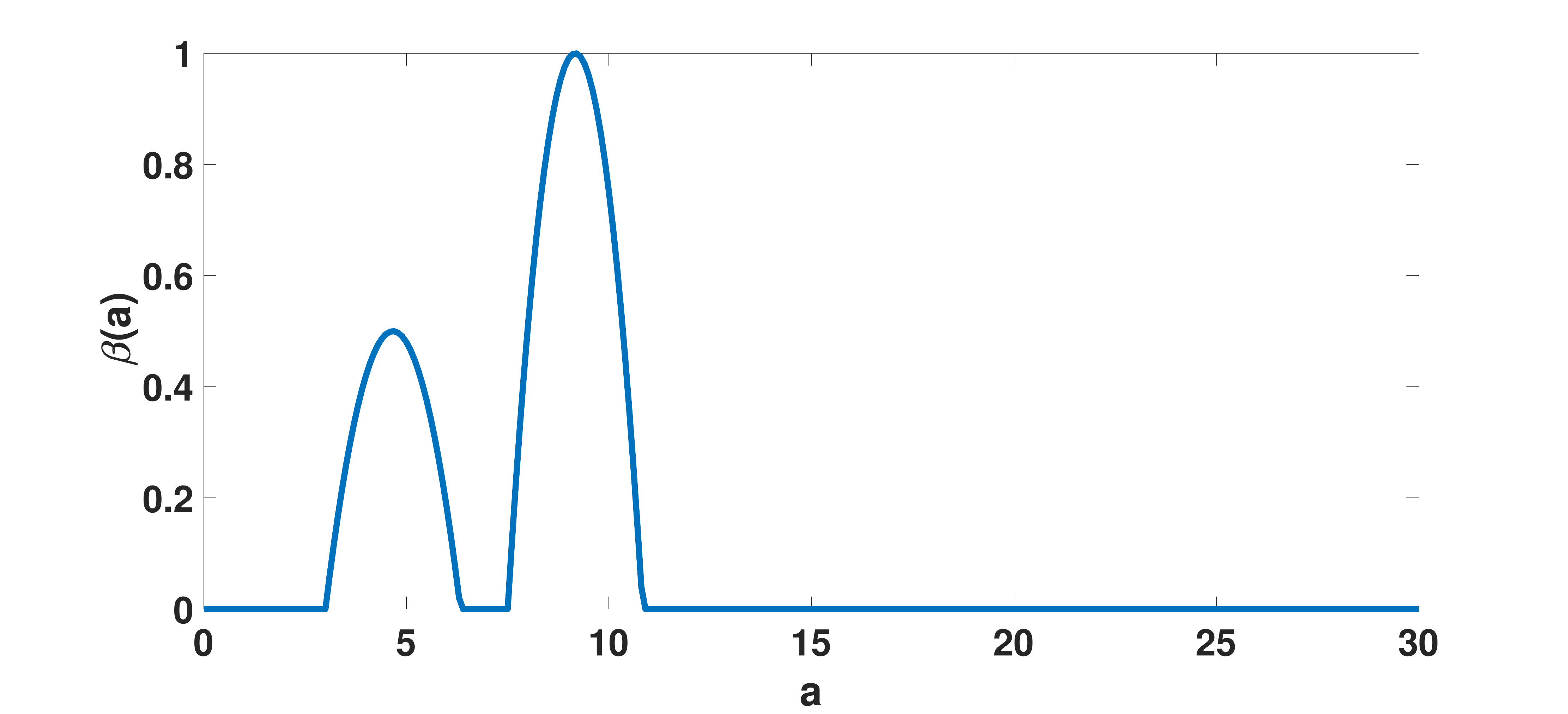}
		\includegraphics[scale=0.15]{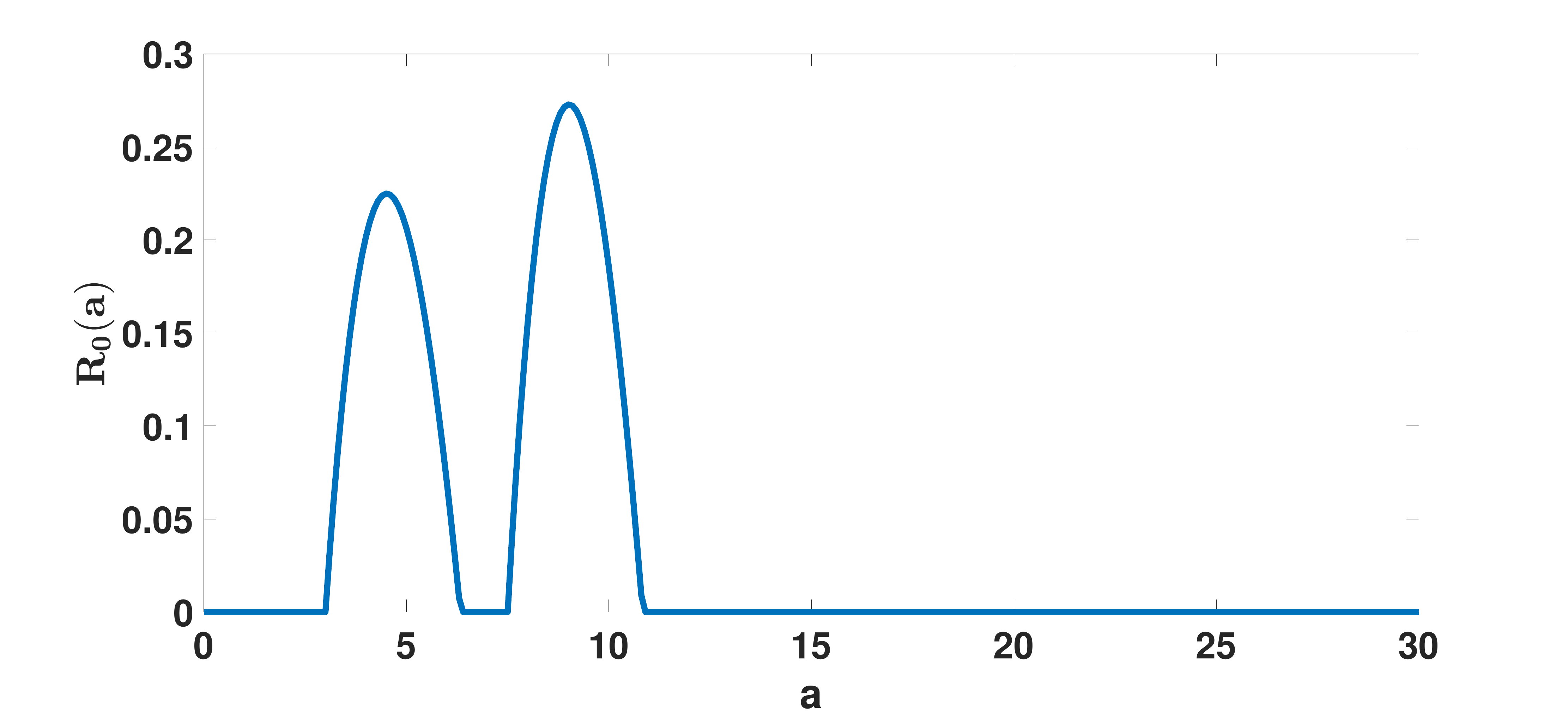}
	\end{center}
	\caption{\textit{On the left-hand side,  we plot the function $a \to \beta(a)$. On the right-hand side, we plot the function $a \to R_0(a)=\tau_0 \times S_0  \times \beta(a) \times  e^{-\nu a}$.}}\label{Fig27}
\end{figure}
In the Figures \ref{Fig28}-\ref{Fig29} we use the IBM  to investigate some properties of the clusters obtained from the stochastic simulations. We compare such a stochastic sample with the original  $a \to R_0(a)$. 
	
	By comparing Figures \ref{Fig30}-\ref{Fig31} (for $I_0=10$) and   Figures \ref{Fig35}-\ref{Fig36} (for $I_0=1\, 000$), we observe the convergence of the IBM to the deterministic model.  
	
	Then in Figures \ref{Fig32}-\ref{Fig34} (for $I_0=10$) and  Figures \ref{Fig37}-\ref{Fig39} (for $I_0=1\,000$), we apply the discrete-time equation \eqref{6.5} to reconstruct $a \to R_0(a)$ from the trajectories of the deterministic or stochastic models. In the deterministic model, we observe the effect of the day-by-day discretization (which corresponds to the daily reported data). In the stochastic case, we observe the effect of the stochasticity of the IBM. \medskip

\medskip 
\noindent \textbf{First generation of secondary cases produced by a single infected:}
\begin{figure}[H]
	\centering
	\begin{minipage}{0.48\textwidth}
		\centering
		\textbf{(a)} \\
			\includegraphics[scale=0.15]{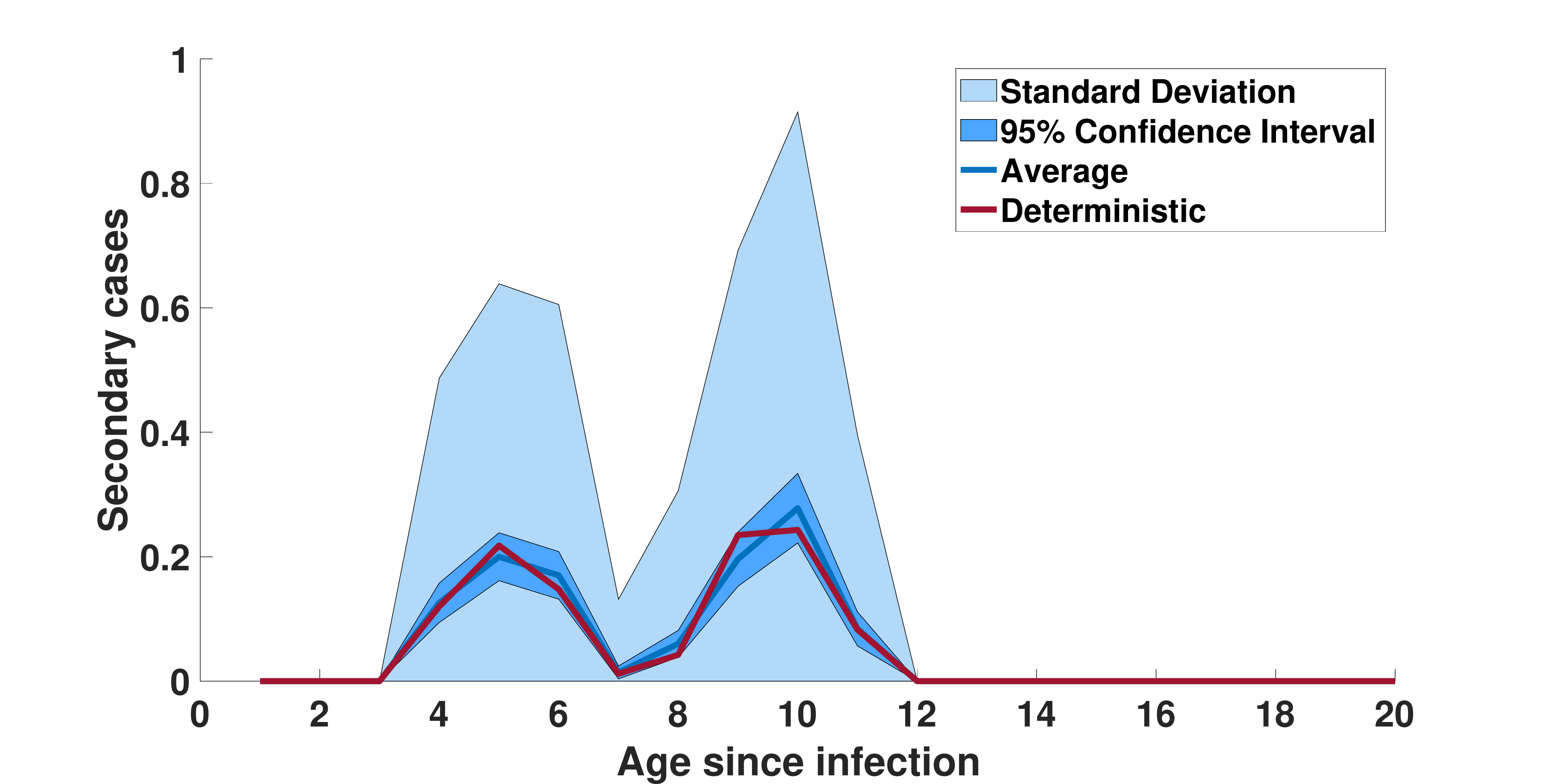}
	\end{minipage} 
	\begin{minipage}{0.48\textwidth}
		\centering
		\textbf{(b)} \\
		\includegraphics[scale=0.15]{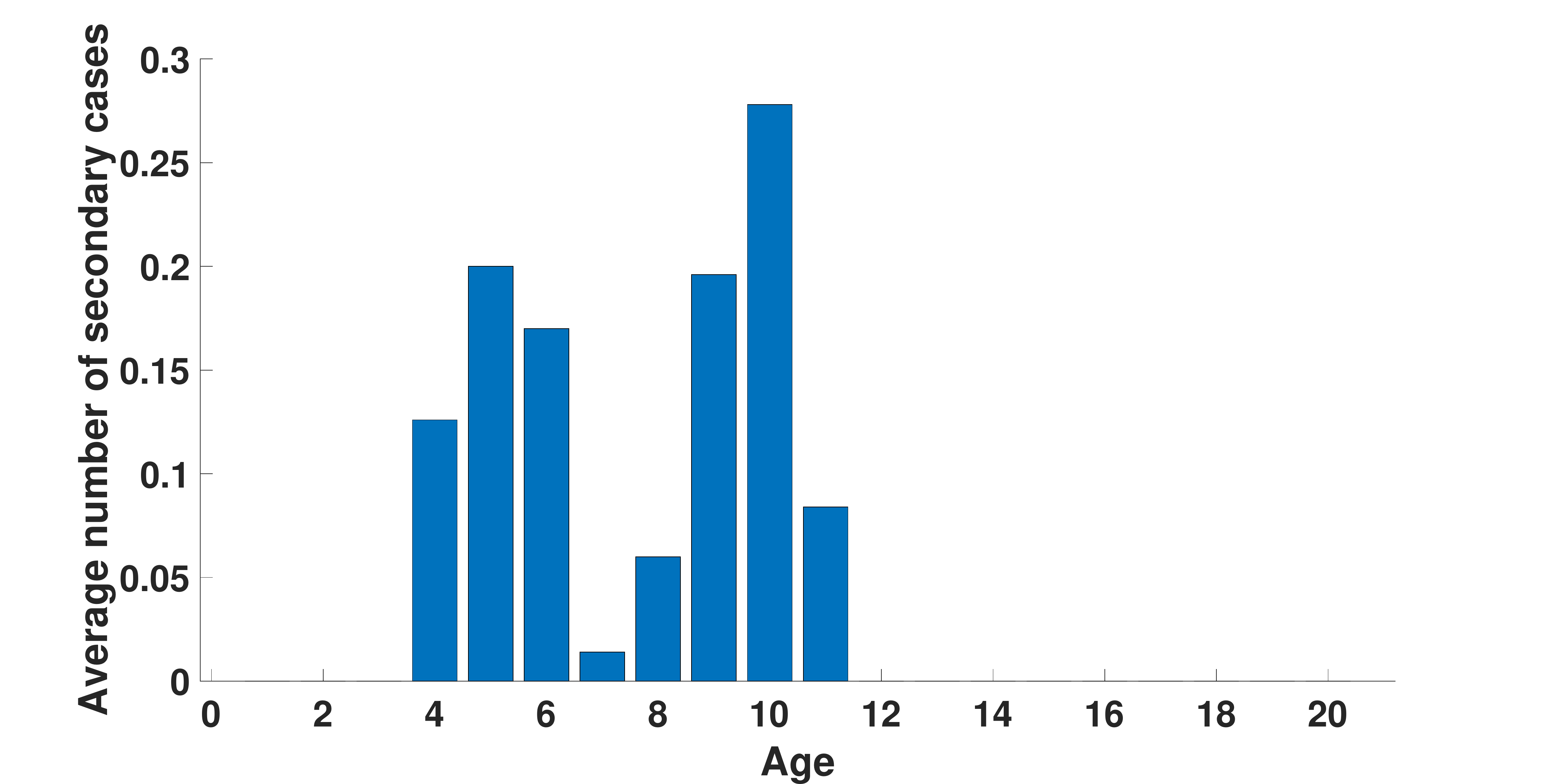}
	\end{minipage} 
	\caption{\textit{In these figures, we present  sets of 500 samples of secondary cases produced by a single infected individual in a population of $S=10^7$ susceptible hosts. Theses samples are produced by using the IBM. \textbf{(a)} Statistical summary: the blue curve represents the average number of cases at age of infection $a$; the dark blue area is the 95\% confidence interval of this average obtained by fitting a Gaussian distribution to the data; the light blue area corresponds to the standard deviation; the orange curve is the deterministic daily basic reproductive number at age $a$. \textbf{(b)} Bar graph of the average number of secondary cases as a function of the age since infection.}}\label{Fig28} 
\end{figure}
	 \begin{figure}[H]
		\begin{center}
			\includegraphics[scale=0.15]{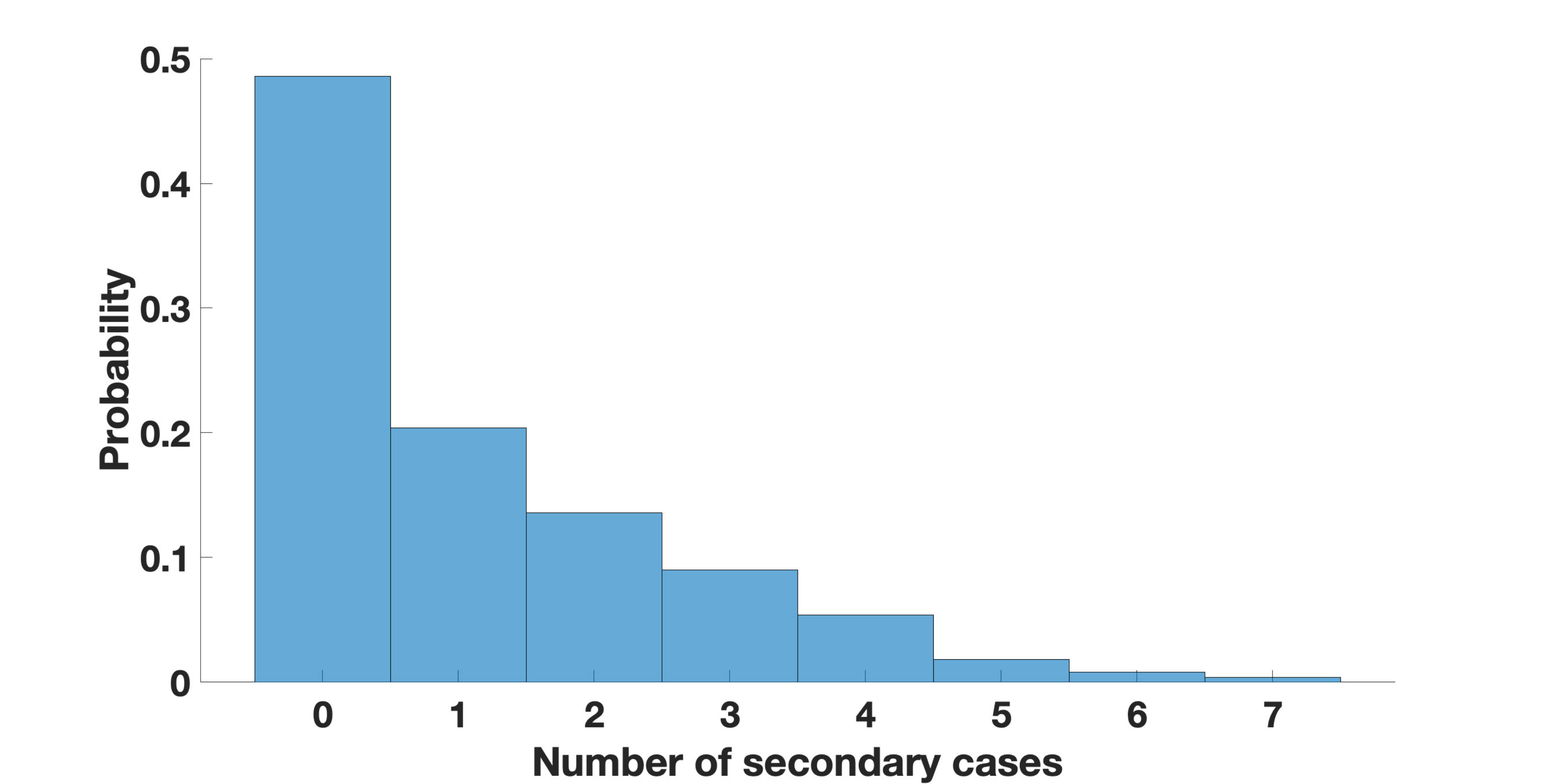}
		\end{center}
	\caption{\textit{In these figures, we present  sets of 500 samples of secondary cases produced by a single infected individual in a population of $S=10^7$ susceptible hosts. Theses samples are produced by using the IBM. We plot a histogram of the total number of secondary cases produced during the whole infection. This estimates the probability of a single infected to generate $n$ secondary cases (with $n$ in the abscissa).}} \label{Fig29}
\end{figure}
\noindent \textbf{Simulations for $I_0=10$:}
 \begin{figure}[H]
	\begin{center}
		\includegraphics[scale=0.15]{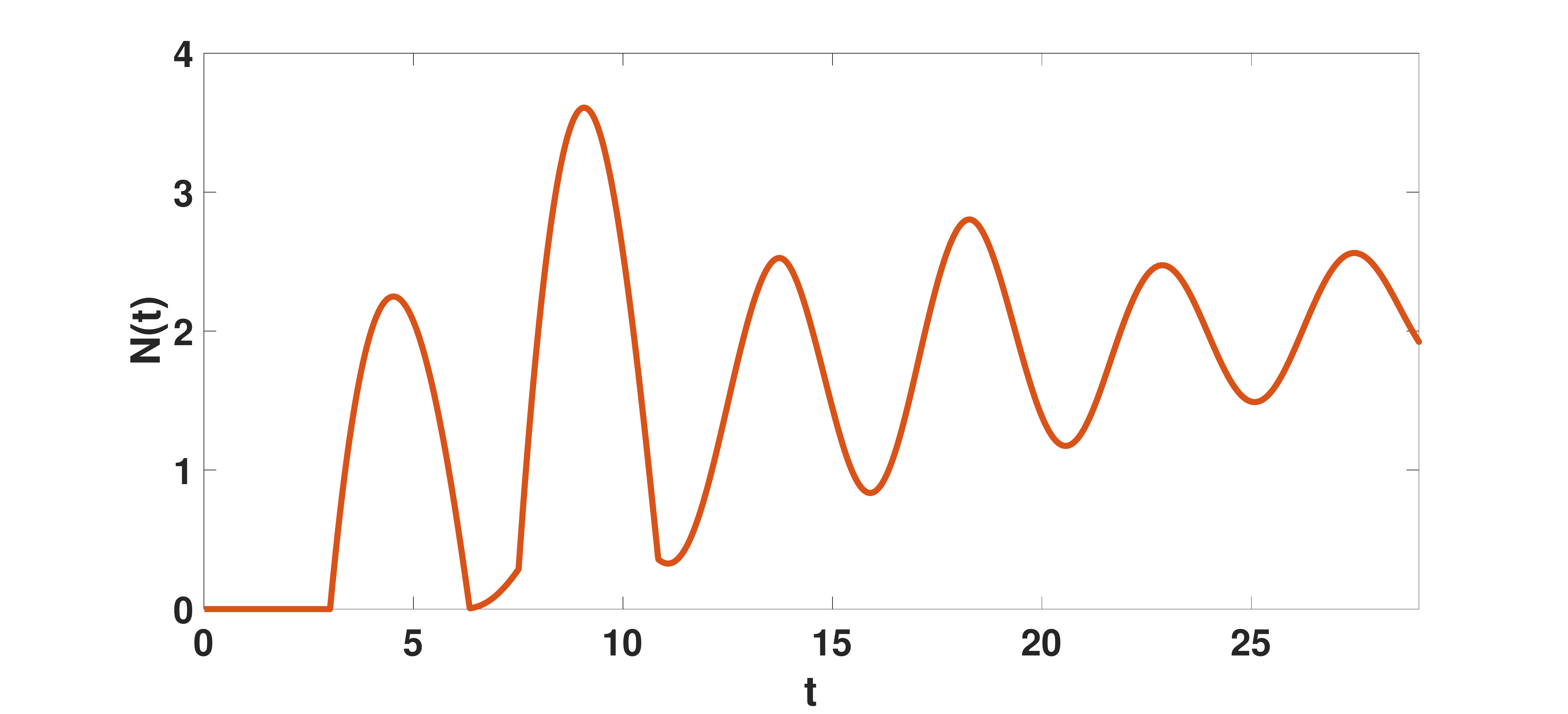}
		\includegraphics[scale=0.15]{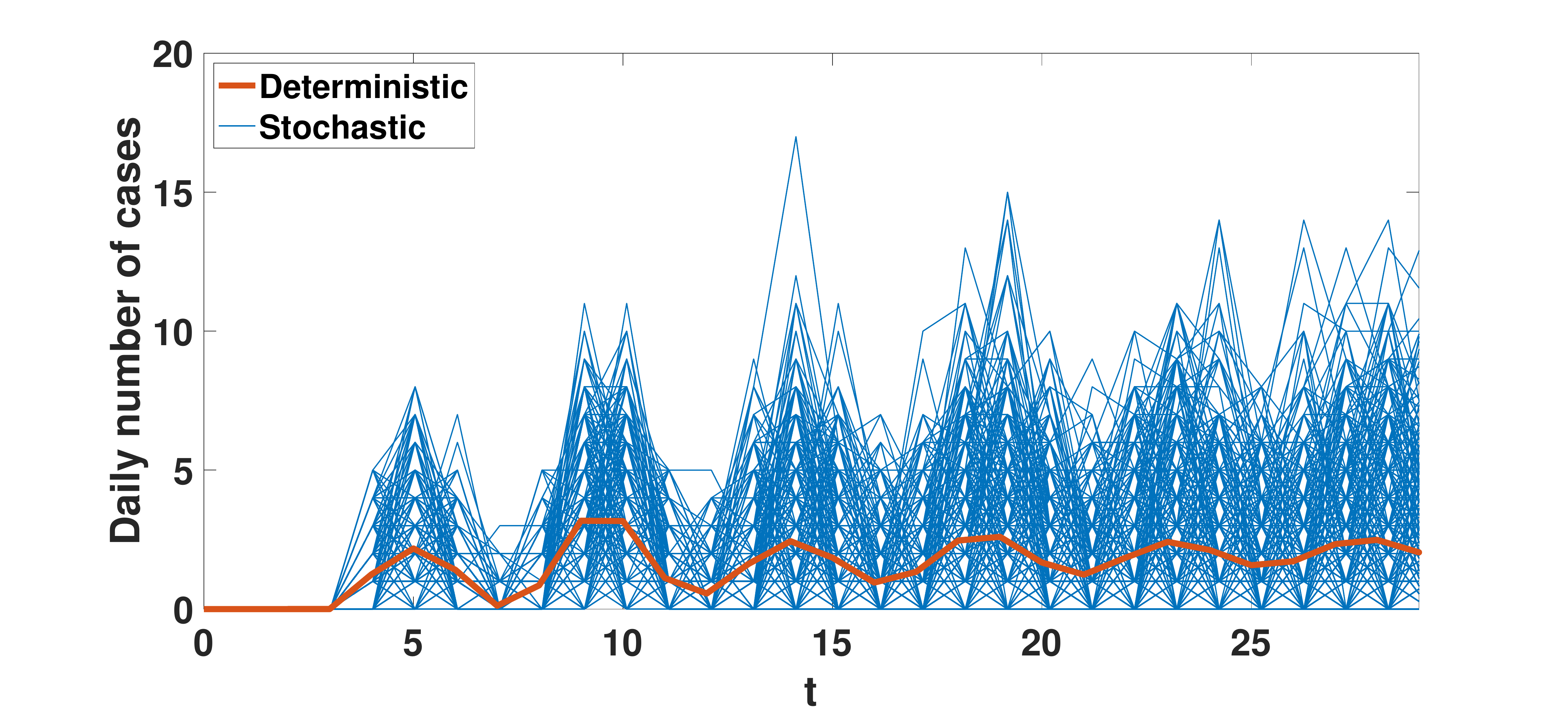}
	\end{center}
	\caption{\textit{On the left-hand side,  we plot the function $t \to N(t)$ solution of \eqref{4.3} with \eqref{2.4}. On the right-hand side, we plot the function $t \to \int_{t-1}^{t} N(s)ds$ (for $t=1,2, \ldots$) which corresponds to the daily number of cases obtained from by solving \eqref{4.3} with \eqref{2.4}, and we compare it with the daily number of cases obtained from $500$ runs of the IBM.}}\label{Fig30}
\end{figure}

\begin{figure}[H]
	\begin{center}
		\includegraphics[scale=0.15]{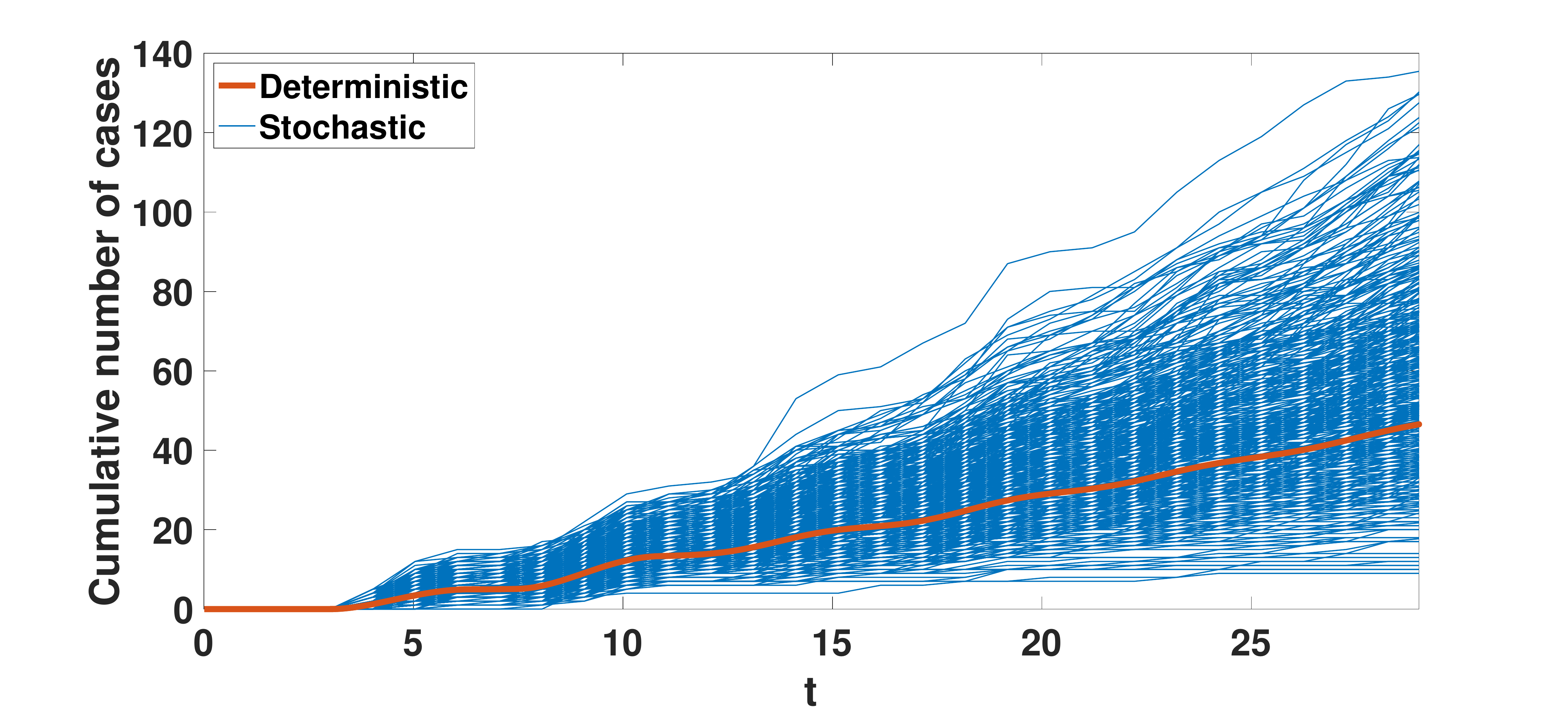}
		\includegraphics[scale=0.15]{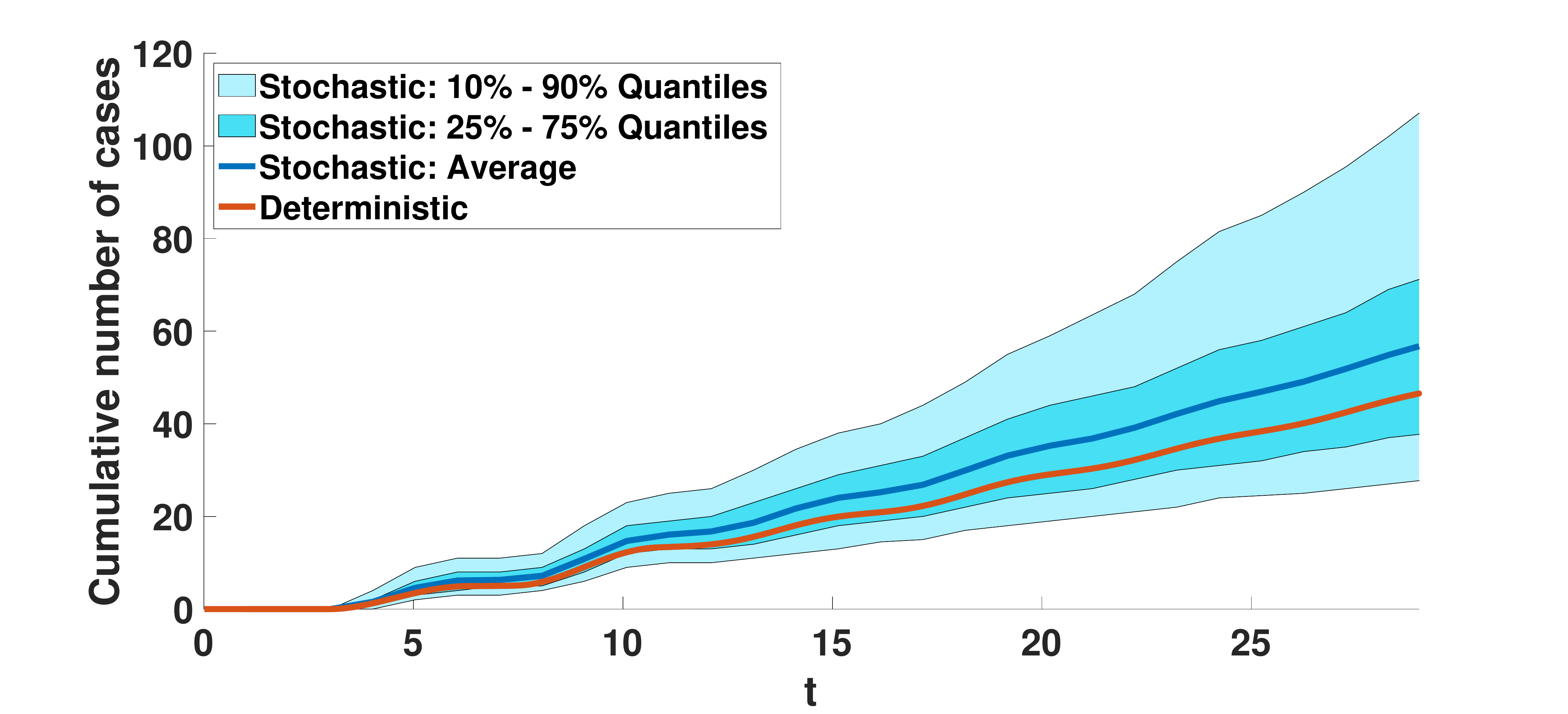}
	\end{center}
	\caption{\textit{On the left-hand side,  we plot the function $t \to \int_0^t N(s)ds$ (for $t=0,1,2, \ldots$) which corresponds to the cumulative number of cases obtained from by solving \eqref{4.3} with \eqref{2.4}, and we compare it with the cumulative number of cases obtained from $500$ runs of the IBM. On the right-hand side, we plot the average values of the $500$ runs obtained from the IBM as well as the quantiles ($10\%-90\%$ (light blue) and  $25\%-75\%$ (blue)).}}\label{Fig31}
\end{figure}
In Figures \ref{Fig32}-\ref{Fig34} we focus on the reconstruction of the daily reproduction number  $R_0(a)=\tau \, S_0 \,  e^{-  \nu \, a } \,  \beta\left(a\right)$. In Figure \ref{Fig32}, we focus on the reconstruction of the daily reproduction number  from deterministic simulations, while in Figures \ref{Fig33}-\ref{Fig34} we focus on the reconstruction of the daily reproduction number  from stochastic simulations. 
\begin{figure}[H]
	\begin{center}
		\includegraphics[scale=0.15]{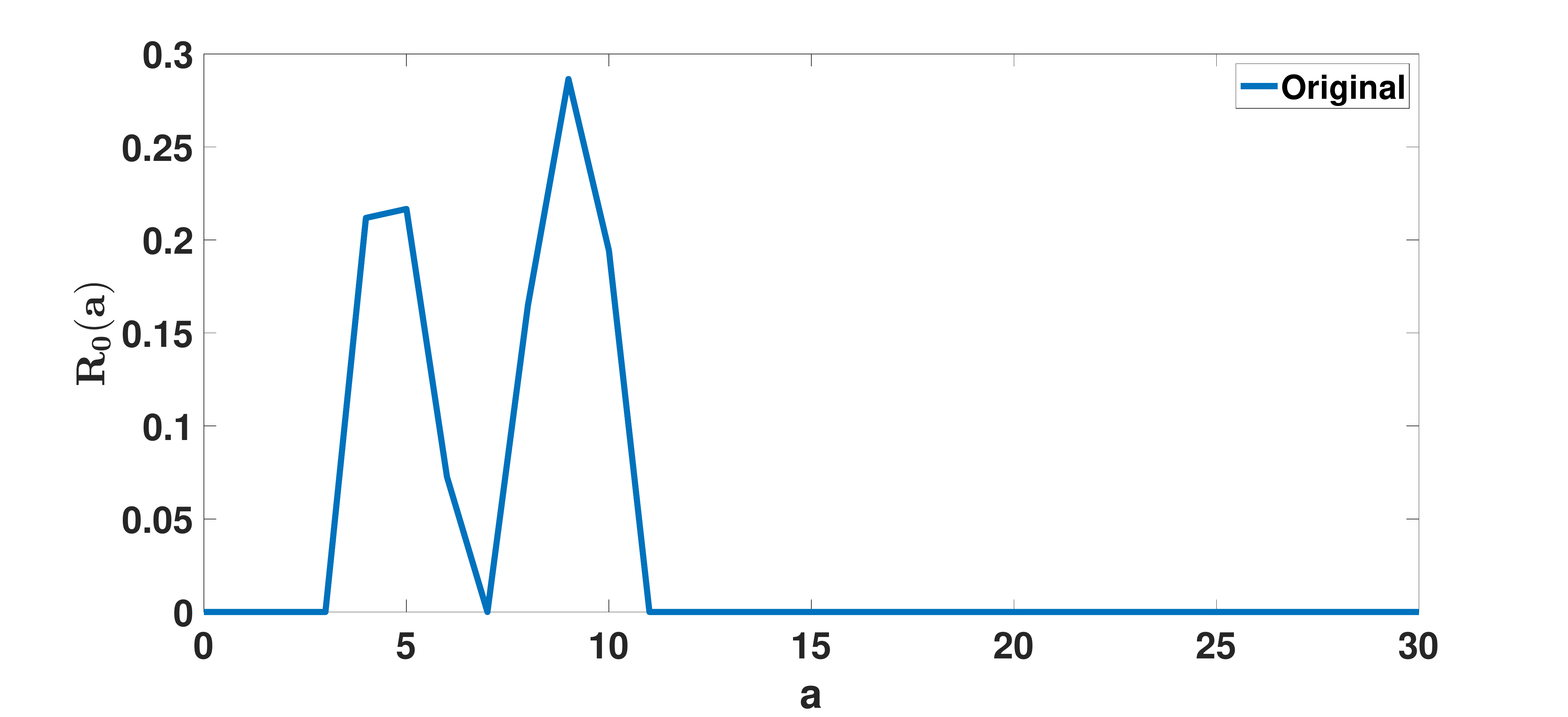}
		\includegraphics[scale=0.15]{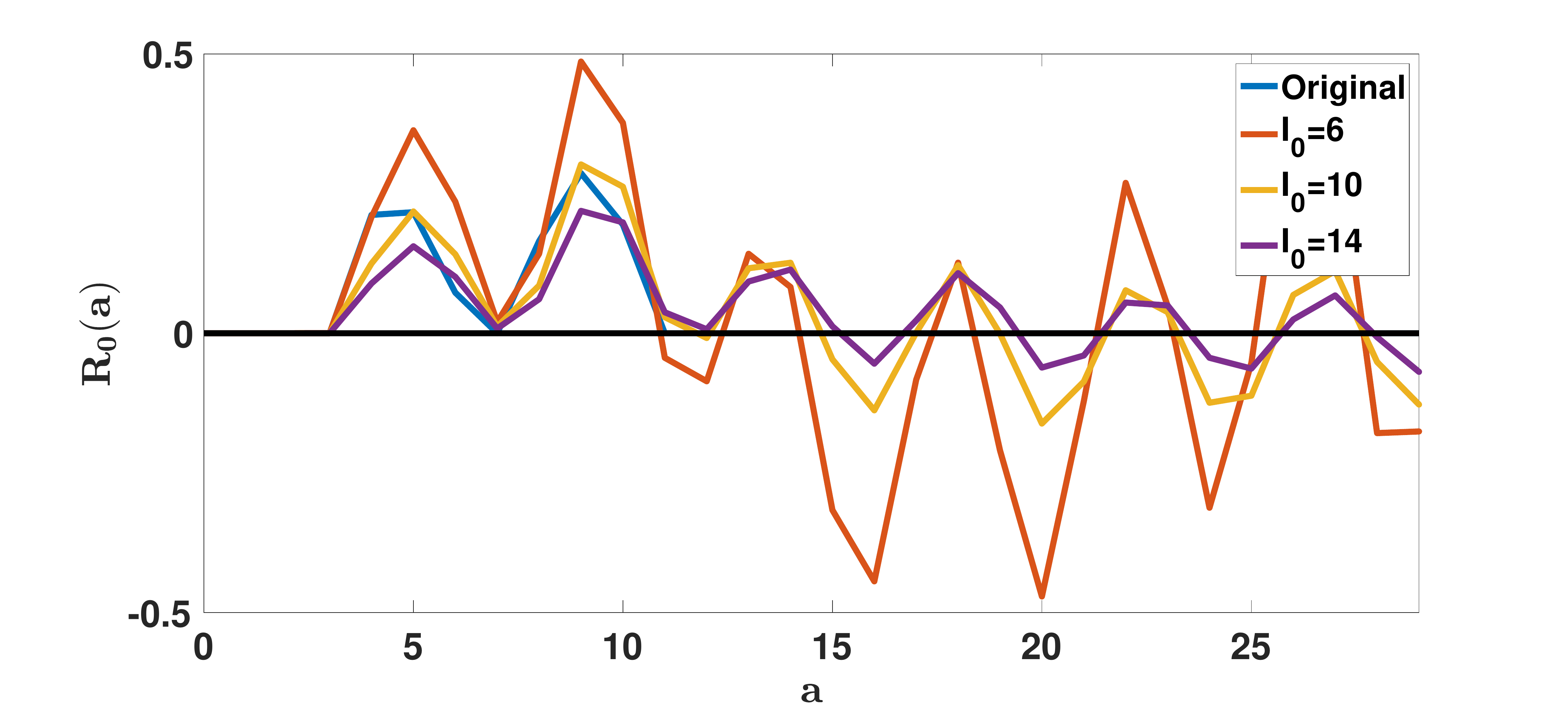}
	\end{center}
	\caption{\textit{On the left hand side, we plot the daily basic reproduction number by using the original formula $ R_0(a)$ with \eqref{7.3}.  On the right-hand side, we apply formula \eqref{6.5} to the flow of new infected obtained from the deterministic model. We vary $I_0=6, 10, 14$. The value $I_0=10$ corresponds to the value used for the simulation of the deterministic model.   The yellow curve gives the best visual fit, and the $R_0(a)$ becomes negative whenever $I_0$ becomes too small. }}\label{Fig32}
\end{figure}

\begin{figure}[H]
	\begin{center}
		\includegraphics[scale=0.15]{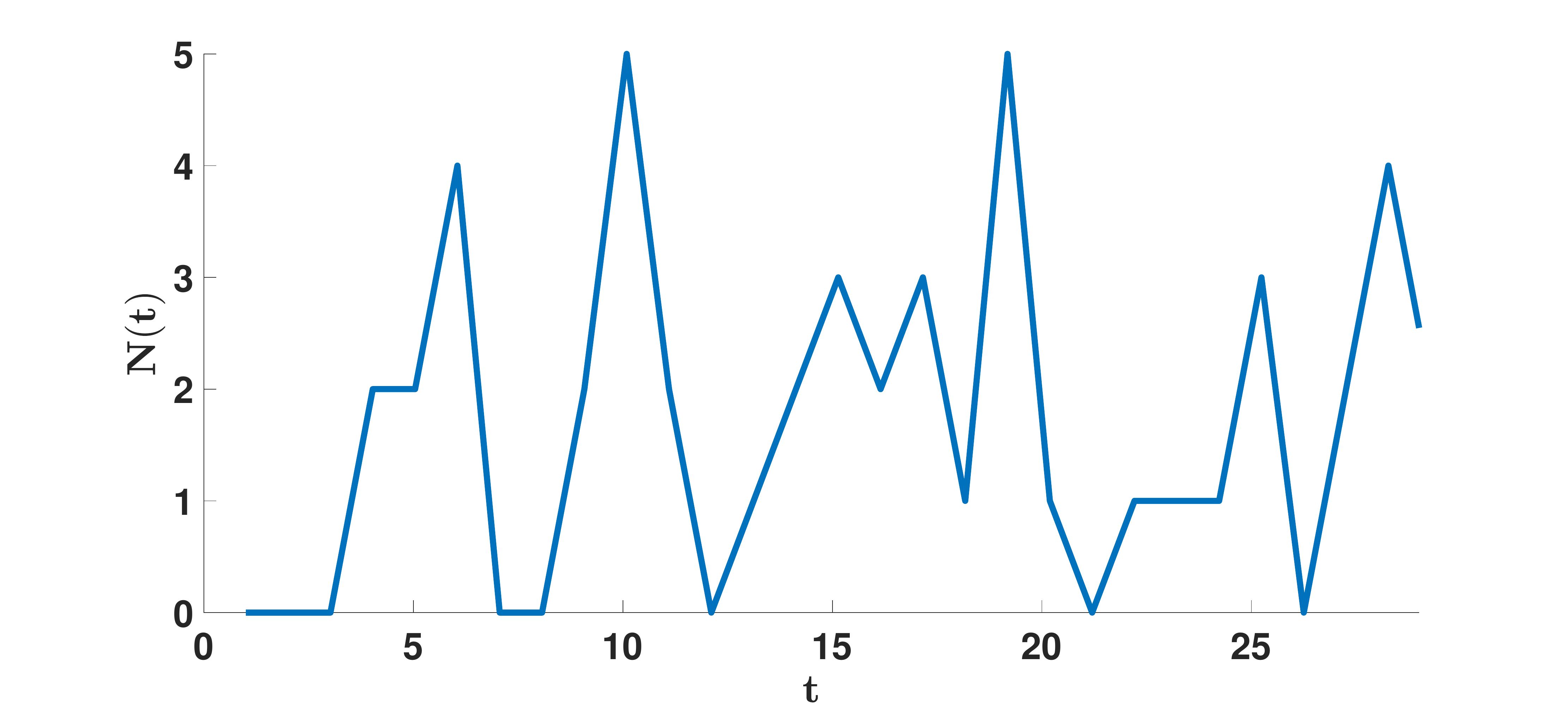}
		\includegraphics[scale=0.15]{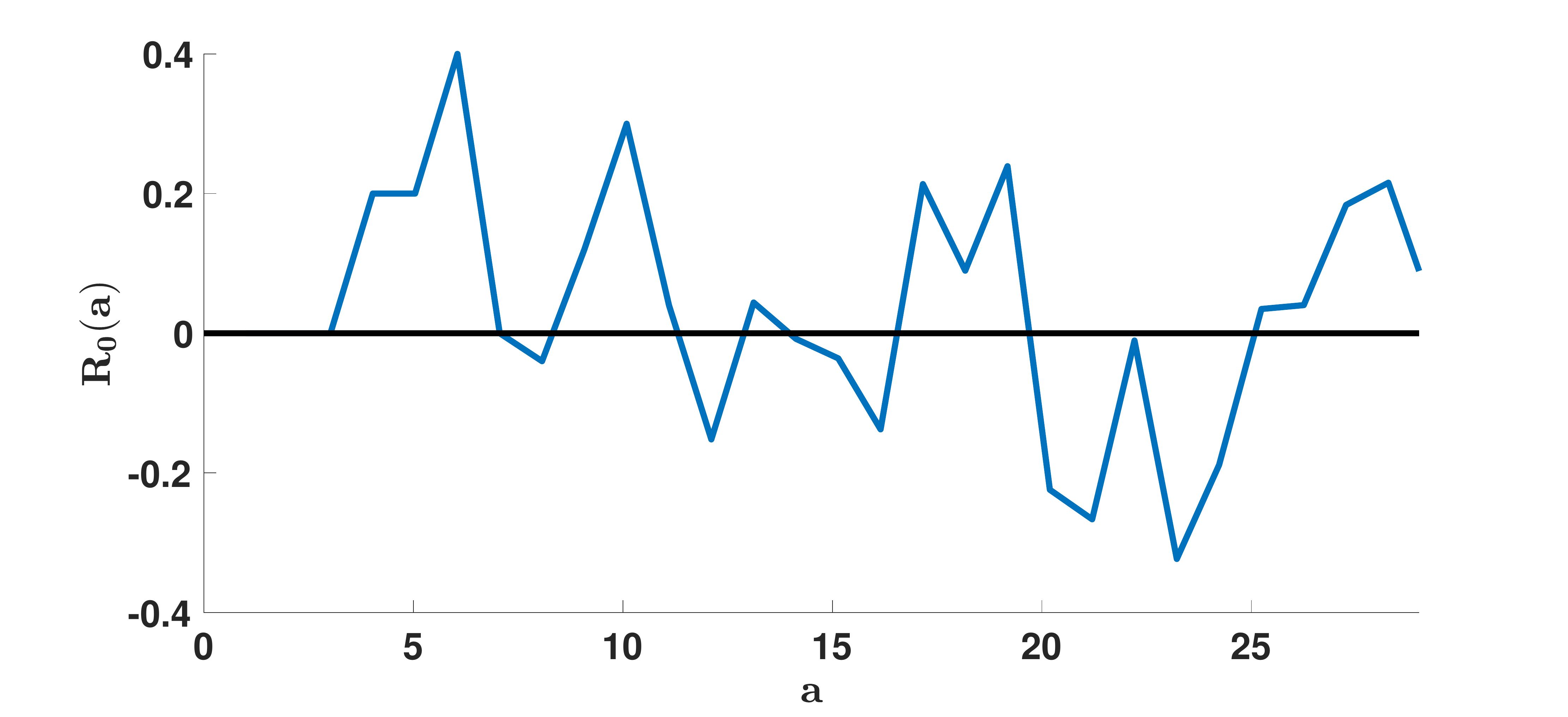}
	\end{center}
	\caption{\textit{On the left-hand side,  we plot the daily number of cases $t \to N(t) $ (for $t=0,1,2, \ldots$) obtained from a single run of the IBM.  On the right-hand side, we apply formula \eqref{6.5}  (with $I_0=10$) to  the daily number of cases of case obtained from the IBM.   }}\label{Fig33}
\end{figure}

\begin{figure}[H]
	\begin{center}
		\includegraphics[scale=0.15]{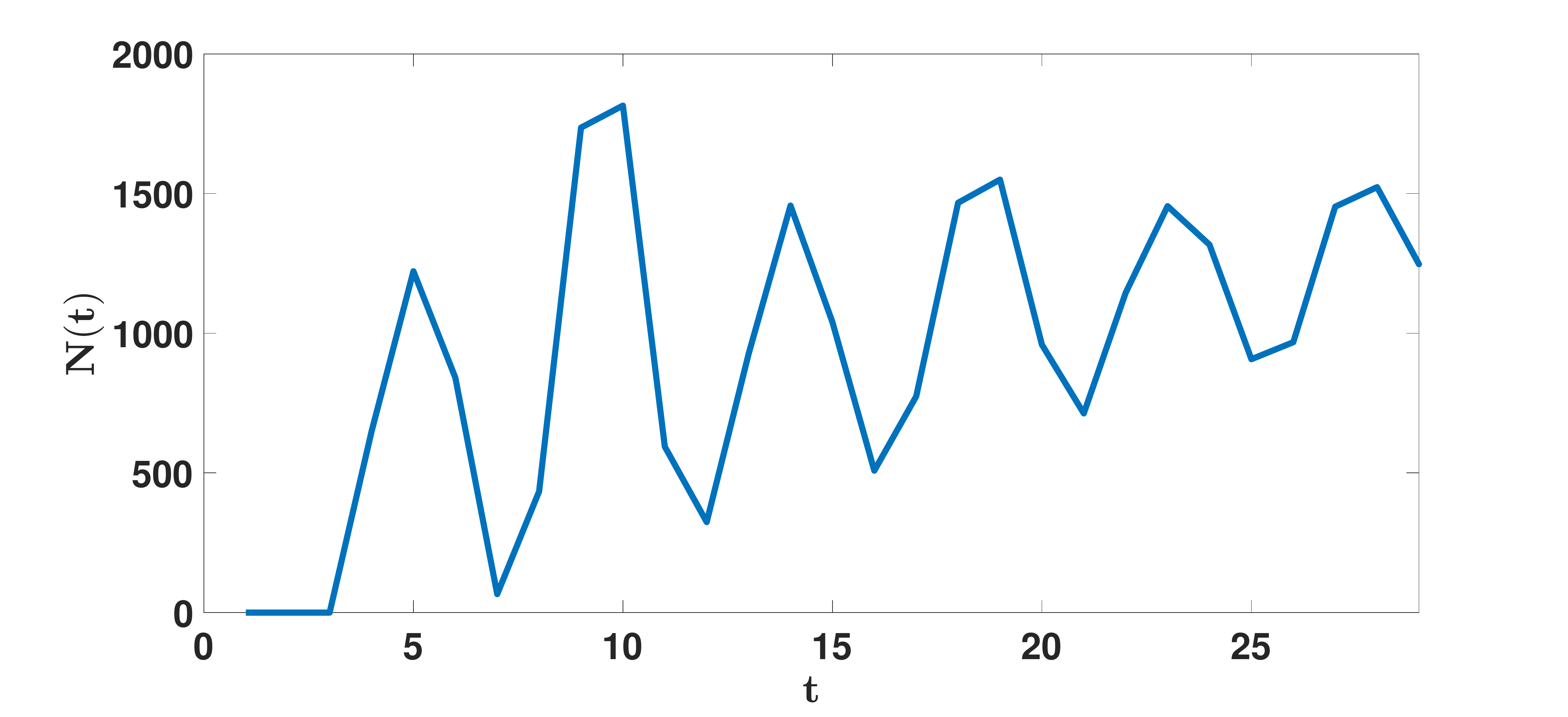}
		\includegraphics[scale=0.15]{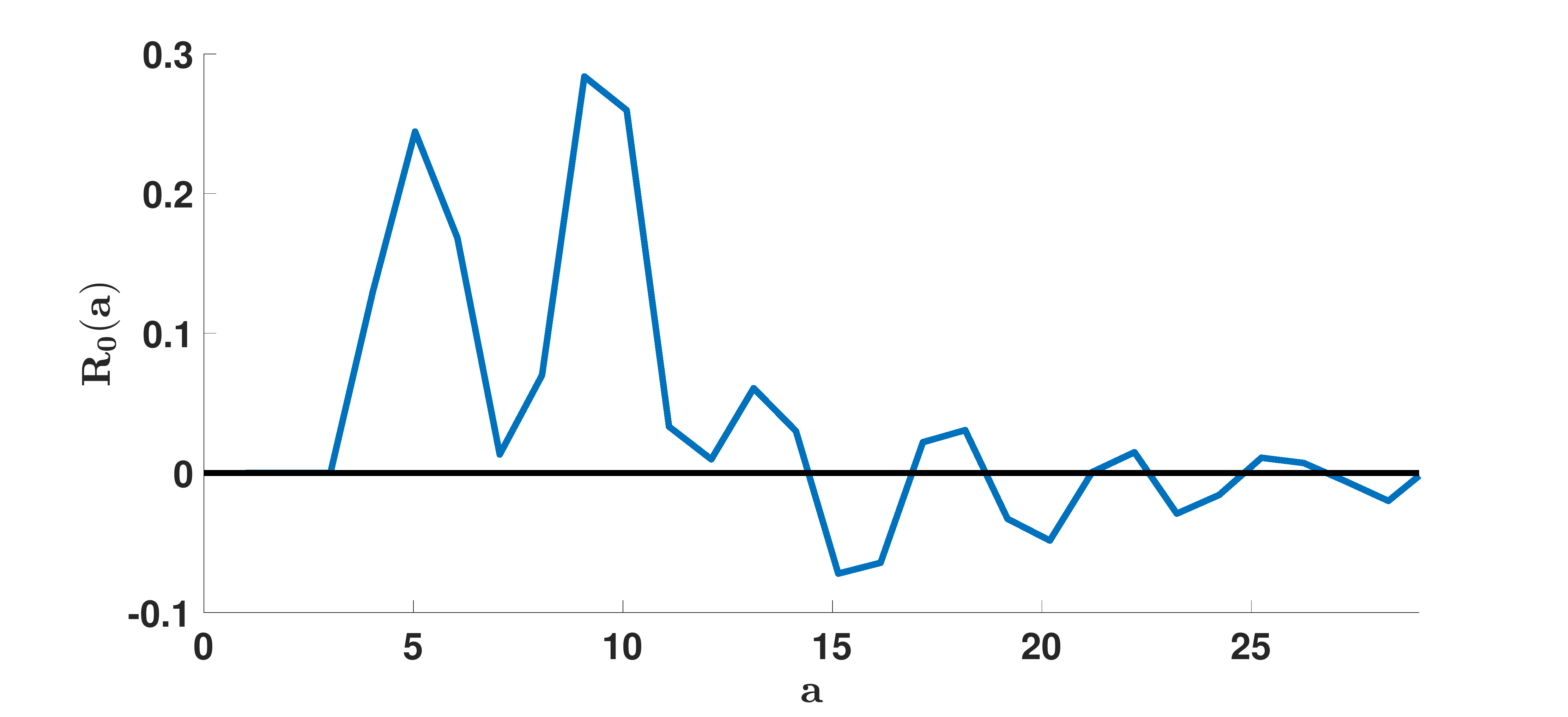}
	\end{center}
	\caption{\textit{On the left-hand side,  we plot the daily number of cases $t \to N(t)$ (for $t=0,1,2, \ldots$) obtained by summing the daily number of cases for $500$ runs of the IBM.  On the right-hand side, we apply formula \eqref{6.5}  (with $I_0=500 \times 10$) to the daily number of cases obtained from the IBM. }}\label{Fig34}
\end{figure}
\noindent \textbf{Simulations for $I_0=1000$:}
 \begin{figure}[H]
	\begin{center}
		\includegraphics[scale=0.15]{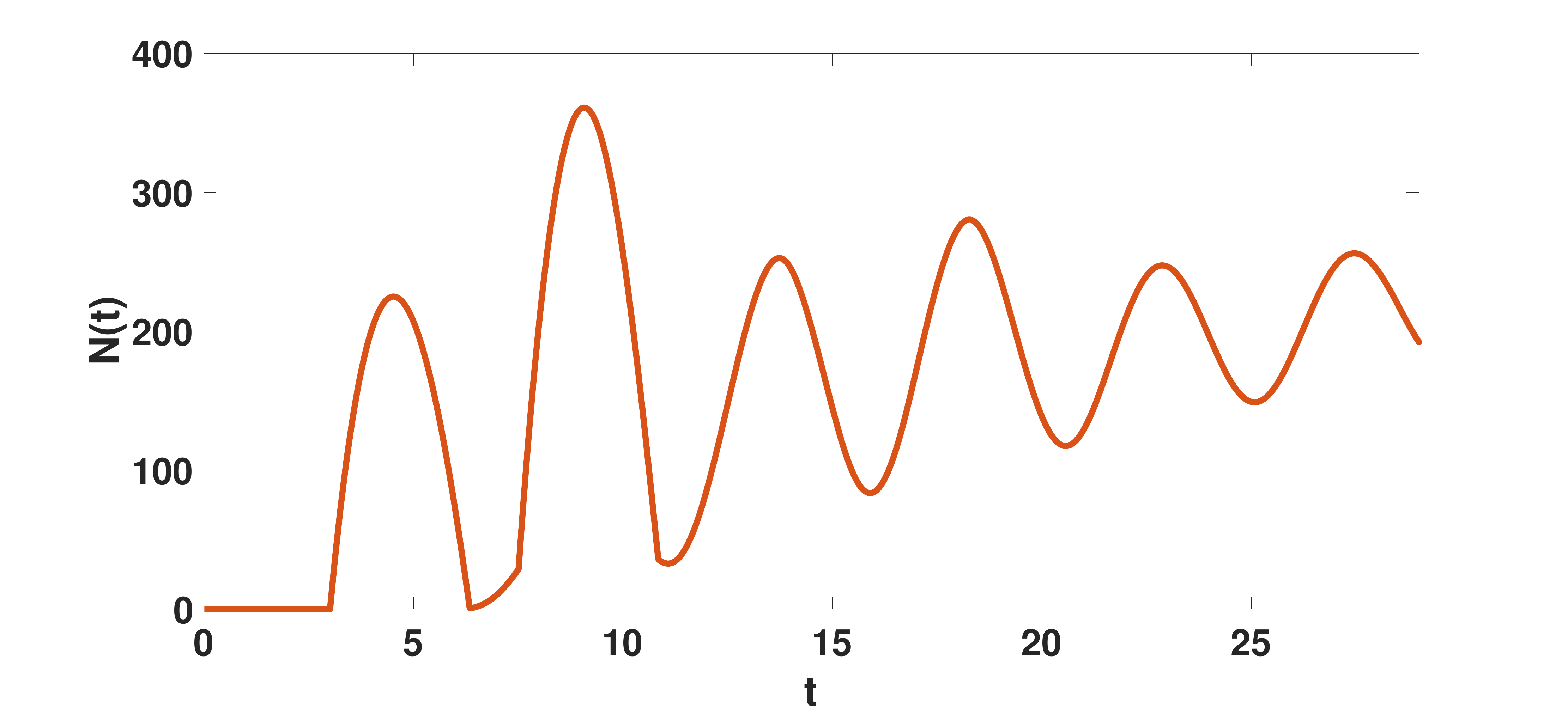}
		\includegraphics[scale=0.15]{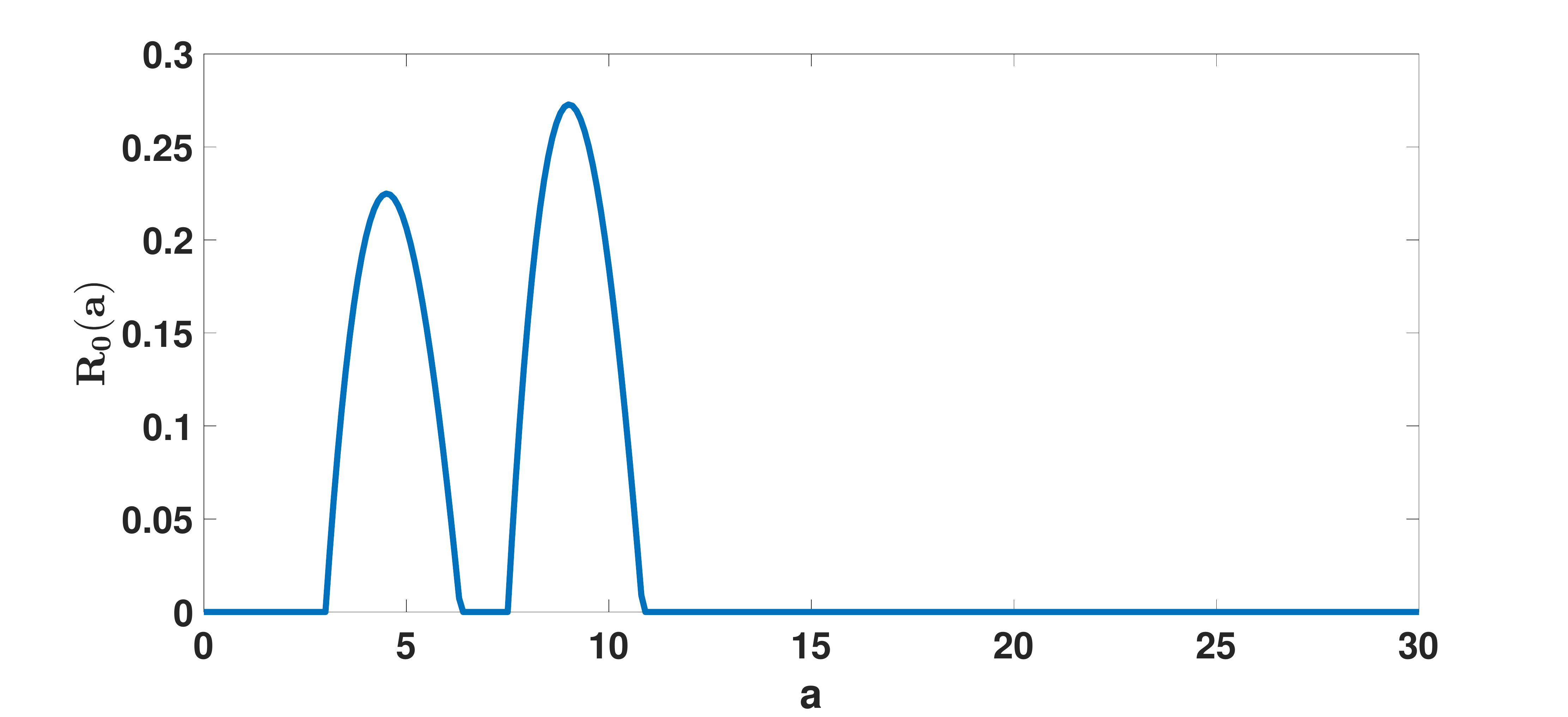}
	\end{center}
	\caption{\textit{On the left-hand side,  we plot the function $t \to N(t)$ solution of \eqref{4.3} with \eqref{2.4}. On the right-hand side, we plot the function $t \to \int_{t-1}^{t} N(s)ds$ (for $t=1,2, \ldots$) which corresponds to the daily number of cases obtained from by solving \eqref{4.3} with \eqref{2.4}, and we compare it with the daily number of cases obtained from $500$ runs of the IBM.}}\label{Fig35}
\end{figure}

\begin{figure}[H]
	\begin{center}
		\includegraphics[scale=0.15]{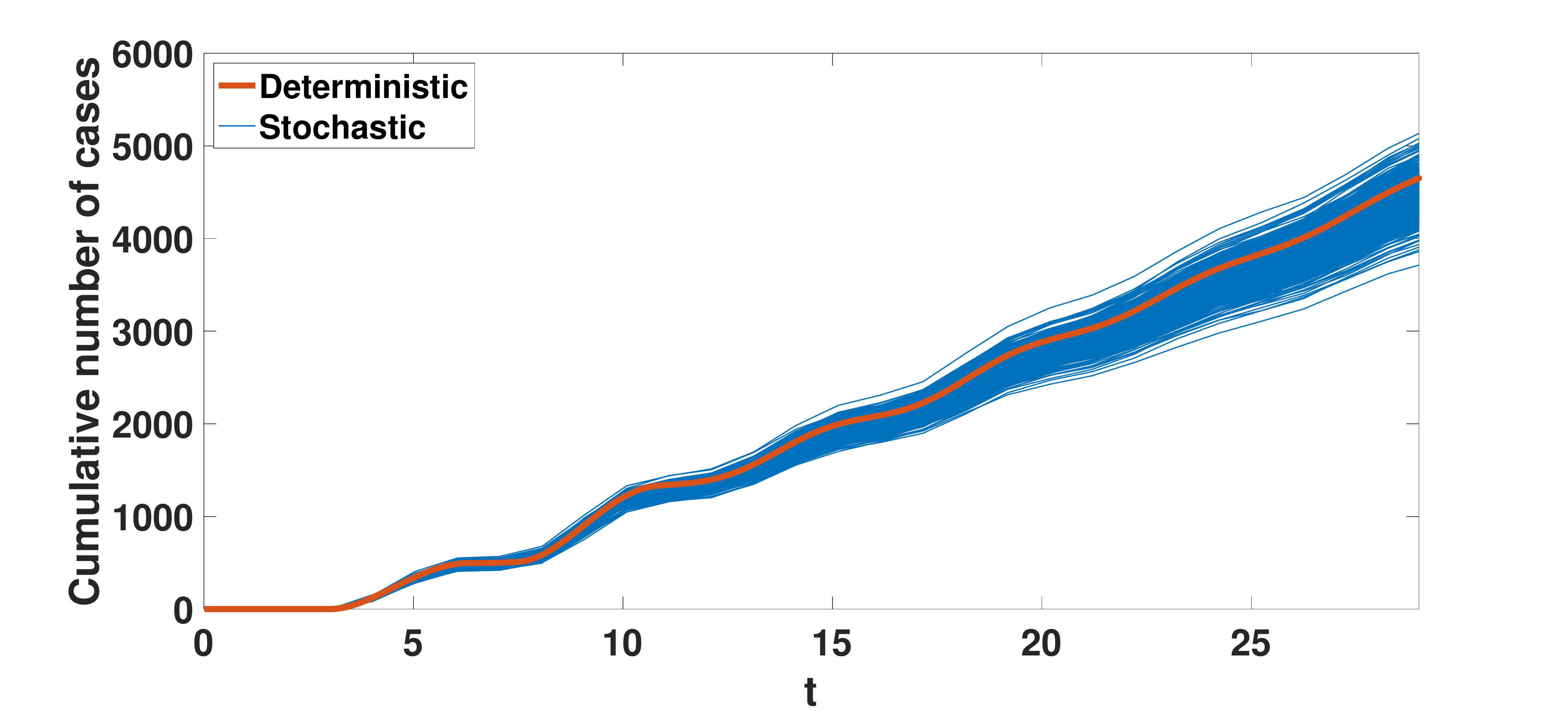}
		\includegraphics[scale=0.15]{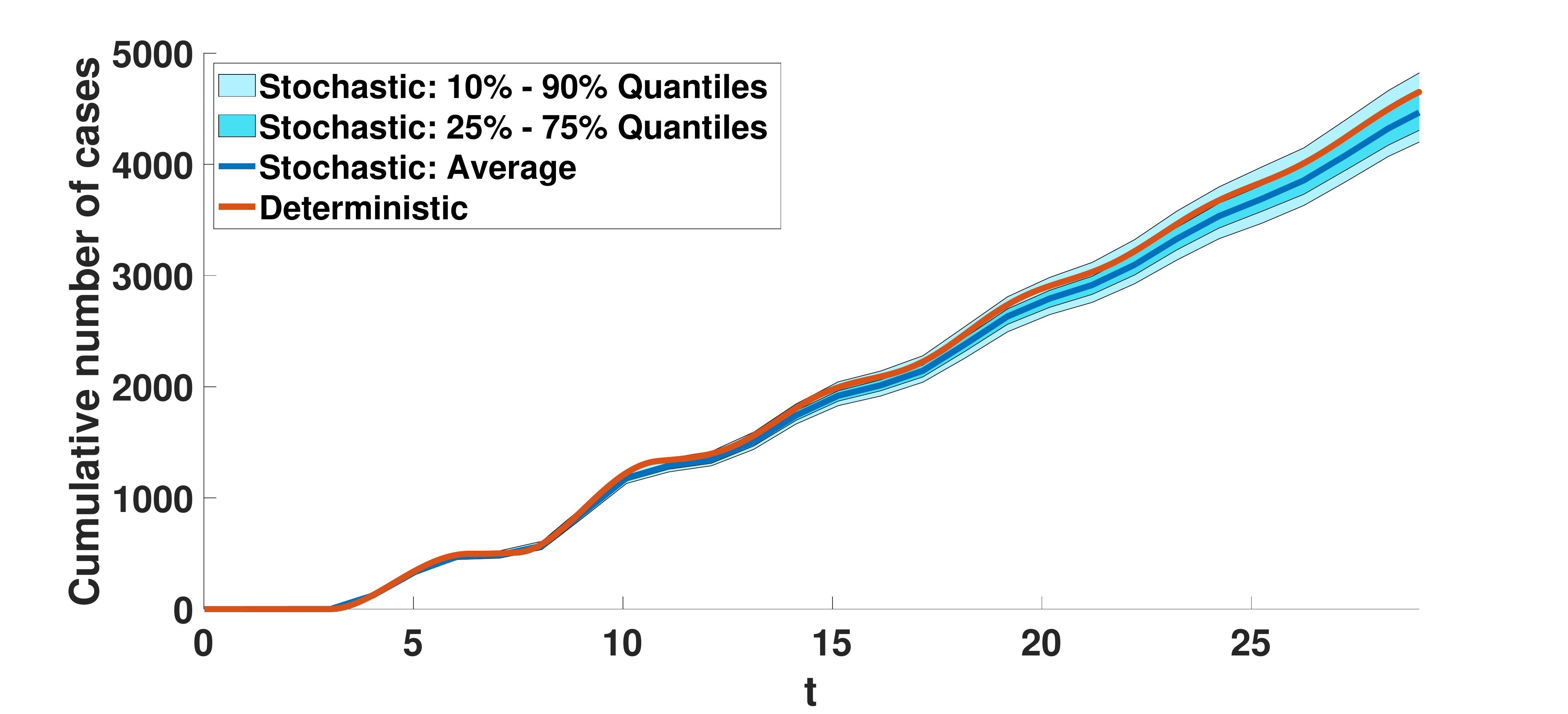}
	\end{center}
	\caption{\textit{On the left-hand side,  we plot the function $t \to \int_0^t N(s)ds$ (for $t=0,1,2, \ldots$) which corresponds to the cumulative number of cases obtained from by solving \eqref{4.3} with \eqref{2.4}, and we compare it with the cumulative number of cases obtained from $500$ runs of the IBM. On the right-hand side, we plot the average values of the $500$ runs obtained from the IBM as well as the quantiles ($10\%-90\%$ (light blue) and  $25\%-75\%$ (blue)).}}\label{Fig36}
\end{figure}

In Figures \ref{Fig37}-\ref{Fig39} we focus on the reconstruction of the daily reproduction number  $R_0(a)=\tau \, S_0 \,  e^{-  \nu \, a } \,  \beta\left(a\right)$. In Figure \ref{Fig37}, we focus on the reconstruction of the daily reproduction number  from deterministic simulations, while in Figures \ref{Fig38}-\ref{Fig39} we focus on the reconstruction of the daily reproduction number  from stochastic simulations. 
\begin{figure}[H]
	\begin{center}
		\includegraphics[scale=0.15]{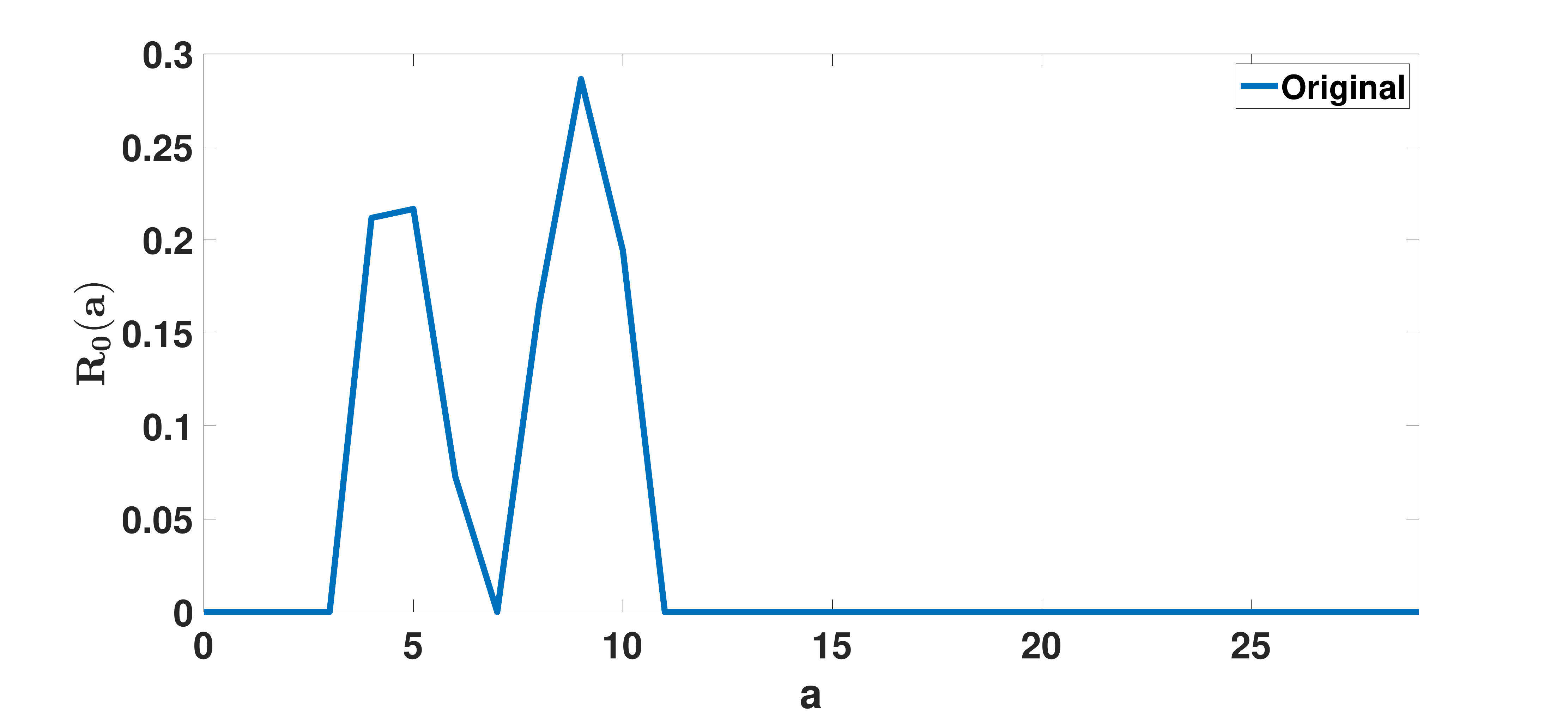}
		\includegraphics[scale=0.15]{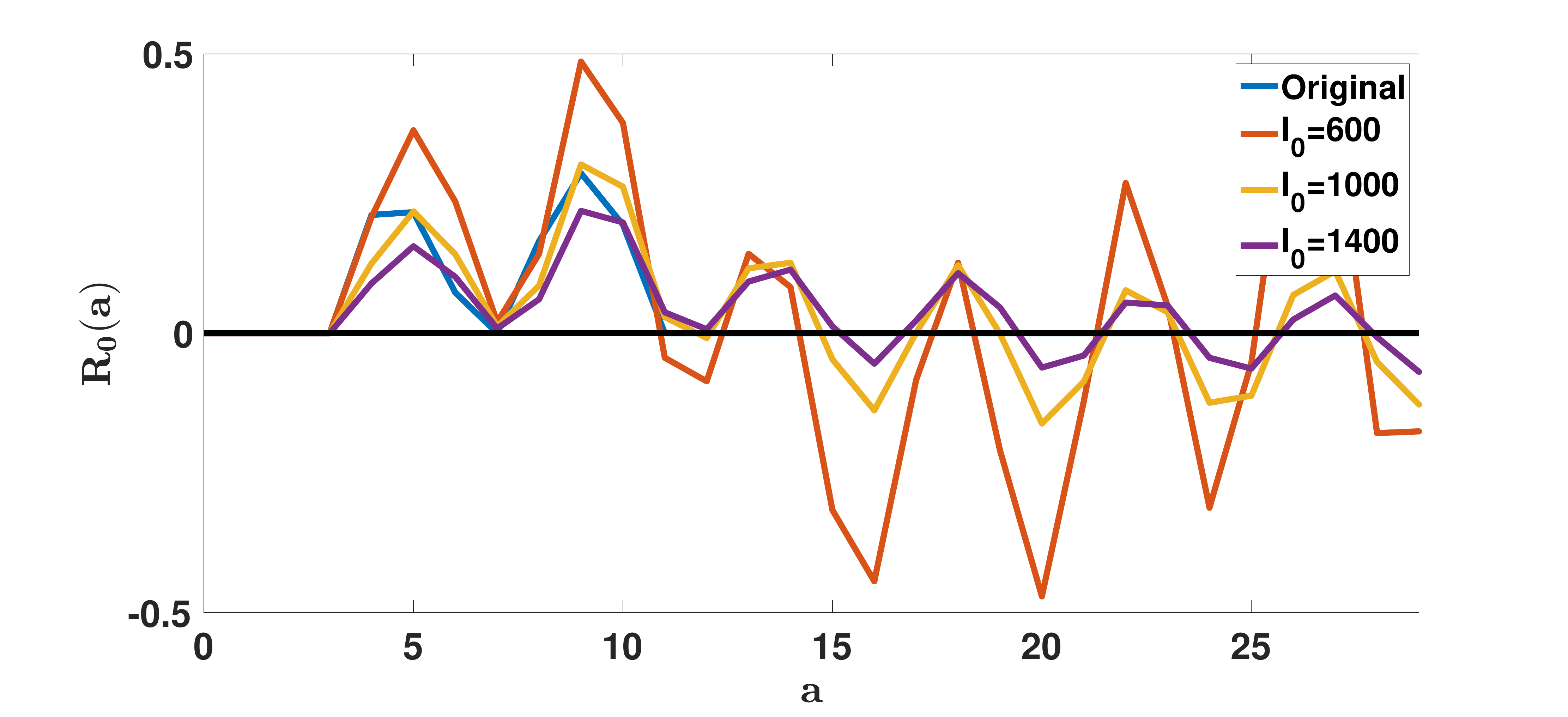}
	\end{center}
	\caption{\textit{On the left hand side, we plot the daily basic reproduction number by using the original formula $ R_0(a)$ with \eqref{7.3}.  On the right-hand side, we apply formula \eqref{6.5} to the flow of new infected obtained from the deterministic model. We vary $I_0=600, 1000, 1400$. The value $I_0=1000$ corresponds to the value used for the simulation of the deterministic model.   The yellow curve gives the best visual fit, and the $R_0(a)$ becomes negative whenever $I_0$ becomes too small. }}\label{Fig37}
\end{figure}

\begin{figure}[H]
	\begin{center}
		\includegraphics[scale=0.15]{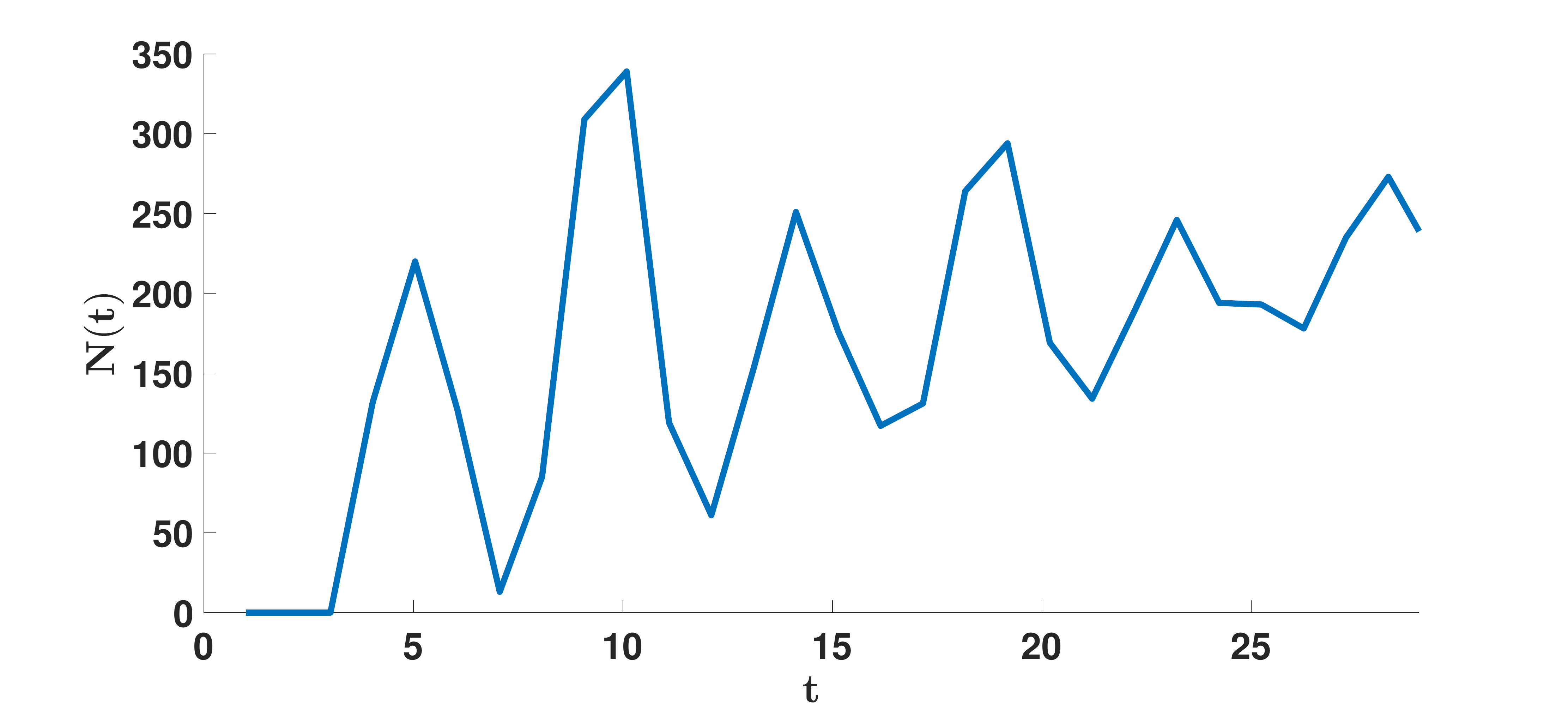}
		\includegraphics[scale=0.15]{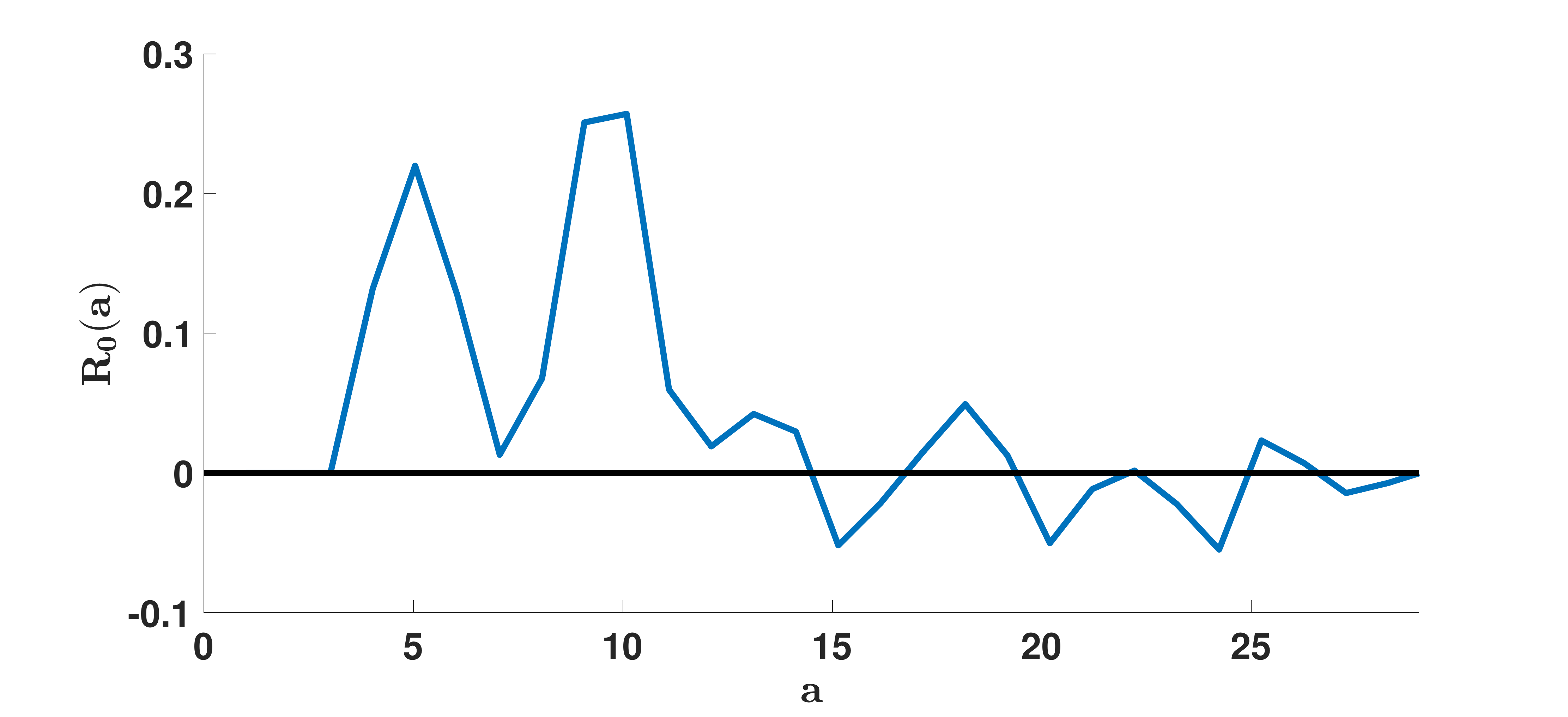}
	\end{center}
	\caption{\textit{On the left-hand side,  we plot the daily number of cases $t \to N(t) $ (for $t=0,1,2, \ldots$) obtained from a single run of the IBM.  On the right-hand side, we apply formula \eqref{6.5}  (with $I_0=1000$) to  the daily number of cases of case obtained from the IBM.   }}\label{Fig38}
\end{figure}

\begin{figure}[H]
	\begin{center}
		\includegraphics[scale=0.15]{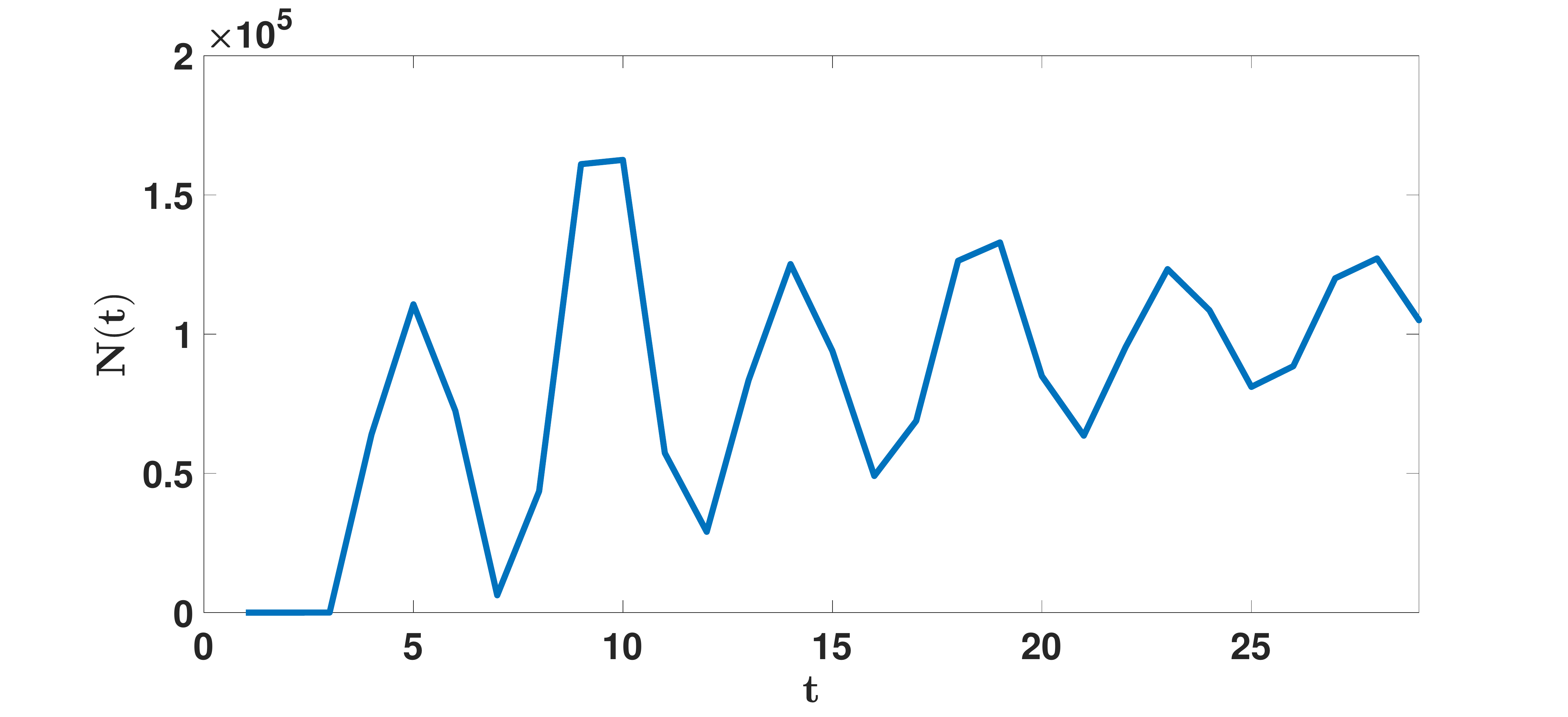}
		\includegraphics[scale=0.15]{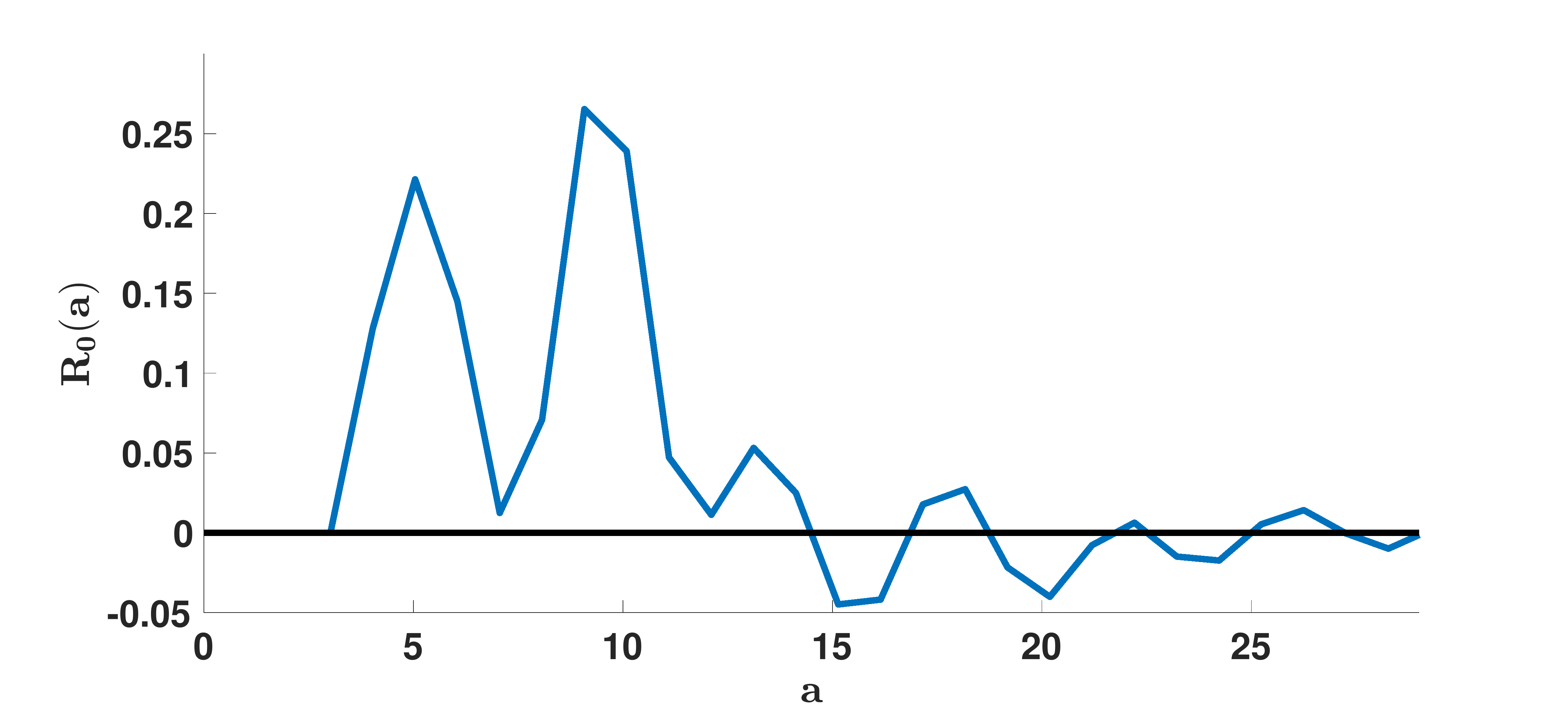}
	\end{center}
	\caption{\textit{On the left-hand side,  we plot the daily number of cases $t \to N(t)$ (for $t=0,1,2, \ldots$) obtained by summing the daily number of cases for $500$ runs of the IBM.  On the right-hand side, we apply formula \eqref{6.5}  (with $I_0=500 \times 1000$) to the daily number of cases obtained from the IBM. }}\label{Fig39}
\end{figure}

\section{Application to SARS-CoV-1}
\label{Section8}
In practice, the  Kermack-McKendrick model starting from a Dirac mass means that the epidemic starts from a single patient 	at time $t_0$ (whenever $I_0=1$) or from a group of $I_0$ infected patients all with the same age of infection $a=0$ at time $t_0$. This assumption corresponds to the standard conception of a cluster in epidemiology. An example of such a cluster  is obtained \cite{CDC1} for the SARS-CoV-1 epidemic in Singapore in 2003. The cluster is represented by a network of contact between individuals in Figure \ref{Fig40}. 
\begin{figure}[H]
	\begin{center}
		\includegraphics[scale=0.6]{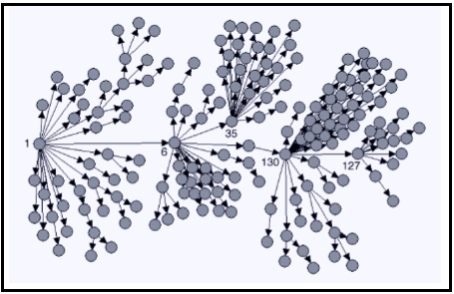}
		\includegraphics[scale=0.15]{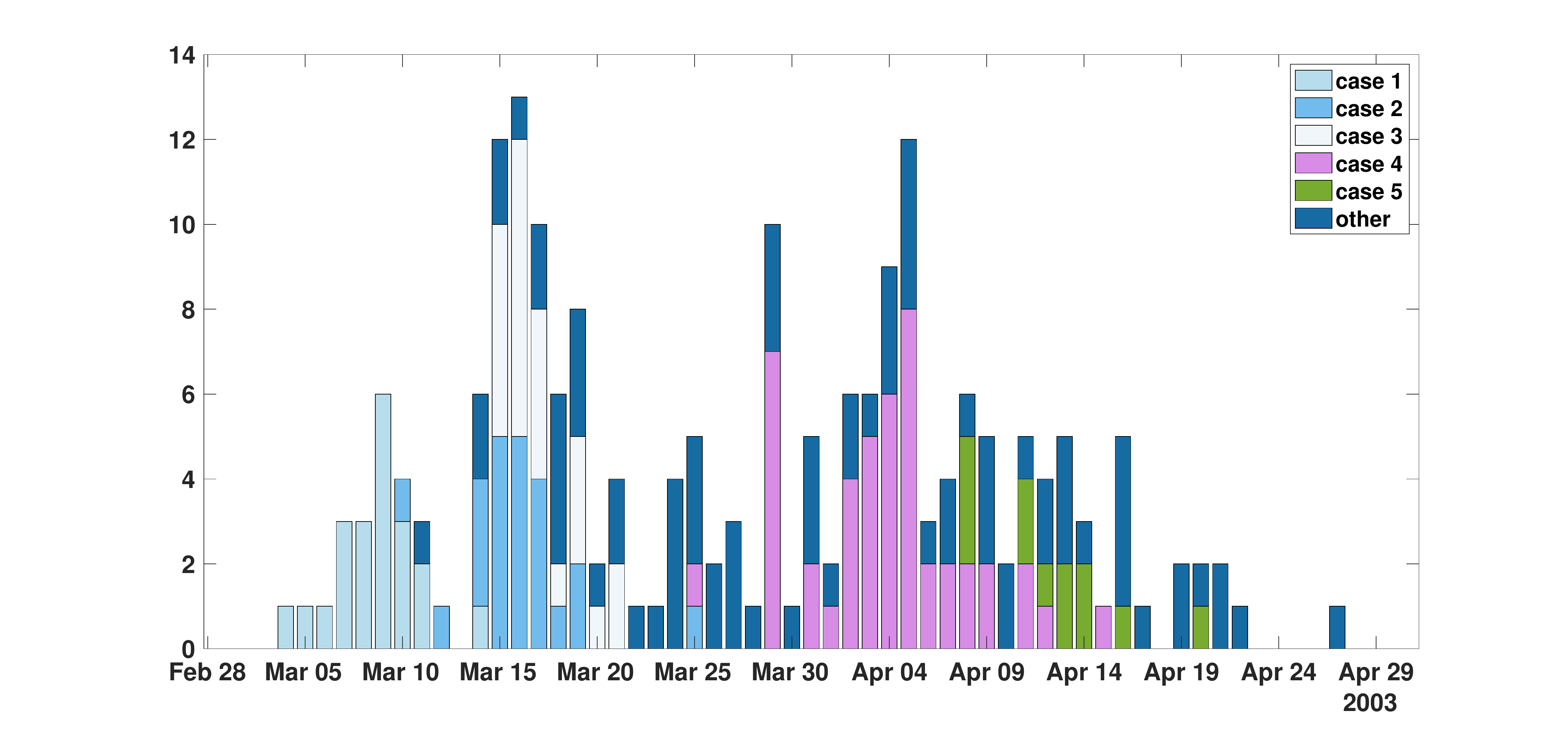}
	\end{center}
	\caption{\textit{On  the left-hand side, we plot the contact network of the five super spreader cases in the SARS epidemic in Singapore in 2003 \cite{CDC1}. The super spreaders are patient 1, patient 6, patient 35, patient 130 and patient 127.  On the right-hand side, we plot the daily reported cases from Singapore for the epidemic of SARS in 2003. 
	Case 1 generated 21 cases and 3 suspected cases, case 2 generated 23 cases and 5 suspected cases, case 3 generated 23 cases and 18 suspected cases, case 4 generated 40 cases and 22 suspected cases, case 5 generated 15 cases and 0 suspected cases \cite{CDC1}. The cases 1,2,3,4,5 correspond respectively to the patients 1, 6, 35, 130 and 127. }}\label{Fig40}
\end{figure}
Figure \ref{Fig40} presents the time series of reported cases by source of infection and date of fever onset. In Figure \ref{Fig41} we present three representations of these data in continuous time: as a step function, regularized by Gaussian average and rolling weekly average. In Figure \ref{Fig42} we apply the continuous-time model to the rolling weekly regularization of the data. Similar to the reconstruction of $R_0(a)$ presented in Figures 
 \ref{Fig19}-\ref{Fig21}, \ref{Fig24}-\ref{Fig26}, \ref{Fig32}-\ref{Fig34}, \ref{Fig37}-\ref{Fig39}, the basic reproduction number $R_0(a)$ becomes negative after a given age. Our interpretation is that the data are far from perfect and involves sampling errors and probably a large number of undetected cases. 
The fact that the transmission rate is subject to variations in time could also explain this negativity. In Figure \ref{Fig43} we apply the discrete model \eqref{1.6} to the original data for different values of $I_0$. Finally, in Figure \ref{Fig44}, we transform the data by taking advantage of the information on the source of infection given in \cite{CDC1}. We fix an incubation period of $5$ days, which corresponds to the average incubation period reported in \cite{CDC1}. Then we shift all secondary cases produced by the six sources identified in the article \cite{CDC1} to the same origin, as if all cases had been produced by the same cluster of six individuals. We present the data on the left-hand side of Figure \ref{Fig44} and apply the method to obtain $R_0(a)$ for parameters $I_0=30$, $I_0=50$, $I_0=100$.

\begin{figure}[H]
	\begin{center}
		\includegraphics[scale=0.15]{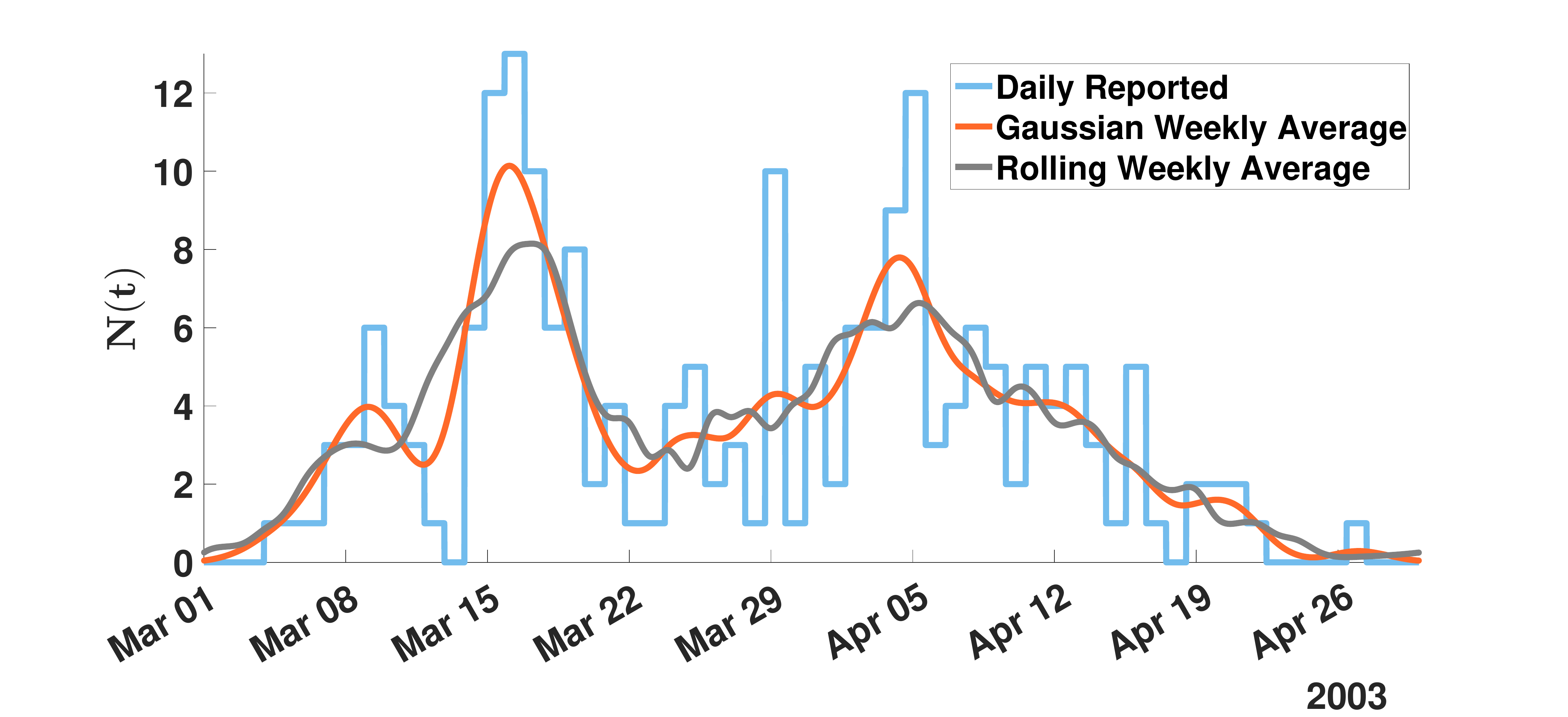}
	\end{center}
	\caption{\textit{Regularizations of the daily cases data from the SARS-CoV-1 outbreak in Singapore \cite{CDC1}. The applications in Figure \ref{Fig42} are done with the ``Rolling Weekly'' regularization. }}\label{Fig41}
\end{figure}
\begin{figure}[H]
	\begin{center}
		\includegraphics[scale=0.15]{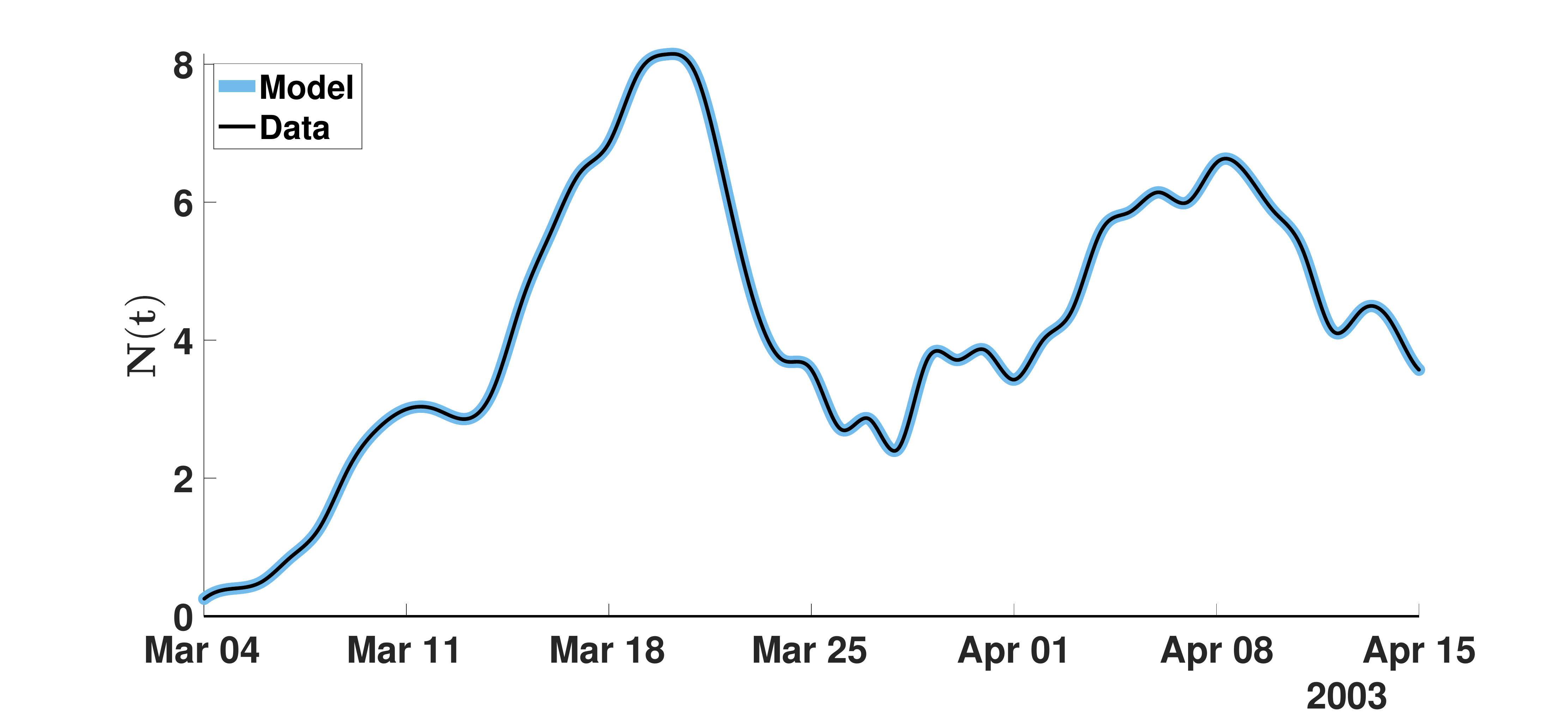}
		\includegraphics[scale=0.15]{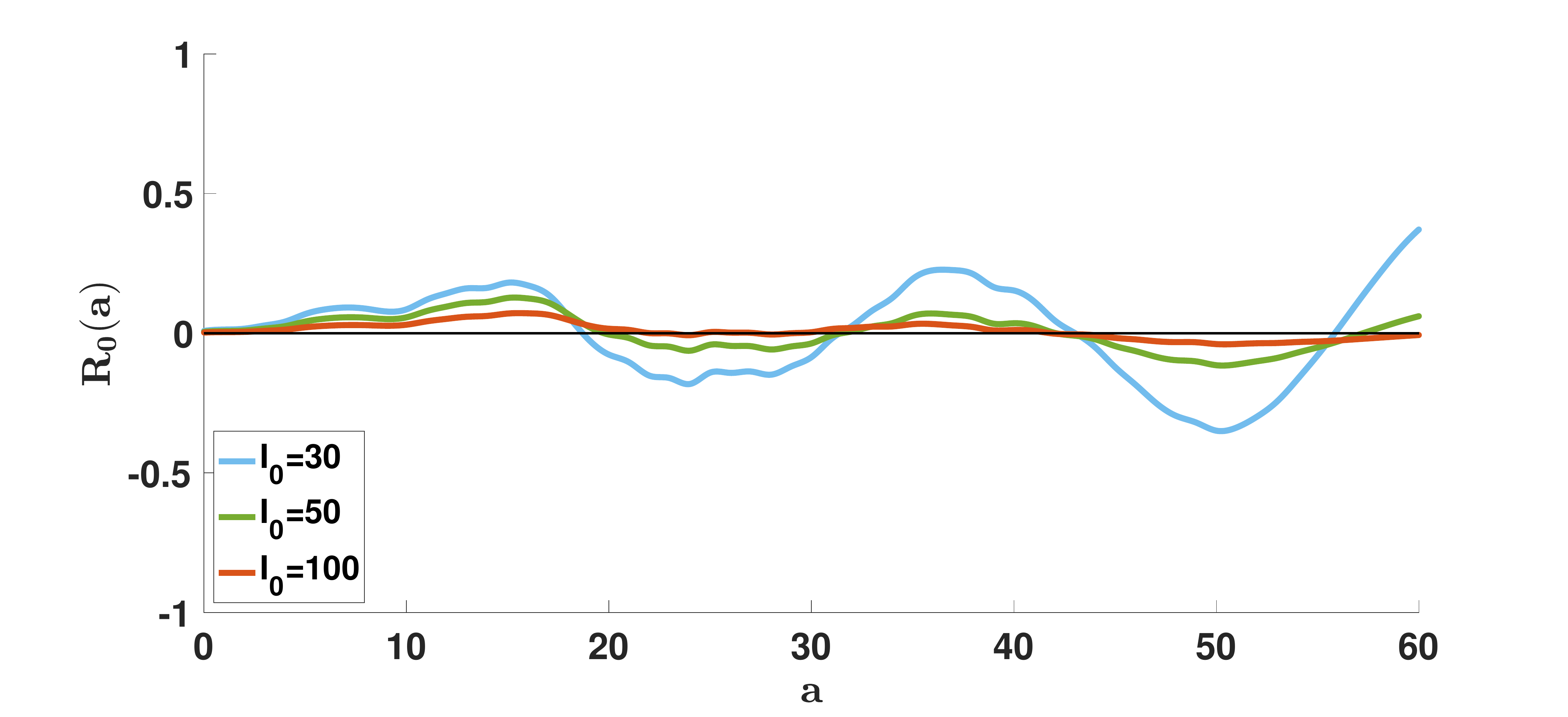}
	\end{center}
	\caption{\textit{Left: Regularized data of the SARS-CoV-1 outbreak in Singapore  in 2003 \cite{CDC1} (black line) and the numerical solution of the model \eqref{1.1} with $I_0=30$ and $R_0(a)$ computed by \eqref{1.6} (blue line). The solutions $N(t)$ of the model \eqref{1.1} with $I_0=50$ and $I_0=100$ are exactly the same when we use the corresponding $R_0(a)$, therefore they are not represented here. Right: numerical solution of the $R_0(a)$ function computed by using the continuous model \eqref{1.4} with $I_0=30$, $I_0=50$ and $I_0=100$. }}\label{Fig42}
\end{figure}

\begin{figure}[H]
	\begin{center}
		\includegraphics[scale=0.15]{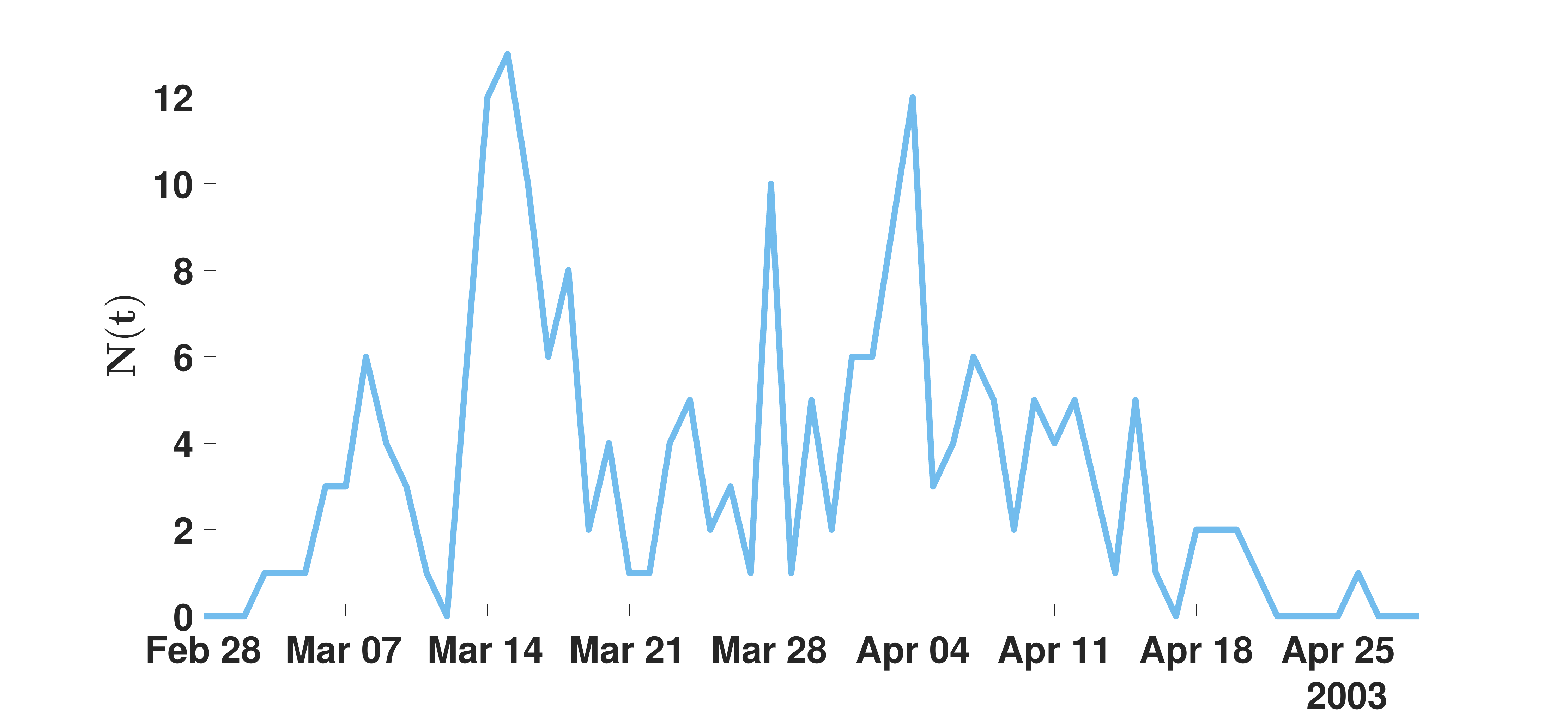}
		\includegraphics[scale=0.15]{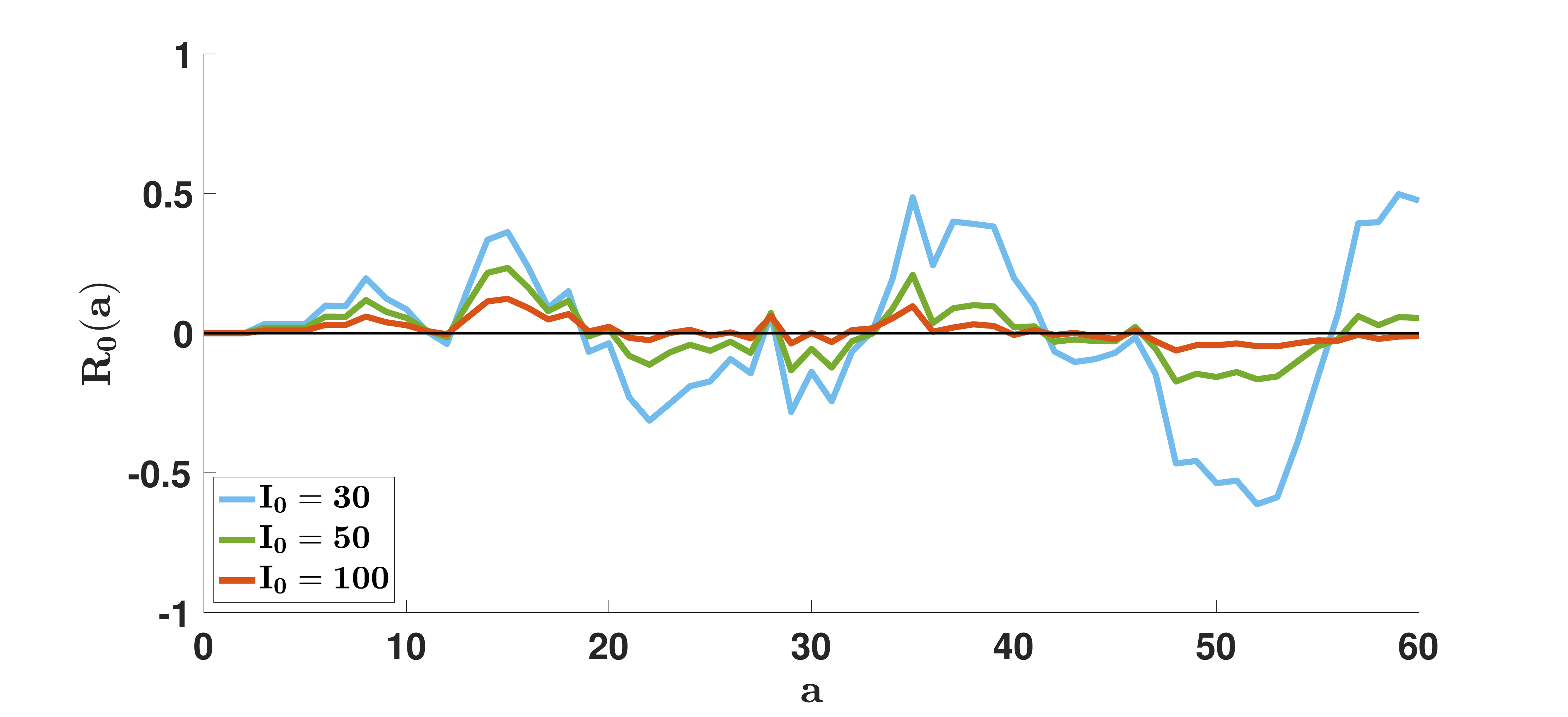}
	\end{center}
	\caption{\textit{Left: Original daily reported cases data of the SARS-CoV-1 outbreak in Singapore  in 2003 \cite{CDC1} (blue line). Right: Numerical solution of the $R_0(a)$ function computed by using the discrete model \eqref{1.6} with $I_0=30$, $I_0=50$ and $I_0=100$. }}\label{Fig43}
\end{figure}

\begin{figure}[H]
	\begin{center}
		\includegraphics[scale=0.15]{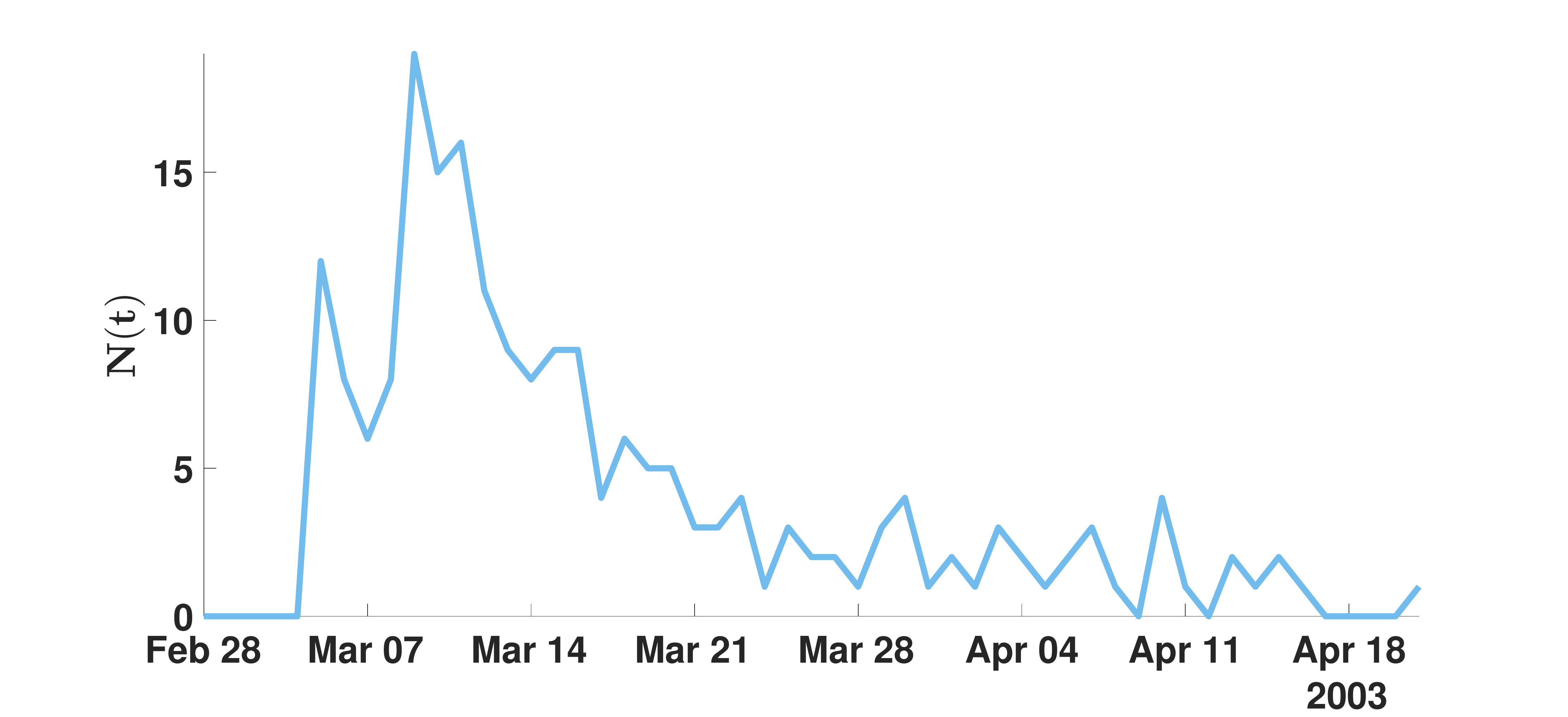}
		\includegraphics[scale=0.15]{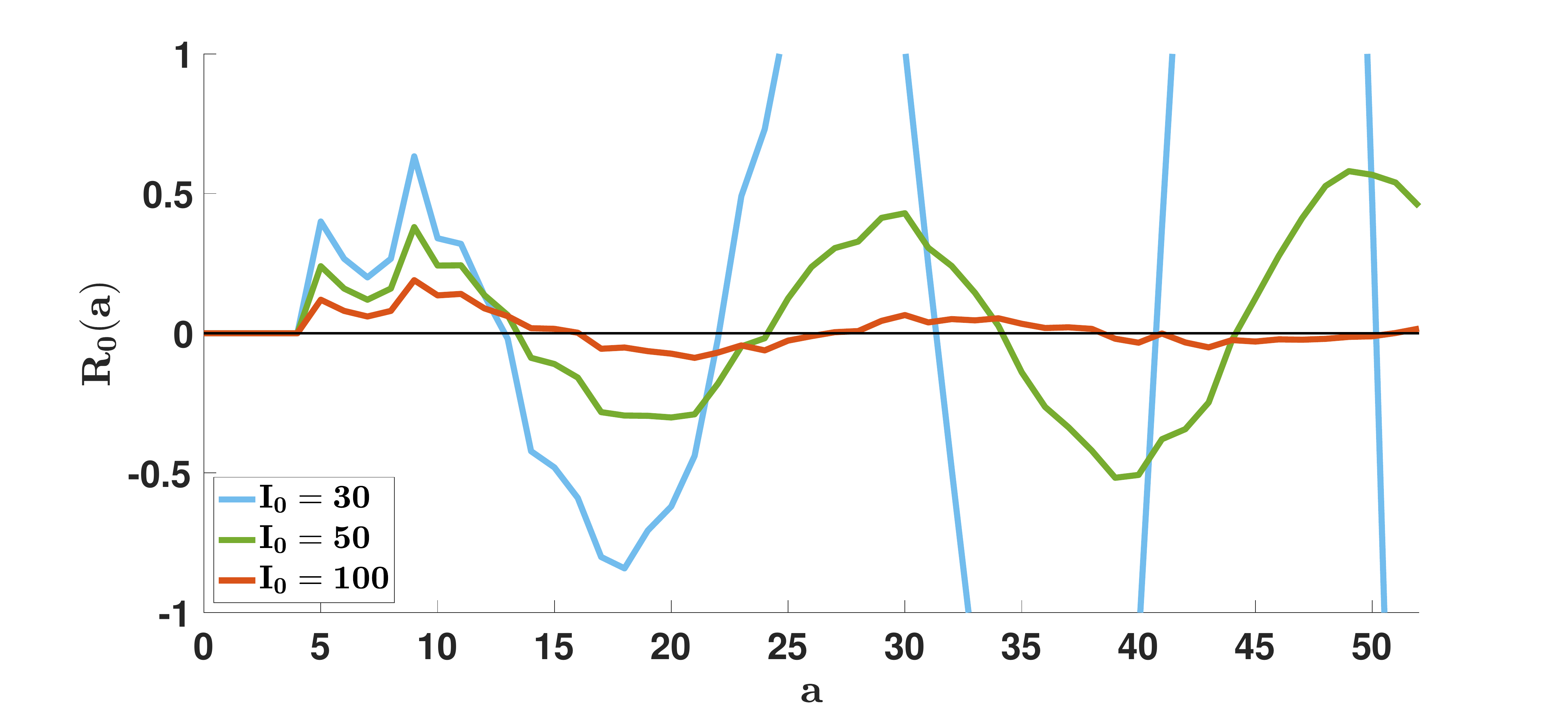}
	\end{center}
	\caption{\textit{Left: Daily reported cases of a theoretical cluster based on the data from the SARS-CoV-1 outbreak in Singapore  in 2003 \cite{CDC1} (blue line). The secondary cases produced by the six patients identified in  \cite{CDC1} are shifted to the same initial infection date of the infector, as if all were produced by the same cluster starting at $t=0$. We also use an incubation period of 5 days, as indicated in  \cite{CDC1}. Right: Numerical solution of the $R_0(a)$ function computed by using the discrete model \eqref{1.6} with $I_0=30$, $I_0=50$ and $I_0=100$.  }}\label{Fig44}
\end{figure}
%
%

\section{Discussion}
\label{Section9}

We see from the numerical simulations in Section 7 that the initial number of infected $I_0$ has a very significant influence on the value of the basic daily reproduction numbers $R_0(d)$: these decrease sharply with $I_0$, until they become negative and their fluctuations increase in stochastic simulations. This tendency to negativity for small $I_0$ and these fluctuations are only corrected when the results are averaged for a large number of stochastic simulations (500). It can also be noted that the stochastic simulations lead to a behavior of the hyper exponential type in the coefficient of variation of the secondary cases produced by an infectious individual, that is to say that it is relatively constant and much greater than 1. This phenomenon is to be related to the exponential character of the gamma distribution used in the simulations. 

In deterministic simulations, one can observe a great dependence of the results with regard to the value of $I_0$ concerning the reconstruction of the daily basic reproduction numbers $R_0(a)$ along the contagiousness period of an individual. In particular, the set of days at which $R_0(a)$ fluctuates becoming in some cases negative is important for the great values of $I_0$. In stochastic simulations, we observe the same behavior for the different curves related to the $R_0(a)$ curves sample, but the expectation of this curves sample considerably attenuates these fluctuations and the coefficient of variation of the curves remains approximately constant, while being greater than 1, as in the case of hyperexponential distributions, in agreement with the exponential character of the part D of the equation \eqref{1.7} defining $R_0(a)$.
 
Concerning the clusters, from observations made during investigations of the start of the outbreak in some countries \cite{Yong,Adam,Desjardins, Pung, Ganyani, Jing, Boehmer, Guttmann, Shams, Chan, Hisada, Han, Ladoy, Tariq}, it is possible to get spatial and temporal information on the start of the epidemic, but  these studies rarely allow the estimation of the parameters $S_0$ and $\tau$ in the concerned population and worse, they give no indication of how long they remain constant. Here we assumed that they remained constant only during the period of exponential growth of new cases observed. 

Our work provides a method to reconstruct the daily basic reproduction number $R_0(a)$ from the daily reported cases data, as long as we consider a cluster starting from a single infected. This is a strong assumption which is usually neglected. It is extremely hard to find in the literature a datasets which satisfies this assumption.  For COVID-19, we did not find any publication including suitable data. While not published yet, we believe that this kind of data could be gathered by a detailed contact-tracing and -- duly anonymized -- could be made available by request. That would allow the development of more realistic and accurate methods for the analysis and forecast of epidemics.

\medskip

\appendix
\addcontentsline{toc}{section}{Appendix} 
\begin{center}
{\LARGE	\textbf{Appendix}}
\end{center}
\section{Euler approximation for the Volterra integral equation}
\label{AppendixA}
We use a numerical scheme for the equation \eqref{4.3} and  \eqref{4.4}.

\begin{equation}  \label{A.1}
	N(t) =\tau  \, S\left(\max \left(t-\Delta t,t_0 \right)\right) \left[ I_0 \times	\Gamma\left(t-t_0\right) +  \int_0^{\max \left(t-\Delta t-t_0,0 \right)} 	\Gamma(a) \, N(t-a) da \right],
\end{equation}
with 
\begin{equation}  \label{A.2}
	S(t)=S_0 - \int_{t_0}^t N( \sigma) d \sigma, 
\end{equation}
and we use the Simpson's rule to compute the integrals. 
\section{Stochastic simulations: Individual Based Model}
\label{AppendixB}
In order to estimate the uncertainty expected in real datasets, we use stochastic simulations that reproduce the first stages of the epidemic in finite populations. We consider a population composed of a finite number $N=S_0+I_0$ of individuals. We start the simulation a time $t=0$ with $S_0 \in \N$ susceptible individuals and $I_0 \in \N $ infected individuals all with age of infection $a=0$.  For each infected individuals we also compute the time spent in the $I$-compartment which follows an exponential law with parameters $1/\nu$.  The principles of the simulations are as follows:
\begin{enumerate}
	\item Individuals meet at random at rate $\tau>0$. In other words, each pair of individual in the population has a contact which occurs at a time following an exponential law of average $1/\tau$. 
	\item When a contact occurs between an infected individual of age $a$ and a susceptible individual, the contact results in a newly infected individual of age $0$ with probability $\beta(a)$. When the infection occurs, the newly infected individual is assigned a duration of infection which follows an exponential law of rate $\nu$. Therefore individuals stay infected on average for a duration of $1/\nu$.
	\item The age of all individuals is updated at fixed intervals of time of size $\Delta t$. Simultaneously the life-span of each infected inviduals is decreased by $\Delta t$ and individual whose life-span has become negative are removed from the system.
\end{enumerate}
The MATLAB code of the IBM is available online at:

\noindent  \url{https://github.com/romainvieme/2022-kermack-mckendrick-single-cohort}.

\end{document}